\documentclass[leqno,11pt]{amsart}
\usepackage{etoolbox}
\makeatletter
\let\old@tocline\@tocline
\let\section@tocline\@tocline

\newcommand{\subsection@dotsep}{4.5}
\newcommand{\subsubsection@dotsep}{4.5}
\patchcmd{\@tocline}
  {\hfil}
  {\nobreak
     \leaders\hbox{$\m@th
        \mkern \subsection@dotsep mu\hbox{.}\mkern \subsection@dotsep mu$}\hfill
     \nobreak}{}{}
\let\subsection@tocline\@tocline
\let\@tocline\old@tocline

\patchcmd{\@tocline}
  {\hfil}
  {\nobreak
     \leaders\hbox{$\m@th
        \mkern \subsubsection@dotsep mu\hbox{.}\mkern \subsubsection@dotsep mu$}\hfill
     \nobreak}{}{}
\let\subsubsection@tocline\@tocline
\let\@tocline\old@tocline

\let\old@l@subsection\l@subsection
\let\old@l@subsubsection\l@subsubsection

\def\@tocwriteb#1#2#3{%
  \begingroup
    \@xp\def\csname #2@tocline\endcsname##1##2##3##4##5##6{%
      \ifnum##1>\c@tocdepth
      \else \sbox\z@{##5\let\indentlabel\@tochangmeasure##6}\fi}%
    \csname l@#2\endcsname{#1{\csname#2name\endcsname}{\@secnumber}{}}%
  \endgroup
  \addcontentsline{toc}{#2}%
    {\protect#1{\csname#2name\endcsname}{\@secnumber}{#3}}}%

\newlength{\@tocsectionindent}
\newlength{\@tocsubsectionindent}
\newlength{\@tocsubsubsectionindent}
\newlength{\@tocsectionnumwidth}
\newlength{\@tocsubsectionnumwidth}
\newlength{\@tocsubsubsectionnumwidth}
\newcommand{\settocsectionnumwidth}[1]{\setlength{\@tocsectionnumwidth}{#1}}
\newcommand{\settocsubsectionnumwidth}[1]{\setlength{\@tocsubsectionnumwidth}{#1}}
\newcommand{\settocsubsubsectionnumwidth}[1]{\setlength{\@tocsubsubsectionnumwidth}{#1}}
\newcommand{\settocsectionindent}[1]{\setlength{\@tocsectionindent}{#1}}
\newcommand{\settocsubsectionindent}[1]{\setlength{\@tocsubsectionindent}{#1}}
\newcommand{\settocsubsubsectionindent}[1]{\setlength{\@tocsubsubsectionindent}{#1}}

\renewcommand{\l@section}{\section@tocline{1}{\@tocsectionvskip}{\@tocsectionindent}{\@tocsectionnumwidth}{\@tocsectionformat}}%
\renewcommand{\l@subsection}{\subsection@tocline{1}{\@tocsubsectionvskip}{\@tocsubsectionindent}{\@tocsubsectionnumwidth}{\@tocsubsectionformat}}%
\renewcommand{\l@subsubsection}{\subsubsection@tocline{1}{\@tocsubsubsectionvskip}{\@tocsubsubsectionindent}{\@tocsubsubsectionnumwidth}{\@tocsubsubsectionformat}}%
\newcommand{\@tocsectionformat}{}
\newcommand{\@tocsubsectionformat}{}
\newcommand{\@tocsubsubsectionformat}{}
\expandafter\def\csname toc@1format\endcsname{\@tocsectionformat}
\expandafter\def\csname toc@2format\endcsname{\@tocsubsectionformat}
\expandafter\def\csname toc@3format\endcsname{\@tocsubsubsectionformat}
\newcommand{\settocsectionformat}[1]{\renewcommand{\@tocsectionformat}{#1}}
\newcommand{\settocsubsectionformat}[1]{\renewcommand{\@tocsubsectionformat}{#1}}
\newcommand{\settocsubsubsectionformat}[1]{\renewcommand{\@tocsubsubsectionformat}{#1}}
\newlength{\@tocsectionvskip}
\newcommand{\settocsectionvskip}[1]{\setlength{\@tocsectionvskip}{#1}}
\newlength{\@tocsubsectionvskip}
\newcommand{\settocsubsectionvskip}[1]{\setlength{\@tocsubsectionvskip}{#1}}
\newlength{\@tocsubsubsectionvskip}
\newcommand{\settocsubsubsectionvskip}[1]{\setlength{\@tocsubsubsectionvskip}{#1}}

\patchcmd{\tocsection}{\indentlabel}{\makebox[\@tocsectionnumwidth][l]}{}{}
\patchcmd{\tocsubsection}{\indentlabel}{\makebox[\@tocsubsectionnumwidth][l]}{}{}
\patchcmd{\tocsubsubsection}{\indentlabel}{\makebox[\@tocsubsubsectionnumwidth][l]}{}{}

\newcommand{\@sectypepnumformat}{}
\renewcommand{\contentsline}[1]{%
  \expandafter\let\expandafter\@sectypepnumformat\csname @toc#1pnumformat\endcsname%
  \csname l@#1\endcsname}
\newcommand{\@tocsectionpnumformat}{}
\newcommand{\@tocsubsectionpnumformat}{}
\newcommand{\@tocsubsubsectionpnumformat}{}
\newcommand{\setsectionpnumformat}[1]{\renewcommand{\@tocsectionpnumformat}{#1}}
\newcommand{\setsubsectionpnumformat}[1]{\renewcommand{\@tocsubsectionpnumformat}{#1}}
\newcommand{\setsubsubsectionpnumformat}[1]{\renewcommand{\@tocsubsubsectionpnumformat}{#1}}
\renewcommand{\@tocpagenum}[1]{%
  \hfill {\mdseries\@sectypepnumformat #1}}

\let\oldappendix\appendix
\renewcommand{\appendix}{%
  \leavevmode\oldappendix%
  \addtocontents{toc}{%
    \protect\settowidth{\protect\@tocsectionnumwidth}{\protect\@tocsectionformat\sectionname\space}%
    \protect\addtolength{\protect\@tocsectionnumwidth}{2em}}%
}
\makeatother

\makeatletter
\settocsectionnumwidth{2em}
\settocsubsectionnumwidth{2.5em}
\settocsubsubsectionnumwidth{3em}
\settocsectionindent{1pc}%
\settocsubsectionindent{\dimexpr\@tocsectionindent+\@tocsectionnumwidth}%
\settocsubsubsectionindent{\dimexpr\@tocsubsectionindent+\@tocsubsectionnumwidth}%
\makeatother

\settocsectionvskip{2pt}
\settocsubsectionvskip{0pt}
\settocsubsubsectionvskip{0pt}

\settocsectionformat{\bfseries}
\settocsubsectionformat{\mdseries}
\settocsubsubsectionformat{\mdseries}
\setsectionpnumformat{\bfseries}
\setsubsectionpnumformat{\mdseries}
\setsubsubsectionpnumformat{\mdseries}

\let\oldtableofcontents\tableofcontents
\renewcommand{\tableofcontents}{%
  \vspace*{-5\linespacing}% Default gap to top of CONTENTS is \linespacing.
  \oldtableofcontents}

\usepackage{amsmath}
\usepackage{amssymb}
\usepackage{amsthm}
\usepackage{mathtools}
\usepackage{thmtools}
\usepackage{hhline}
\usepackage{tikz-cd}
\usepackage{enumitem}
\setlist[enumerate]{label={\normalfont (\roman*)}, itemsep=1ex}

\usepackage[utf8]{inputenc}
\usepackage[T1]{fontenc}
\usepackage{lmodern}
\usepackage[babel]{microtype}
\usepackage[english]{babel}
\usepackage{relsize}
\usepackage{calc}
\usepackage{multicol}

\usepackage{svg}

\usepackage{graphicx}
\usepackage{subcaption}

\counterwithin{figure}{section}

\usepackage{geometry}
\usepackage{xcolor} 
\colorlet{darkishRed}{red!60!black}
\colorlet{darkishBlue}{blue!60!black}
\colorlet{darkishGreen}{green!60!black}
\colorlet{darkblue}{blue!70!black}
\colorlet{darkishViolet}{violet}
\usepackage[linktoc=all]{hyperref}
\hypersetup{
	colorlinks,
    linkcolor={darkishBlue},
	citecolor={darkishGreen},
	urlcolor={darkishBlue}
}
\usepackage[nameinlink, capitalise, noabbrev]{cleveref}
\crefformat{enumi}{#2#1#3}
\crefformat{equation}{#2(#1)#3}
\crefname{equation}{}{}
\crefname{mainresult}{Theorem}{Theorems}
\let\setminus=\smallsetminus

\newcommand{\se}{\subseteq}
\newcommand{\sm}{\setminus}
\renewcommand{\supset}{\supseteq}
\renewcommand{\leq}{\leqslant}
\renewcommand{\geq}{\geqslant}
\renewcommand{\ge}{\geq}
\renewcommand{\le}{\leq}

\newtheorem{theorem}{Theorem}[section] 

\newtheorem{corollary}[theorem]{Corollary}
\newtheorem{lemma}[theorem]{Lemma}
\newtheorem{keylemma}[theorem]{Key Lemma}

\crefname{corlemma}{Correspondence}{Correspondences}

\crefname{liftlemma}{Lift}{Lifts}

\crefname{projlemma}{Projection}{Projections}
\newtheorem{setting}[theorem]{Setting}

\newtheorem{observation}[theorem]{Observation}
\newtheorem{conjecture}[theorem]{Conjecture}
\newtheorem{problem}[theorem]{Problem}

\newtheorem{mainresult}{Theorem} 

\newtheorem{maincorollary}[mainresult]{Corollary}

\theoremstyle{definition}
\newtheorem{example}[theorem]{Example}

\newtheorem{definition}[theorem]{Definition}

\newtheorem{challenge}{Challenge}
\theoremstyle{remark}

\newtheorem*{acknowledgment}{Acknowledgement}
\newtheorem*{claim*}{Claim}
\newtheorem{remark}[theorem]{Remark}

\crefname{claim}{Claim}{Claims}
\crefname{enumi}{}{}
\AtEndEnvironment{proof}{\setcounter{claim}{0}}

\usepackage{etoolbox}

\newcommand{\COMMENT}[1]{{}}
\definecolor{cMaroon}{HTML}{93152a}
\newcommand{\defnMath}[1]{{\color{darkishRed}{#1}}}
\newcommand{\defn}[1]{{\color{darkishRed}{\emph{#1}}}}
\newcommand{\casen}[1]{{\color{darkishViolet}{\emph{#1}}}}

\let\rho=\varrho
\let\phi=\varphi
\def\N{\mathbb N}

\def\Z{\mathbb Z}

\makeatletter

\def\calCommandfactory#1{%
  \expandafter\def\csname c#1\endcsname{\mathcal{#1}}}
\def\frakCommandfactory#1{%
  \expandafter\def\csname frak#1\endcsname{\mathfrak{#1}}}
   
\newcounter{ctr}
\loop
  \stepcounter{ctr}
  \edef\X{\@Alph\c@ctr}
  \expandafter\calCommandfactory\X
  \expandafter\frakCommandfactory\X
\ifnum\thectr<26
\repeat

\makeatletter
\patchcmd{\@tocline}
  {\hfil}
  {\leaders\hbox{\,.\,}\hfil}
  {}{}
\makeatother

\newcounter{mylabelcounter}

\makeatletter
\newcommand{\labelText}[2]{%
#1\refstepcounter{mylabelcounter}%
\immediate\write\@auxout{%
  \string\newlabel{#2}{{1}{\thepage}{{\unexpanded{#1}}}{mylabelcounter.\number\value{mylabelcounter}}{}}%
}%
}

\newlist{defenum}{enumerate}{1}
\setlist[defenum]{label=(\upshape\thedefinition.\arabic*)}

\newcommand{\allref}[1]{\nameref{#1}~(\ref{#1})}
\newcommand{\pl}{${}^+$}

\newcommand{\trisp}{strict tri-separation}

\usepackage[pagewise]{lineno}

\usepackage{csquotes}
\usepackage[backend=biber, style=alphabetic, maxnames=99, maxalphanames=99, giveninits=true, sortcites=true]{biblatex}
\DeclareLabelalphaTemplate{
    \labelelement{
        \field[final]{shorthand}
        \field{label}
        \field[strwidth=1,strside=left]{labelname}
    }
    \labelelement{
        \field[strwidth=2,strside=right]{year}
    }
}

\title[A Tutte-type canonical decomposition of 3- and 4-connected graphs]{A Tutte-type canonical decomposition\\of 3- and 4-connected graphs}

\author[Jan Kurkofka]{Jan Kurkofka${}^{\includegraphics[height=.7\baselineskip]{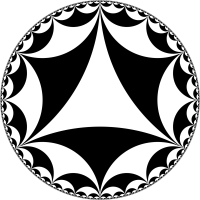}}$}
\author[Tim Planken]{Tim Planken$^\dagger$}

\thanks{$\includegraphics[height=.7\baselineskip]{Figures/MiniFarey.png}$ funded by the Deutsche Forschungsgemeinschaft (DFG, German Research Foundation) -- 566118291; 546892829}

\thanks{${}^\dagger$ funded by the Deutsche Forschungsgemeinschaft (DFG, German Research Foundation) -- 546892829}

\keywords{4-connected, decomposition, Tutte, canonical, double wheel, quasi-5-connected, 4-separation, 4-separator, tetra-separation, tetra-separator}
\subjclass[2020]{05C40, 05C75, 05C83, 05E18}

\begin{document}
\thispagestyle{empty}

\begin{abstract}
    We provide a unique decomposition of every 4-connected graph into
    parts that are either quasi-5-connected, cycles of triangle-torsos and 3-connected torsos on $\le 5$ vertices, generalised double-wheels, or thickened $K_{4,m}$'s.
    The decomposition can be described in terms of a tree-decomposition but with edges allowed in the adhesion-sets.
    Our construction is explicit, canonical, and exhibits a defining property of the Tutte-decomposition.
    
    As a corollary, we obtain a new Tutte-type canonical decomposition of 3-connected graphs into parts that are either quasi-4-connected, generalised wheels or thickened $K_{3,m}$'s.
    This decomposition is similar yet different from the tri-separation decomposition~\cite{Tridecomp}.
    
    As an application of the decomposition for 4-connectivity, we obtain a new theorem~\cite{TetraTrans} characterising all vertex-transitive finite connected graphs
    as essentially quasi-5-connected or on a short explicit list of graphs.
\end{abstract}
\topskip0pt
\vspace*{\fill}
\maketitle

\vspace*{\fill}

\thispagestyle{empty}

\newpage
\thispagestyle{empty}
\tableofcontents

\newpage
\setcounter{page}{1}
\section{Introduction}

\begin{problem}\label{prob:dec}
    Is there for every $k\in\N$ an explicit canonical way to decompose every $k$-connected graph along $k$-`separators'
    into smaller pieces that are $(k+1)$-connected or `basic'?
\end{problem}
Applications of solutions for small $k$ range from the two-paths problem~\cite{GMIX,GMXIII,GMXVI} and finding excluded minors, such as for Kuratowski's theorem~\cite{ThomassenKura}, to group theory \cites{StallingsNilpotent,SmallHittingSet}, connectivity-augmentation~\cite{conny_aug_soda}, and even the graph isomorphism problem~\cite{Isomorphism}.
SPQR-trees have been introduced to represent solutions for $k=2$ \cites{SPQRincremental,SPQRonline,SPQRHopcroftTarjan,SPQRBienstock} and have since become a standard tool in the study of planar graphs from an algorithmic perspective; 
see~\cites{SPQRupwarPlanarity,Cplanarity,PlanarityPolylog,AlphaPlanarityPrelim,MinimumDepth,ComplexityEmbedding,Orthogonal,OrthogonalCubic,TriangulationAlgo,goetze2024}.

\cref{prob:dec} has a trivial solution for $k=0$, and for $k=1$ a positive answer is provided by the block-cutvertex-decomposition.
Already for $k=2$, the solution is a classical theorem of Tutte~\cite{TutteCon} from 1966.
A non-canonical solution for $k=3$ was found by Grohe~\cite{grohe2016quasi} in 2016.
Still, it was not known how to extend Tutte's approach to $k=3$ until recently, when Carmesin and Kurkofka discovered a new notion of 3-separators which allowed them to find a Tutte-type canonical decomposition of 3-connected graphs~\cite{Tridecomp}.
However, examples show that the obvious extension of their solution to $k = 4$ fails.
In particular, it was not clear whether there even exists a notion of 4-separators that allows for such an extension.

In this paper, we introduce a new type of 4-separator and, as our main result, use it to prove a Tutte-type canonical decomposition theorem for 4-connectivity.
As a corollary, we obtain a new Tutte-type canonical decomposition result for 3-connectivity that provides an alternative perspective to \cite{Tridecomp}.
We also obtain a new theorem that characterises vertex-transitive quasi-4-connected finite graphs.\medskip

\noindent\textbf{Background.}
To solve \cref{prob:dec} for $k=0$ and $k=1$, one simply cuts along all $k$-separators of the graph.
But already for $k=2$, this approach fails.
Indeed, different 2-separators can \emph{cross} in that they cut through each other, as happens in a cycle.
Then if we cut at one of the 2-separators, we lose the other.
Tutte realised that the solution is to cut the graph at all 2-separators that are \emph{totally-nested} in that they are not crossed by any 2-separators, and indeed no conflicts arise when cutting in this way.
The resulting pieces of Tutte's decomposition are 3-connected, cycles or $K_2$'s.
All decompositions for $k\le 2$ can be described in terms of tree-decompositions.
For instance, Tutte's selection of 2-separators gives a recipe for a tree-decomposition with all adhesion-sets of size~$2$ and all whose torsos\footnote{A \emph{torso} of a tree-decomposition is obtained from a bag by making each of its adhesion-sets into a clique.} are 3-connected, cycles or $K_2$'s.

Tutte's approach, when extended verbatim, fails for $k\ge 3$.
There are at least two main challenges that must be overcome.

\begin{challenge}\label{challenge1}
There are $k$-connected graphs which just slightly fail to be $(k+1)$-connected, but all whose $k$-separators are crossed.
For example, if $G$ is $k$-connected and the neighbourhoods of its vertices are precisely the $k$-separators of $G$, then the $k$-separators around neighbouring vertices cross, but $G$ fails to be $(k+1)$-connected.
Such $G$ exist for all $k\ge 3$: take a circular saw as in \cref{fig:QKC}, for instance.
We believe that such graphs should not have to be decomposed any further at level~$k$.
\end{challenge}

\begin{figure}[ht]
    \centering
    \includegraphics[height=7\baselineskip]{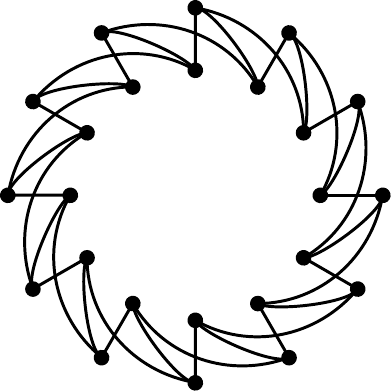}
    \caption{A \emph{circular saw} has vertex-set $\Z_n\times\Z_2$ with $n\gg k$ and every vertex $(v,0)$ sends edges to the $k$ vertices $(v,1),\ldots,(v+k-1,1)$. It is $k$-connected, and when $k\ge 3$ its neighbourhoods of vertices are precisely its $k$-separators (\cref{SawWorks}).}
    \label{fig:QKC}
\end{figure}

\begin{challenge}\label{challenge2}
We consider the $k$-connected graph $G$ depicted in \cref{fig:BigNecklace} if $k$ is odd and in \cref{fig:BigNecklaceEven} if $k$ is even.
Every $k$-separator of $G$ is formed from a red set, a non-adjacent blue set, and if $k$ is even also the apex-vertex in the centre.
Hence every $k$-separator of $G$ is crossed by another $k$-separator.
However, $G$ is far from $(k+1)$-connected, and in no way `basic': every $42k$-clique can be replaced with an arbitrarily large $(k+1)$-connected graph, so $G$ represents a whole zoo of examples.
\end{challenge}

\begin{figure}[ht]
\centering
\begin{minipage}[b]{.45\textwidth}%
    \centering
    \captionsetup{width=1\textwidth}
    \includegraphics[height=8\baselineskip]{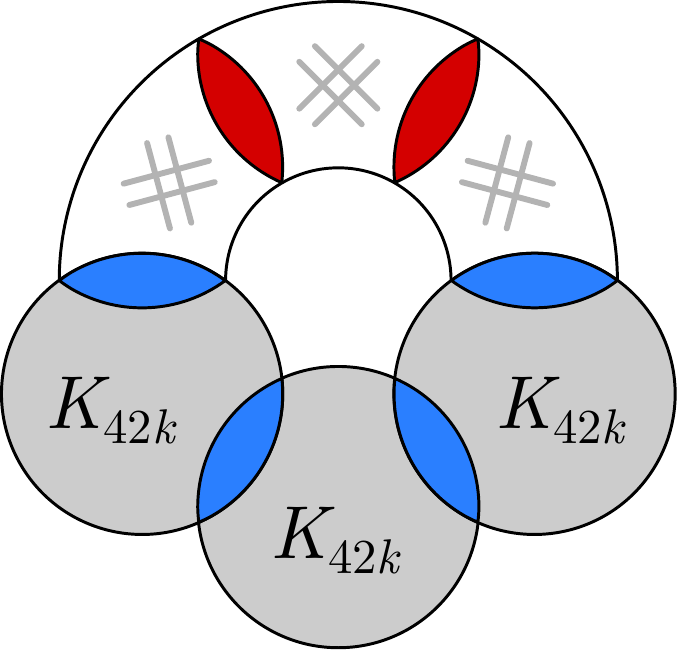}
    \captionof{figure}{$k$ is odd. Red: $\lfloor k/2\rfloor$-cliques. Blue: $\lceil k/2\rceil$-cliques. Grey indicates all possible edges.}
    \label{fig:BigNecklace}
\end{minipage}\hfill\begin{minipage}[b]{.45\textwidth}
    \centering
    \captionsetup{width=1\textwidth}
    \includegraphics[height=8\baselineskip]{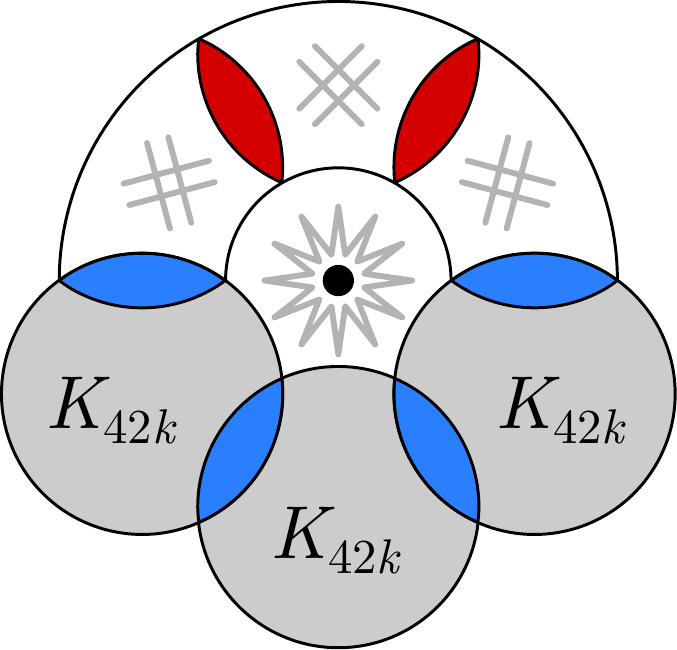}
    \captionof{figure}{$k$ is even. Red: $(\tfrac{k}{2}-1)$-cliques. Blue: $\tfrac{k}{2}$-cliques. Grey indicates all possible edges. The central vertex is an apex.}
    \label{fig:BigNecklaceEven}
\end{minipage}
\end{figure}

\noindent\textbf{Decomposing 3- and 4-connected graphs.}
We found a way to overcome the two challenges for $k=3$ and $k=4$ with a twofold approach, which we first explain for $k=4$.\smallskip

\textbf{Ad \cref{challenge1}.} We relax the notion of $5$-connected to \emph{quasi-5-connected}, an idea that we learned about from Grohe's work \cite{grohe2016quasi}.
A graph $G$ is \emph{quasi-5-connected} if it is 4-connected and every 4-separation $(A,B)$ of $G$ satisfies $|A\sm B|=1$ or $|B\sm A|=1$. 
We will not decompose quasi-5-connected graphs any further.\smallskip
    
\textbf{Ad \cref{challenge2}.}
We introduce the new notion of a \emph{tetra-separation}, which we use instead of 4-separators.
There are two key differences.
On the one hand, tetra-separations may use edges in addition to vertices to cut up the graph, an idea that we learned about from the work of Carmesin and Kurkofka \cite{Tridecomp} (we will discuss the relation later).
On the other hand, every vertex used by a tetra-separation to cut up the graph into two sides must have at least two neighbours in each side, not counting neighbours that are also used by the tetra-separation.
This degree-condition is new.

A \emph{mixed-separation} of a graph $G$ is a pair $(A,B)$ with $A\cup B=V(G)$ and $A\sm B\neq\emptyset\neq B\sm A$.
The \emph{separator} of $(A,B)$ is the set of all vertices in $A\cap B$ and all edges in $G$ with one end in $A\sm B$ and the other in $B\sm A$.
In a 4-connected graph, every mixed-separation has a separator of size~$\ge 4$ (\cref{kConMixed}).
A \emph{tetra-separation} of $G$ is a mixed-separation $(A,B)$ of $G$ whose separator has size four and which satisfies the following two conditions:\medskip

\noindent\begin{tabular}{@{}l@{}cl}
    \emph{degree-condition} & {}\,: &  every vertex in the separator has $\ge 2$ neighbours in both $A \sm B$ and $B \sm A$;\\
    \emph{matching-condition} & {}\,: & the edges in the separator form a matching.
\end{tabular}\medskip

Two tetra-separations $(A,B)$ and $(C,D)$ are \emph{nested} if, after possibly exchanging the name $A$ with $B$ or the name $C$ with $D$, we have $A\se C$ and $B\supseteq D$.
A tetra-separation of $G$ is \emph{totally-nested} if it is nested with every tetra-separation of $G$.
\cref{fig:NecklaceTotally} shows the totally-nested tetra-separations of an example-graph similar to the one shown in \cref{fig:BigNecklaceEven}.\smallskip

\begin{figure}[ht]
    \centering
    \includegraphics[height=8\baselineskip]{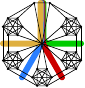}
    \caption{The totally-nested tetra-separations of a graph. Their separators are coloured.}
    \label{fig:NecklaceTotally}
\end{figure}

While the decompositions for $k$-connectivity for $k\le 2$ such as the Tutte-decomposition can be described as tree-decompositions, the decomposition of 4-connected graphs via tetra-separations requires a notion that sits in between the notions of tree-decomposition (which allows only vertices in adhesion-sets) and Wollan's notion of tree-cut decomposition (which allows only edges in adhesion-sets) \cite{WollanExcludedImmersion}.
To this end, the notion of \emph{mixed-tree-decomposition} has been introduced which allows both vertices and edges in adhesion-sets~\cite{Tridecomp}.
The set of totally-nested tetra-separations of a 4-connected graph $G$ canonically corresponds to a unique mixed-tree-decomposition of~$G$. See \cref{sec:MixedTdcs} for details.

In order to state our decomposition-result for 4-connected graphs, we need the following definitions to describe the `basic' parts of the decomposition -- those that are not quasi-5-connected.
Graph-decompositions generalise tree-decompositions in that the decomposition-tree may take the form of an arbitrary graph, such as a cycle, in which case we speak of a \emph{cycle-decomposition}~\cite{canonicalGraphDec}.
Let $X$ be a graph.
We call $X$ a \emph{generalised double-wheel} with \emph{centre} $\{u,v\}$ if $X$ is 4-connected and $u,v$ are distinct vertices of $X$ such that $X-u-v$ has a cycle-decomposition of adhesion 1 into $K_2$'s and triangles (\cref{fig:genDoubleWheel}).
We call $X$ a \emph{cycle of $X_1$-torsos, \dots , $X_k$-torsos, $Y_1$-bags, \dots, $Y_\ell$-bags} where the $X_i$ and $Y_j$ are graphs, if $X$ is 4-connected and has a cycle-decomposition with all adhesion-sets of size~2 such that for every bag either its torso is isomorphic to some~$X_i$ or the bag itself is isomorphic to some $Y_j$ (\cref{fig:CycleOfEverything}).
Recall that the only 3-connected graphs on $\le 5$ vertices are $K_4$, the $4$-wheel, $K_5$ minus an edge, and~$K_5$.
A \emph{thickened $K_{4,m}$} is obtained from a $K_{4,m}$ by turning the left side into a~$K_4$.
A \emph{sprinkled $K_{4,m}$} is obtained from a $K_{4,m}$ by adding any number of edges to the left side.

\begin{figure}[ht]
\centering
\begin{minipage}[b]{.45\textwidth}%
    \centering
    \captionsetup{width=1\textwidth}
    \includegraphics[height=7\baselineskip]{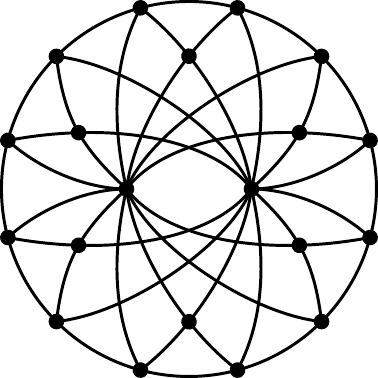}
    \captionof{figure}{A generalised double-wheel\\{\,}}
    \label{fig:genDoubleWheel}
\end{minipage}\hfill\begin{minipage}[b]{.5\textwidth}
    \centering
    \captionsetup{width=1\textwidth}
    \includegraphics[height=7\baselineskip]{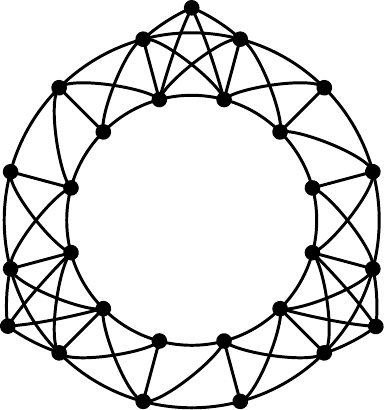}
    \captionof{figure}{A cycle of triangle-bags, $K_4$-bags and $K_5$-bags}
    \label{fig:CycleOfEverything}
\end{minipage}
\end{figure}

\begin{mainresult}\label{MainDecomp}
    Let $G$ be a 4-connected graph.
    Let $\cT(G)$ denote the mixed-tree-decomposition of $G$ that is uniquely determined by the set $N(G)$ of all totally-nested tetra-separations of~$G$.
    Every torso $\tau$ of $\cT(G)$ satisfies one of the following:
    \begin{enumerate}
        \item\label{MainDecompQuasi5con} $\tau$ is quasi-5-connected;
        \item\label{MainDecompBagel} $\tau$ is a cycle of triangle-torsos and 3-connected torsos on $\le 5$ vertices; 
        \item\label{MainDecompDoubleWheel} $\tau$ is a generalised double-wheel;
        \item\label{MainDecompK4m} $\tau$ is a thickened $K_{4,m}$ with $m\ge 0$, or the entire graph $G$ is a sprinkled $K_{4,m}$ with $m\ge 4$.
    \end{enumerate}
\end{mainresult}

\begin{remark}\label{MainDecompCanonical}
    Both $N(G)$ and $\cT(G)$ in \cref{MainDecomp} are canonical: if $\varphi\colon G\to G'$ is a graph-isomorphism, then $N(\varphi(G))=\varphi(N(G))$ and $\cT(\varphi(G))=\varphi(\cT(G))$, and the automorphism group of every $G$ induces a group-action on~$N(G)$ and on~$\cT(G)$.
\end{remark}

\begin{remark}\label{MainDecompInfinite}
    \cref{MainDecomp} includes infinite graphs $G$ with no additional restrictions whatsoever.
    More generally, all graphs in this paper are allowed to be infinite, unless stated otherwise.
\end{remark}

\cref{MainDecomp} implies a new Tutte-type canonical decomposition for 3-connectivity.
A \emph{\trisp } is defined like a tetra-separation, except that the separator is required to have size three instead of four.
A \trisp\ of a graph $G$ is \emph{totally-nested} if it is nested with every \trisp\ of~$G$.
A \emph{generalised wheel} with \emph{centre} $v$ is a 3-connected graph $X$ with a vertex $v$ such that $X-v$ has a cycle-decomposition with all adhesion-sets of size 1 and into $K_2$'s and triangles.

\begin{maincorollary}\label{3Decomp}
    Let $G$ be a 3-connected graph.
    Let $M(G)$ denote the set of all \trisp s of $G$ that are totally-nested (with regard to the \trisp s).
    Let $\cT(G)$ denote the mixed-tree-decomposition of $G$ that is uniquely determined by $M(G)$.
    Every torso $\tau$ of $\cT(G)$ satisfies one of the following:
    \begin{enumerate}
        \item\label{3DecompQuasi4con} $\tau$ is quasi-4-connected;
        \item\label{3DecompWheel} $\tau$ is a generalised wheel;
        \item\label{3DecompK3m} $\tau$ is a thickened $K_{3,m}$ with $m\ge 0$, or the entire graph $G$ is a sprinkled $K_{3,m}$ with $m\ge 4$.
    \end{enumerate}
\end{maincorollary}

Both \cref{MainDecompCanonical,MainDecompInfinite} also apply to \cref{3Decomp}.\medskip

\noindent\textbf{Comparison to the tri-separation decomposition.}
Carmesin and Kurkofka introduced the first Tutte-type canonical decomposition for 3-connected graphs \cite{Tridecomp}.
They used tri-separations instead of \trisp s to obtain their decomposition.
The main difference between tri-separations on the one hand, and \trisp s and tetra-separations on the other hand, is how their degree-conditions handle neighbours in the separator: tri-separations count them, but \trisp s and tetra-separations ignore them.
This seemingly subtle difference has a surprisingly big impact: the verbatim extension of the tri-separation approach to 4-connectivity fails \cref{challenge2}.

For 3-connected graphs, the \trisp\ decomposition can be viewed as a refinement of the tri-separation decomposition, but the details are a bit technical; see \cref{sec:comparison}.\medskip

\noindent\textbf{Algorithmic aspects.}
It is straightforward to come up with a polynomial-time algorithm that produces the decompositions in~\cref{MainDecomp} and \cref{3Decomp} for all 4- and 3-connected finite graphs.

Korhonen~\cite{korhonen2024} recently showed that there is an algorithm that, given an $n$-vertex $m$-edge finite graph~$G$ and an integer~$k\ge 1$, returns a $k$-lean tree-decomposition of~$G$ in time $k^{\cO(k^2)}n+\cO(m)$.
We believe that every 5-lean tree-decomposition of every 4-connected finite graph~$G$ can be made into a tetra-separation decomposition of~$G$ as in \cref{MainDecomp} in linear time:

\begin{conjecture}
    There is an algorithm that, given a 4-connected $n$-vertex $m$-edge finite graph~$G$, returns the tetra-separation decomposition of~$G$ from \cref{MainDecomp} in linear time $\cO(n+m)$.
\end{conjecture}

A positive answer would also provide a linear-time algorithm for the decomposition of 3-connected finite graphs in \cref{3Decomp} by following the short proof of \cref{3Decomp}.

Total-nestedness makes parallel computing a feasible option for computing the tetra-separation decomposition even more efficiently.
Indeed, all the partial solutions found in parallel, which come in the form of sets of totally-nested tetra-separations, can always be combined without conflict.
\medskip

\noindent\textbf{Combining the decompositions for different~$k$.}
Every finite graph~$G$ can be decomposed canonically by combining the partial solutions to \cref{prob:dec} inductively, as follows.
\begin{enumerate}[label=\arabic*)]
    \item First, we trivially decompose $G$ into its connected components.
    \item Each connected component we further decompose with the block-cutvertex decomposition into bags that are 2-connected or~$K_2$'s.
    \item Each 2-connected bag we further decompose with the Tutte-decomposition into torsos that are 3-connected, cycles or~$K_2$'s.
    \item Each 3-connected torso we further decompose with \cref{3Decomp} into torsos that are quasi-4-connected, generalised wheels or thickened~$K_{3,m}$'s. Alternatively, we use~\cite{Tridecomp}.
\end{enumerate}
We have relaxed `4-connected' to `quasi-4-connected' to overcome \cref{challenge1}, and we now pay the price: it is not possible to directly apply \cref{MainDecomp} to the quasi-4-connected torsos.
Here, we have two options.
On the one hand, we could try to generalise \cref{MainDecomp} from 4-connected to quasi-4-connected graphs.
We have explored this option, but did not find a practical result; see \cref{fig:Quasi4Necklace} for an illustrative example of how the naive approach fails.
On the other hand, we could try to gently make the quasi-4-connected graphs into 4-connected ones to bring them into the scope of \cref{MainDecomp}.
We have opted for this solution, and to this end introduce the following canonical $Y$--$\Delta$ operation that deals with all degree-three vertices simultaneously.

\begin{figure}[ht]
    \centering
    \includegraphics[height=6\baselineskip]{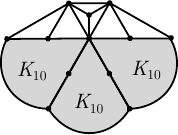}
    \caption{A quasi-4-connected version of \cref{challenge2}: all tetra-separations are crossed, yet the graph is far from $5$-connected and each $K_{10}$ can be replaced with an arbitrary 5-connected graph}
    \label{fig:Quasi4Necklace}
\end{figure}

Let $H$ be an arbitrary graph.
The following definition of $H^\Delta$ is supported by~\cref{fig:GDeltaDef}.
Let $U$ be the set of all vertices in $H$ of degree three.
We obtain $H^s$ from $H$ by subdividing every edge that joins two vertices in~$U$.
Then we obtain $H^{s\Delta}$ from $H^s$ by making each neighbourhood $N_{H^s}(u)$ around a vertex $u\in U$ into a triangle~$\Delta(u)$.
Finally, we let $H^\Delta:=H^{s\Delta}-U$.
If $H$ is quasi-4-connected and has $>6$ vertices, then $H^\Delta$ is 4-connected by \cref{GDelta4con}.
The operation $H\mapsto H^\Delta$ is canonical: every isomorphism $H_1\to H_2$ induces an isomorphism $H^\Delta_1\to H^\Delta_2$.

\begin{figure}[ht]
    \centering
    \includegraphics[width=\linewidth]{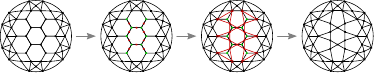}
    \caption{From $H$ (left) over $H^s$ and $H^{s\Delta}$ to $H^\Delta$ (right). 
    Central graphs: vertices in $U$ are green, neighbours of~$U$ are circled red, subdivision vertices are fully red, and the triangles $\Delta(u)$ are red.}
    \label{fig:GDeltaDef}
\end{figure}

With the canonical $Y$--$\Delta$ operation at hand, we conclude decomposing $G$ as follows:
\begin{enumerate}[resume,label=\arabic*)]
    \item Each quasi-4-connected torso $\tau$ with $>6$ vertices we make into a 4-connected graph~$\tau^\Delta$.\\ (Quasi-4-connected graphs on $\le 6$ vertices are treated as `basic' outcome.)
    \item Finally, we decompose each 4-connected $\tau^\Delta$ with \cref{MainDecomp}.
\end{enumerate}

The operation $\tau\mapsto \tau^\Delta$ harmonises well with the other decomposition steps.
Indeed, we have followed all steps to classify all vertex-transitive finite graphs of low connectivity~\cite{TetraTrans}.\medskip

\noindent\textbf{When canonicity and an explicit description matter.}
There are applications for partial solutions to \cref{prob:dec}, and in particular for \cref{MainDecomp} and \cite{Tridecomp}, that require canonicity and explicitness, often both.
Examples of these applications include:
\begin{enumerate}[leftmargin=*,label={\Large $\cdot$}]
    \item \underline{Vertex-transitive graphs.}
    Let $\cV$ denote the class of all vertex-transitive finite connected graphs.
    Can the graphs in~$\cV$ be classified?
    Droms, Servatius and Servatius~\cite{infiniteSPQR} used the Tutte-decomposition to classify all $G\in\cV$ of connectivity~$<3$: they are the cycles and the complete graphs on~$\le 2$ vertices.
    The graphs $G\in\cV$ of connectivity~3 were classified by Carmesin and Kurkofka using their tri-separation decomposition~\cite{Tridecomp}.
    We used \cref{MainDecomp} in combination with the canonical $Y$--$\Delta$ operation to classify all graphs in $\cV$ that are not essentially 5-connected~\cite{TetraTrans}.
    \item \underline{Connectivity augmentation.}
    Carmesin and Sridharan used the canonicity and explicitness of the tri-separation decomposition \cite{Tridecomp} to construct the first FPT-algorithm for connectivity augmentation from 0 to 4~\cite{conny_aug}.
    That is, they showed that there is an algorithm with runtime $C(\ell)\cdot n^{\cO(1)}$ and the following specifications:\\
    \begin{tabular}{@{}l@{}cl}
        Input & {}\,: & a graph $G$ on~$n$ vertices, a number $\ell\in\N$ and a set $F\se [V(G)]^2\sm E(G)$ of non-edges.\\
        Output & {}\,: & a set $X\se F$ of $\le\ell$ edges such that $G+X$ is 4-connected, if such an $X$ exists, and\\
        & & `no' otherwise.
    \end{tabular}
    We strongly believe that \cref{MainDecomp} makes it possible to extend their algorithm to 5-connectivity.
    \item \underline{Finding excluded minors.} A challenge at the interface of computer science and graph-minor theory is to determine the excluded minors for particular minor-closed graph-classes.
    Indeed, while the graph-minor theorem implies the existence of a cubic-time membership test for each minor-closed graph-class, explicit knowledge of the excluded minors is required to obtain an explicit description of that test~\cite{GMXIII}.
    Partial solutions to \cref{prob:dec} for $k\le 3$ can be used to determine the excluded minors for planar graphs and series-parallel graphs.
    We have conviction that solutions for larger~$k$ imply exact structure theorems for other graph-classes, such as the class of apex-graphs, and that these structure theorems can be used to determine the excluded minors for these classes and thereby provide explicit cubic-time membership tests.
    \item\underline{Graph isomorphism.} Canonical tree-decompositions offer angles of attack on finding efficient algorithms for the graph isomorphism problem on particular classes of graphs.
    Recently, an FPT-algorithm was found~\cite{Isomorphism} that solves the graph-isomorphism problem for $H$-minor-free graphs $G,G'$ in time $f(H)\cdot n^{\cO(1)}$.
    In that paper, the authors wish for a canonical tree-decomposition into bags that are highly cohesive or of bounded size~\cite[{}1.3]{Isomorphism}, and they mention long cycles as an illustrative example of an obstruction to the existence of such decompositions.
    Our results show that such canonical tree-decompositions exist for low-order connectivity if one trades `bounded size' for `basic' and an explicit description thereof.
    Hence we wonder if our results allow for even more efficient algorithms that solve the graph isomorphism problem for, say, graphs without quasi-5-connected minors.
    \item \underline{Cops and robber.}
    In the cops and robber game of Robertson and Seymour, 
    $k$ cops have a winning strategy for catching a robber in a graph iff that graph has a tree-decomposition of width~$<k$ \cite{GMcopsRobber}.
    Consider now a randomised variation of this game where the robber is invisible, the cops can visit vertices at most once, and both players use mixed strategies.\footnote{A mixed strategy is a random distribution over the set of all pure (i.e.\ deterministic) strategies for a player.}
    Now the maximum winning probability for $k$~cops depends on the minimum number of leaf-bags attained by any tree-decomposition of width~$<k$.
    Our results give such tree-decompositions for $k$-connected graphs when $k\le 4$~\cite{CopsInvisibleRobber}.
    \item \underline{Local separators and graph-decompositions.}
    A recent development in graph-minor theory is the study of local separators: vertex-sets that separate a ball of given radius around them, but not necessarily the entire graph.
    Local separators capture choke points in real-world infrastructure networks~\cite{Ponyhof}.
    Like nested separators correspond to tree-decompositions, nested local separators correspond to graph-decompositions (such as the cycle-decompositions we used in \cref{MainDecomp}).
    Recent results include a Tutte-type decomposition for local 2-separators \cite{Local2sep} with applications to data science~\cite{Sarah}, a theory of locally chordal graphs \cite{locallychordal}, and algorithmic constructions of canonical graph-decompositions that display the local bottlenecks in graphs \cites{canonicalGraphDec,Ponyhof}.
    Can \cref{MainDecomp} and \cref{3Decomp} be extended to local separators?
    \item \underline{Finite groups.}
    Stallings' theorem applies tree-decomposition methods to Cayley graphs to detect product structure in groups.
    However, Stallings' approach only works for infinite groups.
    Local separators have recently been used to prove a low-order Stallings-type theorem for finite nilpotent groups \cite{StallingsNilpotent}.
    The proof heavily relies on the Tutte-type decomposition for local 2-separators \cite{Local2sep} and exploits the combination of canonicity and total-nestedness.
    The next step towards fully generalising Stallings theorem to finite groups will be to use versions for local separators of \cref{MainDecomp} and of either \cref{3Decomp} or \cite{Tridecomp}.
\end{enumerate}\smallskip

\subsection{Overview of the proof}

To get started, let $G$ be a 4-connected graph.
Let $N$ denote its set of totally-nested tetra-separations, and let $\cT=(T,\cV)$ denote the mixed-tree-decomposition of $G$ that is uniquely determined by~$N$.
Let $\tau$ denote the torso of $\cT$ at an arbitrary node $t$ of $T$ (see \cref{sec:MixedTdcs} for the definition of torso in the context of mixed-tree-decompositions).
It will be straightforward to check that $\tau$ is 4-connected or a $K_4$.

In order to show that $\tau$ satisfies one of the outcomes of \cref{MainDecomp}, we will study how $\tau$ interacts with the tetra-separations of $G$.
Indeed, the outcome will be determined by this.
Let $(A,B)$ be a tetra-separation of~$G$, which may possibly be crossed.
We say that $(A,B)$ \emph{interlaces} $\tau$ if, roughly speaking, both $A\sm B$ and $B\sm A$ intersect~$\tau$.
Here we may think of $(A,B)$ as `cutting through'~$\tau$.
Whether $\tau$ is interlaced or not already yields a first outcome:
\begin{equation}\label{ProofOverviewQuasi5con}
    \textit{If $\tau$ is not interlaced by any tetra-separation of $G$, then $\tau$ is quasi-5-connected.}
\end{equation}
(See \cref{Quasi5conSummary}.)
A technical lemma from \cite{Tridecomp} does most of the work of proving \cref{ProofOverviewQuasi5con}.

By \cref{ProofOverviewQuasi5con}, we may assume from now on that some tetra-separation $(A,B)$ of $G$ interlaces~$\tau$.
Then $(A,B)$ is not in~$N$, so some tetra-separation $(C,D)$ crosses $(A,B)$.
The remaining outcomes will be determined by the way in which $(A,B)$ and $(C,D)$ cross.
More specifically, the outcome will be determined by the size of the vertex-centre~$Z$, which consists of all vertices in the intersection of the separators of $(A,B)$ and $(C,D)$.
We will show that
\begin{equation}\label{ProofOverviewCrossingLemma}
    |Z| \in\{0,2,4\}.
\end{equation}
(See \allref{keylem:crossing}.)
So there will be exactly three outcomes.
Here we focus only on the cycle-of-graphs outcome, which we get for $|Z|=0$.
This outcome is the hardest to prove, and we use all the new methods introduced in this paper in the proof of this outcome.
It is also different from all situations encountered in \cite{Tridecomp}.

Assume now that $|Z|=0$.
As our first step, we will use $(A,B)$ and $(C,D)$ to cut $G$ into four 2-connected pieces, arranged in the shape of a 4-cycle, and apply the Tutte-decomposition to each piece; see \cref{fig:TutteBagelIntro}.
Since $G$ is 4-connected, the Tutte-decompositions will be path-decompositions.
We will then merge the four Tutte-decompositions into a cycle-decomposition $\cO=(O,(G_o)_{o\in V(O)})$.
So all adhesion-sets of $\cO$ have size two.
We then use the 4-connectivity of $G$ to note that the torsos of $\cO$ are triangles, 4-cycles, or 3-connected.
Moreover, we will find some rules on how these types of torsos may be arranged on $O$.
For instance, a torso that is a 4-cycle must always have two 3-connected torsos as neighbours, and if two triangle-torsos are neighbours then their adhesion-set is spanned by an edge.

\begin{figure}[ht]
    \centering
    \includegraphics[height=9\baselineskip]{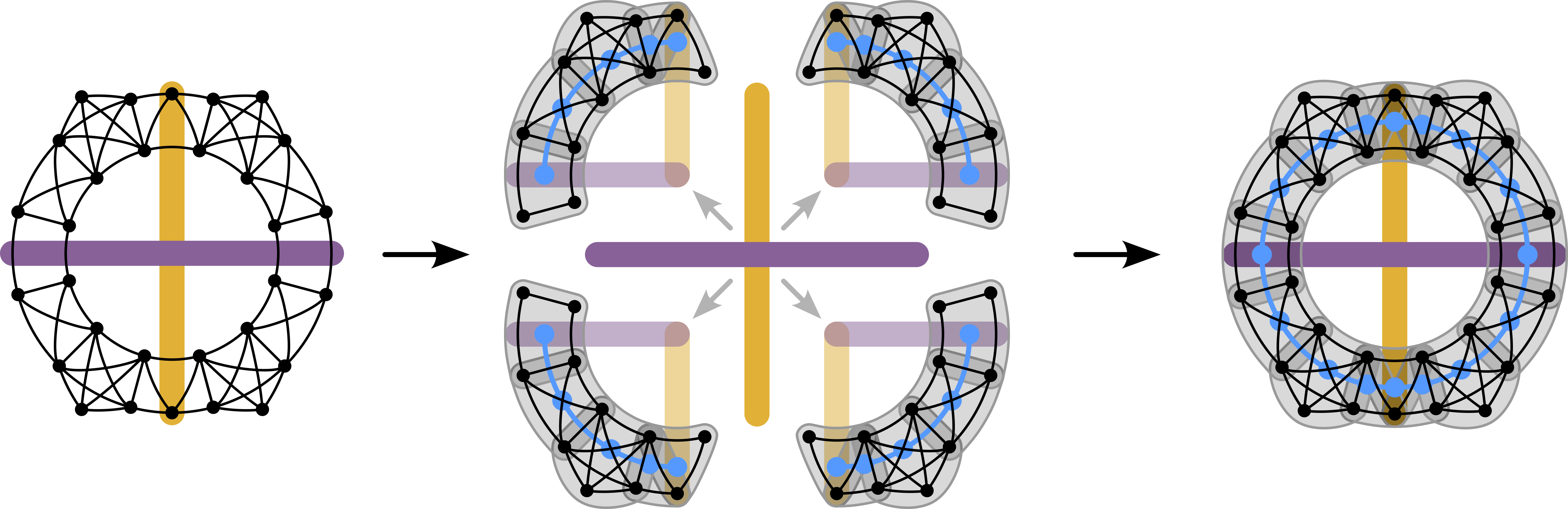}
    \caption{Constructing a cycle-decomposition from four Tutte-decompositions}
    \label{fig:TutteBagelIntro}
\end{figure}

Every node $o$ of $O$ defines a ($\le 4$)-separation of $G$: its sides are $V(G_o)$ and $\bigcup_{o'\in V(O-o)} V(G_{o'})$, so its separator is the union of the two adhesion-sets at $G_o$.
From this ($\le 4$)-separation, we then obtain a candidate $(U_o,W_o)$ for a totally-nested tetra-separation by performing slight alterations to the separator, replacing some vertices with edges in a systematic way.
The first key result will be a characterisation of those torsos $H_o$ that give rise to a totally-nested tetra-separation $(U_o,W_o)$ in terms of an efficiently testable property: being \emph{good}.
For example, only 3-connected torsos can be good, and having $\ge 6$ vertices in addition to being 3-connected suffices for being good.
\begin{equation}\label{ProofOverviewBagelGood}
    \textit{$(U_o,W_o)$ is a totally-nested tetra-separation of $G$ if and only if $H_o$ is good.}
\end{equation}
(See \cref{cor:bagel-totally-nested}).
In order to show the hard backward implication of \cref{ProofOverviewBagelGood}, we will introduce a 
\begin{equation}\label{ProofOverviewExternallyConnected}
    \textit{characterisation of total-nestedness in terms of connectivity-properties of separators.}
\end{equation}
(See \cref{keylem:nestedness-external-connectivity}.)
Examples of connectivity-properties of a tetra-separator $S(A,B)$ that are necessary for $(A,B)$ to be totally-nested include, roughly speaking, that
for every bipartition $(\pi_1,\pi_2)$ of $S(A,B)$ into two classes $\pi_i$ of size~2 there must be
\begin{itemize}[leftmargin=*]
    \item 5 independent $\pi_1$--$\pi_2$ paths in $G$ on the one hand, and
    \item 3 independent paths between the two elements of $\pi_1$ in $G-\pi_2$ on the other hand.
\end{itemize}
A major difference to tri-separations is that the connectivity-properties required of the separators are much more complex, as tetra-separations can cross in significantly more complicated ways than tri-separations; see \cref{fig:ComplexCrossing} for an example.

\begin{figure}[ht]
    \centering
    \includegraphics[height=8\baselineskip]{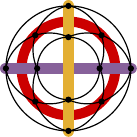}
    \caption{Three tetra-separations cross in a way that tri-separations cannot.}
    \label{fig:ComplexCrossing}
\end{figure}

With \cref{ProofOverviewExternallyConnected} at hand, the problem of proving the hard backward implication of \cref{ProofOverviewBagelGood} becomes a problem about finding many paths in many configurations, roughly speaking.
To this end, we develop a number of lemmata for finding such paths systematically, and then combine them to prove \cref{ProofOverviewBagelGood}.
Finding the paths takes a lot of work, especially when $O$ is short; see \cref{fig:ComplexPaths} for example.

\begin{figure}[ht]
    \centering
    \includegraphics[height=8\baselineskip]{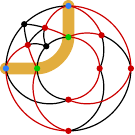}
    \caption{The bag $G_o$ is the top left quarter. Its candidate $(U_o,W_o)$ for a totally-nested tetra-separation is indicated in yellow. 
    The sets $\pi_1$ and $\pi_2$ consist of the blue and the green vertices, respectively. 
    As a small part of proving \cref{ProofOverviewBagelGood} via \cref{ProofOverviewExternallyConnected}, we have to find 5 independent $\pi_1$--$\pi_2$ paths (red).
    We find two such paths in $G_o$ using Menger as $G_o$ is good.
    To find the other three, we essentially go around $O$ in both directions starting from~$o$, constructing one path in the one direction and two in the other.
    The example shows that the construction must be highly efficient to avoid collision of the paths when the decomposition-cycle $O$ is short.}
    \label{fig:ComplexPaths}
\end{figure}

By \cref{ProofOverviewBagelGood}, each tetra-separation $(U_o,W_o)$ for a good torso $H_o$ corresponds to a directed edge $e_o$ of $T$.
The second key result about $\cO$ will be that
\begin{equation}\label{ProofOverviewBagelSplitting}
    \textit{The edges $e_o$ are precisely the directed edges of the form $(s,t)$ with $st\in E(T)$,}
\end{equation}
where $t$ is the node of $T$ with $\tau$ as torso.
Phrased differently, the tetra-separations $(U_o,W_o)$ from good torsos $H_o$ are precisely the elements of $N$ that `carve out'~$\tau$.
(See \cref{BagelSplittingStar}.)

Once we have proved \cref{ProofOverviewBagelSplitting}, we are almost done.
Indeed, it now follows that the torso $\tau$ can be obtained from $G$ by essentially replacing the good torsos of $\cO$ with $K_4$'s and contracting a few edges.
Then it will be fairly immediate to see that $\tau$ is a cycle of triangle-torsos and 3-connected torsos on $\le 5$ vertices.
(See \cref{bagelTorso}.)

Proving \cref{ProofOverviewBagelSplitting}, however, is a little tricky.
We will assume for a contradiction that \cref{ProofOverviewBagelSplitting} fails, so some totally-nested tetra-separation $(E,F)$ will `cut through' $\cO$.
Since $(E,F)$ is nested with both $(A,B)$ and $(C,D)$, we have $E\se A\cap C$, say; see \cref{fig:MovingLink}.
We then start adjusting both $(A,B)$ and $(C,D)$: by moving bags of $\cO$ from $A\sm B$ to $B\sm A$ and from $C\sm D$ to $D\sm C$, one bag at a time, while using the total-nestedness of $(E,F)$ to maintain the inclusion $E\se A\cap C$.
Eventually, there will be essentially only one bag $G_o$ left, and we will see that its torso $H_o$ is good.
We will then show that $(E,F)=(U_o,W_o)$ is the only possibility.
Since $(U_o,W_o)$ does not `cut through' $\cO$, this will yield a contradiction.\medskip

\begin{figure}[ht]
    \centering
    \includegraphics[height=8\baselineskip]{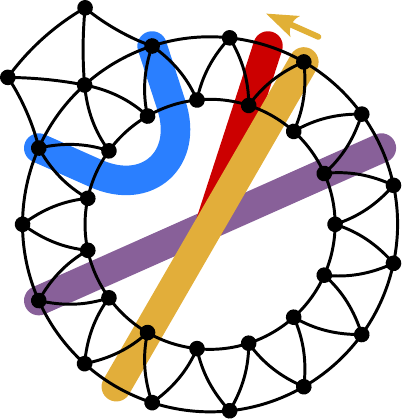}
    \caption{$(E,F)$ (blue) is nested with $(A,B)$ (yellow) and $(C,D)$ (purple). We use the structure of $\cO$ to move $(A,B)$ closer to $(E,F)$ by moving one bag to the other side. This replaces the top half of yellow with red.}
    \label{fig:MovingLink}
\end{figure}

\noindent\textbf{More related work.}
Greedy decompositions for 3-connectivity include \cite{grohe2016quasi} and \cite{4tangles}.
We would also like to mention the recent work of Esperet, Giocanti and Legrand-Duchesne \cite{Ugo}.
A~related question in graph-minor theory is which substructures of $k$-connected graphs can be deleted or contracted while preserving $k$-connectivity.
Work in this area includes that of Ando, Enomoto and Saito~\cite{ando1987contractible} and of Kriesell~\cite{KriesellNonEdges,KriesellContractibleSubgraph,KriesellAlmostAll,KriesellTriangleFree,kriesell2008number,KriesellVertexSuppression} on 3-connectivity, and \cites{generatingInternally4con,chain4con,Slater4con,Splitting4con,Martinov,MartionvTwo,Number4con,Quasi4conStructure} on 4-connectivity.
The typical subgraphs of 3-connected and 4-connected graphs have been determined by Oporowski, Oxley and Thomas~\cite{typical3con4con}.
Works that classify the structure of $k$-cuts in $k$-connected graphs include~\cite{StructureMinimumVertexCuts,ReinventingWheel}.

Mixed-separators (without additional conditions on their vertices) have been used before, see for example Mader's theorem~\cite{Mader,bibel} or~\cite{Beineke_Harary_1967,OnMixedCony,TwoPointFiveCony}.

The structure of the $k$-separations of $k$-connected matroids is a well-studied topic, see for example the work \cites{Structure2sepsMatroids,cunningham_edmonds_1980,Seymour2sepMatroid} on 2-connectivity, the work of Oxley, Semple and Whittle~\cite{Structure3sepsMatroids} on 3-connectivity, and the works \cites{Structure4sepsMatroids,4connectedMatroids} on 4-connectivity.

For directed graphs, the case $k=1$ of \cref{prob:dec} has recently been solved almost completely by Bowler, Gut, Hatzel, Kawarabayashi, Muzi and Reich~\cite{Flo}.

Our work complements the tangle-approach to decomposing graphs 
\cites{GroheTangles3,FiniteSplinters,RefiningToT,StructuralASS,AlbrechtsenRefiningToTinASS,AlbrechtsenRefiningTDCdisplayKblocks,AlbrechtsenOptimalToT,CarmesinToTshort,CDHH13CanonicalAlg,CDHH13CanonicalParts,confing,CG14:isolatingblocks,ProfilesNew,infinitesplinter,jacobs2023efficiently,entanglements,Kblocks} which has algorithmic applications such as \cites{elbracht2020tangles,ComputingWithTangles,korhonen2024}; see \cite{TangleBook} for an overview.
There is also a connection to twin-width, see for example the recent work of Heinrich and Raßmann~\cite{TwinWidth}.
\medskip

\noindent\textbf{Organisation of the paper.}
\cref{sec:essentials} introduces or recalls all essential terminology.
\cref{sec:ProofsStrengthenings} provides stronger but more technical versions of \cref{MainDecomp} and \cref{3Decomp}, and proves all while assuming a number of key lemmata that remain to be proved later in the paper.
\cref{sec:crossingAnalysis} provides a characterisation of how tetra-separations cross.
\cref{sec:shifts} introduces a systematic way for obtaining tetra-separations from mixed-4-separations.
\cref{sec:ExternalConnectivity} gives a characterisation of total-nestedness for tetra-separations in terms of connectivity-properties of their separators.
The next sections investigate the different outcomes of \cref{MainDecomp}:
\cref{sec:K4m} deals with the $K_{4,m}$;
\cref{sec:DoubleWheel} deals with the double-wheel;
\cref{sec:CycleOfGraphs} deals with the cycle-of-graphs; and
\cref{sec:Quasi5con} deals with quasi-5-connectivity.
This will conclude the proof of \cref{MainDecomp}.
\cref{sec:Angry} introduces an `angry theorem' for 4-connectivity: a characterisation of the 4-connected graphs all whose tetra-separations are crossed (hence the name).
\cref{sec:YDelta} shows that the canonical $Y$--$\Delta$ operation makes quasi-4-connected graphs into 4-connected ones.
\cref{sec:saw} proves that the circular saws from \cref{challenge1} really exhibit the properties we used.
\cref{sec:comparison} compares the \trisp -decomposition (\cref{3Decomp}) with the tri-separation decomposition \cite{Tridecomp}.

\section{Essential terminology}\label{sec:essentials}

This section introduces all terminology used by \cref{MainDecomp}.
For graph-theoretic terminology we follow \cite{bibel}.

\subsection{Tetra-separations}
A \defn{mixed-separation\pl} of a graph $G$ is a pair $(A,B)$ such that $A \cup B = V(G)$. 
The separator $\defnMath{S(A,B)}$ of $(A,B)$ is the disjoint union of the vertex set $A \cap B$ and the edge-set $E(A \sm B, B \sm A)$. 
The \defn{order} of $(A,B)$ is the size of $S(A,B)$. 
A mixed-separation\pl{} of order $k$ is also called \defn{mixed-$k$-separation\pl}.
A mixed-separation\pl{} is \defn{proper} if both $A \sm B$ and $B \sm A$ are nonempty. 
In this case, we refer to $(A,B)$ as a \defn{mixed-separation} without the plus-sign. A mixed-separation of order $k$ is also called \defn{mixed-$k$-separation}.
A mixed-separation with $k$ vertices in its separator but no edges is called a \defn{separation}, or \defn{$k$-separation}.

\begin{definition}
\label{dfn:tetra-separation}
    A \defn{tetra-separation} of a graph $G$ is a mixed-4-separation $(A,B)$ of~$G$ that satisfies the following two conditions:
    
    \noindent\begin{tabular}{@{}ll}
        \defn{degree-condition}: & \label{itm:tetra-separation-1} every vertex in $A \cap B$ has $\ge 2$ neighbours in both $A \sm B$ and $B \sm A$;\\
        \defn{matching-condition}: & \label{itm:tetra-separation-2} the edges in $S(A,B)$ form a matching.
    \end{tabular}
\end{definition}

\begin{definition}
    A \defn{\trisp } of a graph $G$ is a mixed-3-separation of $G$ that satisfies both the degree-condition and the matching-condition.
\end{definition}

Recall that a graph $G$ is \defn{quasi-$k$-connected} if it is $(k-1)$-connected and every $(k-1)$-separation $(A,B)$ of $G$ satisfies $|A\sm B|= 1$ or $|B\sm A|= 1$.

\begin{lemma}\label{quasi5conVtetra}
    Let $G$ be a 4-connected graph.
    \begin{enumerate}
        \item\label{quasi5conVtetra1} If $G$ has no tetra-separation, then $G$ is quasi-5-connected.
        \item\label{quasi5conVtetra2} If $G$ is quasi-5-connected and $|G|\ge 8$, then $G$ has no tetra-separation.
    \end{enumerate}
\end{lemma}

We will prove \cref{quasi5conVtetra} in \cref{sec:LeftRightShift}.

\begin{example}
    The lower-bound $|G|\ge 8$ in \cref{quasi5conVtetra}~\ref{quasi5conVtetra2} is tight: the graph depicted in \cref{fig:quasi5conVtetra} is quasi-5-connected, has $7$ vertices, but has a tetra-separation.
\end{example}

\begin{figure}[ht]
    \centering
    \includegraphics[height=5\baselineskip,angle=180,origin=c]{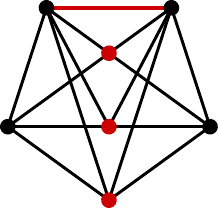}
    \caption{A small quasi-5-connected graph with a tetra-separation (red)}
    \label{fig:quasi5conVtetra}
\end{figure}

The mixed-separations\pl\ of a graph come with the usual partial ordering:
\[
    (A,B)\;\defnMath{\le}\; (C,D):\Leftrightarrow A\se C\text{ and }B\supseteq D.
\]
Two mixed-separations\pl{} $(A,B)$ and $(C,D)$ are \defn{nested} if, after possibly switching the name~$A$ with~$B$ or the name~$C$ with~$D$, we have $(A,B)\le (C,D)$. 
Otherwise, they \defn{cross}.
A set of mixed-separations\pl\ is \defn{nested} if all its elements are pairwise nested.
In the context of $(A,B)$ and $(C,D)$ crossing, the \defn{vertex-centre} is the vertex-set $A\cap B\cap C\cap D$.

A tetra-separation of a graph~$G$ is \defn{totally-nested} if it is nested with every tetra-separation of~$G$.

\subsection{Stars, mixed-tree-decompositions and torsos}\label{sec:MixedTdcs}
The following definition is supported by \cref{fig:MixedTDC}.
Let $G$ be a graph.
A pair $\cT=(T,(V_t)_{t\in T})$ of a tree $T$ and a family of vertex sets $V_t\se V(G)$ indexed by the nodes $t\in T$ is a \defn{mixed-tree-decomposition} of~$G$ if it satisfies the following two conditions:
\begin{enumerate}[label=\textnormal{(M\arabic*)}]
    \item $V(G)=\bigcup_{t\in T}V_t$;
    \item\label{M2} the subgraph of $T$ induced by $\{\,t\in T:v\in V_t\,\}$ is connected for every vertex $v\in G$.
\end{enumerate}
We refer to the vertex sets $V_t$ as \defn{bags}.
Mixed-tree-decompositions generalise both tree-decompositions and tree-cut decompositions.
Every oriented edge $(t_1,t_2)$ of $T$ \defn{induces} a mixed-separation\pl\ $(A_1,A_2)$ where $A_i=\bigcup_{t\in V(T_i)}V_t$ for the component $T_i$ of $T-t_1 t_2$ with $t_i\in T_i$.
This defines a map $\defnMath{\alpha_{\cT}(t_1,t_2)}:=(A_1,A_2)$ with domain the set $\defnMath{\vec E(T)}$ of oriented edges of~$T$.
The set of induced mixed-separations\pl\ of $\cT$ is nested.
It is \defn{symmetric}: for each element $(A_1,A_2)$ it also contains $(A_2,A_1)$.
We call $\cT$ \defn{proper} if all its induced mixed-separations\pl\ are proper.
The \defn{adhesion-set} of an edge $t_1 t_2$ of $T$ is the separator of $\alpha_{\cT}(t_1,t_2)$, which is well-defined.

\begin{figure}[ht]
    \centering
    \includegraphics[height=5\baselineskip]{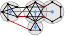}
    \caption{A mixed-tree-decomposition. The decomposition-tree is blue. Bags are grey. Edges and vertices in adhesion-sets are red.}
    \label{fig:MixedTDC}
\end{figure}

A \defn{star of mixed-separations\pl }, or \defn{star} for short, is a set $\sigma=\{\,(A_i,B_i):i\in I\,\}$ of mixed-separations\pl\ $(A_i,B_i)$ such that $(A_i,B_i)\le (B_j,A_j)$ for all distinct $i,j\in I$.
For every node $t$ of $T$, the set $\sigma_t:=\{\,\alpha_{\cT}(s,t):st\in E(T)\,\}$ is a star.
Now let $S$ be a set of mixed-separations\pl\ with $\sigma\se S$.
We call $\sigma$ a \defn{splitting star} of $S$ if $\sigma$ is a star and for every $(C,D)\in S$ there is $i\in I$ with either $(C,D)\le (A_i,B_i)$ or $(D,C)\le (A_i,B_i)$.
If $\cT$ is proper, then the splitting stars of the image of $\alpha_{\cT}$ are precisely the stars~$\sigma_t$.
Conversely, if $S$ is a nested symmetric set of mixed-separations of a finite graph $G$, then there is a unique (up to isomorphisms for decomposition-trees) mixed-tree-decomposition $\cT(S)$ of $G$ such that $\alpha_{\cT(S)}$ is a bijection $\vec E(T)\to S$.
The mixed-tree-decomposition $\cT(S)=(T,\cV)$ has an explicit description (see \cite[§3.1]{Tridecomp}): 
\begin{enumerate}
    \item the nodes of $T$ are precisely the splitting stars of $S$;
    \item two nodes $\sigma,\sigma'$ form an edge iff there is $(A,B)\in\sigma$ with $(B,A)\in\sigma'$;
    \item $V_\sigma:=\bigcap\,\{\,B:(A,B)\in\sigma\,\}$ 
\end{enumerate}
where $V_\emptyset:=V(G)$ if necessary.
Then $S$ can be recovered from $\cT(S)$: it equals the image of $\alpha_{\cT(S)}$.
In this sense, $S$ and $\cT(S)$ contain the same information, and we can switch at will between the two perspectives that they offer.

When $G$ is infinite and $S$ is the set of totally-nested tetra-separations of $G$, then $\cT(S)$ is a mixed-tree-decomposition of $G$ by standard arguments as in \cite[Lemma~3.3.10]{Tridecomp}.

A mixed-separation $(C,D)$ \defn{interlaces} a star $\sigma$ of mixed-separations if for every $(A,B)\in\sigma$ we have either $(A,B)<(C,D)$ or $(A,B)<(D,C)$.
The notion of interlacing yields a characterisation of splitting stars:
a star $\sigma$ included in a nested set $S$ of mixed-separations of a graph $G$ is a splitting star of $S$ if and only if no element of $S$ interlaces~$\sigma$ \cite[Lemma~2.2.3]{Tridecomp}.

Let $\cT=(T,\cV)$ be a mixed-tree-decomposition of a graph~$G$.
For a mixed-separation\pl\ $(A,B)$, we let the vertex-set $\defnMath{\lambda(A,B)}$ consist of all vertices in $S(A,B)$ and the endvertex in $A\sm B$ for each edge in $S(A,B)$.
The \defn{expanded-torso} of $t$, or of $V_t$, is obtained from the subgraph of $G$ induced by
\[
    V_t\cup\bigcup_{st\in E(T)}\lambda(\alpha_{\cT}(s,t))
\]
by making each of the vertex-sets $\lambda(\alpha_{\cT}(s,t))$ into a clique.
From this expanded-torso we obtain the \defn{compressed-torso}, or \defn{torso} for short, by contracting all edges contained in adhesion-sets of the edges $st\in E(T)$; see \cref{fig:MixedTorsos}.

More generally, let $\sigma=\{\,(A_i,B_i):i\in I\,\}$ be a star of mixed-separations\pl .
The \defn{expanded-torso} of $\sigma$ is obtained from the subgraph $V_\sigma\cup\bigcup_{i\in I}\lambda(A_i,B_i)$
by making each of the vertex-sets $\lambda(A_i,B_i)$ into a clique.
The \defn{compressed-torso}, or \defn{torso}, of $\sigma$ is obtained from this expanded torso by contracting all edges contained in any separators $S(A_i,B_i)$.
The two notions of torsos are compatible: the torso of a node $t$ of $T$ equals the torso of the star $\sigma_t$.

\begin{figure}[ht]
    \centering
    \includegraphics[height=7\baselineskip]{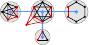}
    \caption{The torsos of the mixed-tree-decomposition depicted in \cref{fig:MixedTDC}.}
    \label{fig:MixedTorsos}
\end{figure}

\subsection{Graph-decompositions}
Let $G$ and $H$ be graphs.
Let $\cG = (G_h : h \in V(H))$ be a family of subgraphs $G_h \subseteq G$ indexed by the nodes of~$H$. 
The pair $(H, \cG)$ is an \defn{$H$-decomposition} or \defn{graph-decomposition} of $G$ if
\begin{enumerate}[label=\rm{(H\arabic*)}]
    \item\label{GraphDec1} $G = \bigcup_{h \in V(H)} G_h$, and
    \item\label{GraphDec2} for every $v \in V(G)$, the subgraph $H_v \subseteq H$ induced on $\{h \in V(H) : v \in V(G) \}$ is connected.
\end{enumerate}
We refer to the $G_h$ as its \defn{parts} and to $H$ as its \defn{decomposition graph}.
When $H$ is a cycle, we call $(H,\cG)$ a \defn{cycle-decomposition}

The \defn{adhesion-graphs} of $(H,\cG)$ are the intersections $G_h\cap G_{h'}$ for the edges $hh'\in E(H)$.
The \defn{adhesion-sets} of $(H,\cG)$ are the vertex-sets of its adhesion-graphs.
The \defn{torso} of a part $G_h$ is obtained from $G_h$ by making every adhesion-graph $G_h\cap G_{h'}$ with $hh'\in E(G)$ into a complete graph by adding all possible edges.
An edge of a torso with both ends contained in the same adhesion-set is a \defn{torso-edge}, even if it already exists in~$G$.
The \defn{interior} $\defnMath{\mathring{G}_h}$ of $G_h$ is the graph $G_h-\bigcup_{h'\in H-h}G_{h'}$.
The \defn{interior} of a torso is defined similarly.
The \defn{interior} vertices of $G_h$ are the vertices of the interior of $G_h$, and the \defn{interior} vertices of a torso are defined similarly.

Let $X$ be a graph.
We call $X$ a \defn{generalised wheel} with \defn{centre $v$} if $X$ is 3-connected with $v\in V(X)$ such that $X-v$ has a cycle-decomposition with all adhesion-sets of size 1 and into $K_2$'s and triangles.
We call $X$ a \defn{generalised double-wheel} with \defn{centre $\{u,v\}$} if $X$ is 4-connected and $u,v$ are distinct vertices of $X$ such that $X-u-v$ has a cycle-decomposition with all adhesion-sets of size 1 and into $K_2$'s and triangles.
We call $X$ a \defn{cycle of $X_1$-torsos, \dots, $X_k$-torsos} where the $X_i$ are graphs, if $X$ is 4-connected and has a cycle-decomposition with all adhesion-sets of size two such that every torso is isomorphic to some~$X_i$.
A \defn{thickened $K_{4,m}$} is obtained from a $K_{4,m}$ by turning the left side into a $K_4$.
A \defn{sprinkled $K_{4,m}$} is obtained from a $K_{4,m}$ by adding any number of edges to the left side.

\section{Proofs of \texorpdfstring{\cref{MainDecomp} and \cref{3Decomp}}{the decomposition theorems} and strengthenings thereof}
\label{sec:ProofsStrengthenings}

Instead of proving \cref{MainDecomp} and \cref{3Decomp} directly, we will first prove stronger but slightly more technical versions of them, \cref{MainDecompStrengthening} and \cref{3DecompStrengthening} below, and then immediately obtain \cref{MainDecomp} and \cref{3Decomp} from their strengthenings.
The proof of \cref{MainDecompStrengthening} uses a number of lemmas that have not yet been introduced.
Stating and proving these lemmas will be the main objectives of later sections.

\begin{theorem}[Technical strengthening of \cref{MainDecomp}]\label{MainDecompStrengthening}
    Let $G$ be a 4-connected graph.
    Let $N$ denote the set of all totally-nested tetra-separations of~$G$.
    Let $\sigma$ be a splitting star of $N$ with torso $\tau$.
    \begin{enumerate}
        \item\label{MainDecompStrengtheningQuasi5con} If $\sigma$ is not interlaced by any tetra-separation, then $\tau$ is quasi-5-connected.
    \end{enumerate}
    Assume now that $\sigma$ is interlaced by a tetra-separation $(A,B)$, which is crossed by another tetra-separation $(C,D)$ with vertex-centre~$Z$.
    Then $|Z|\in\{0,2,4\}$.
    \begin{enumerate}[resume]
        \item\label{MainDecompStrengtheningBagel} If $|Z|=0$, then $\tau$ is a cycle of triangle-torsos and 3-connected torsos on $\le 5$ vertices.
        \item\label{MainDecompStrengtheningDoubleWheel} If $|Z|=2$, then $\tau$ is a generalised double-wheel with centre $Z$.
        \item\label{MainDecompStrengtheningK4m} If $|Z|=4$, then $\tau$ is a thickened $K_{4,m}$ with $Z$ as the left side and $m\ge 0$, or the entire graph $G$ is a sprinkled $K_{4,m}$ with $Z$ as the left side and $m\ge 4$.
    \end{enumerate}
\end{theorem}

\begin{proof}
    If no tetra-separation of $G$ interlaces $\sigma$, then $\tau$ is quasi-5-connected by \cref{Quasi5conSummary}.
    So we may assume that some tetra-separation $(A,B)$ of $G$ interlaces~$\sigma$.
    Then $(A,B)$ is crossed by some tetra-separation $(C,D)$, with $Z$ denoting the vertex-centre.
    By the \allref{keylem:crossing}, we have $|Z|\in\{0,2,4\}$.

    \cref{MainDecompStrengtheningBagel}.
    Assume $|Z|=0$.
    Then $\tau$ is a cycle of triangle-torsos and 3-connected torsos on $\le 5$ vertices by \cref{bagelTorso}.

    \cref{MainDecompStrengtheningDoubleWheel}.
    Assume $|Z|=2$.
    Then $\tau$ is a generalised double-wheel with centre $Z$ by \cref{wheelTorso}.

    \cref{MainDecompStrengtheningK4m}.
    Assume $|Z|=4$.
    If $\sigma$ is non-empty, then $\tau$ is a thickened $K_{4,m}$ with $Z$ as left side and $m\ge 0$ by \cref{SprinkledK4mSummary}.
    Otherwise the entire graph $G$ is a sprinkled $K_{4,m}$ with $Z$ as left side and $m\ge 4$, also by \cref{SprinkledK4mSummary}.
\end{proof}

\begin{proof}[Proof of \cref{MainDecomp}]
    Immediate from \cref{MainDecompStrengthening}.
\end{proof}

\begin{corollary}[Technical strengthening of \cref{3Decomp}]\label{3DecompStrengthening}
    Let $G$ be a 3-connected graph.
    Let $N$ be the set of all \trisp s of~$G$ that are totally-nested (with regard to the \trisp s of~$G$).
    Let $\sigma$ be a splitting star of $N$ with torso~$\tau$.
    \begin{enumerate}
        \item\label{3DecompStrengtheningQuasi4con} If $\sigma$ is not interlaced by any \trisp , then $\tau$ is quasi-4-connected.
    \end{enumerate}
    Assume now that $\sigma$ is interlaced by a \trisp\ $(A,B)$, which is crossed by another \trisp\ $(C,D)$ with vertex-centre~$Z$.
    \begin{enumerate}[resume]
        \item\label{3DecompStrengtheningWheel} If $Z$ consists of exactly one vertex~$v$, then $\tau$ is a generalised wheel with centre~$v$.
        \item\label{3DecompStrengtheningK3m} Otherwise $Z$ has size three.
        If $\sigma\neq\emptyset$, then $\tau$ is a thickened $K_{3,m}$ with $m\ge 0$, and otherwise $G$ is a sprinkled $K_{3,m}$ with $m\ge 4$.
    \end{enumerate}
\end{corollary}

\begin{proof}
    Let $\hat G$ be the graph obtained from $G$ by adding an apex-vertex~$\alpha$.
    Every tetra-separation of $\hat G$ contains $\alpha$ in its separator.
    Hence the map $\varphi$ that sends $(A,B)$ to $(A\cup\{\alpha\},B\cup\{\alpha\})$ is a bijection between the set of \trisp s of $G$ and the set of tetra-separations of~$\hat G$.
    Two \trisp s of $G$ are nested if and only if their images under $\varphi$ are nested.
    Hence $\varphi$ restricts to a bijection between the set $N$ of totally-nested \trisp s of $G$ and the set $\hat N$ of totally-nested tetra-separations of~$\hat G$.
    Let $\hat\sigma:=\varphi(\sigma)$, which is a splitting star of $\hat N$.
    Let $\hat\tau$ denote the torso of $\hat\sigma$.
    Since $\alpha$ lies in all separators of tetra-separations of $\hat G$, no edge incident to $\alpha$ is contracted in the definition of the torso~$\hat\tau$.
    Thus $\tau=\hat\tau-\alpha$.

    \cref{3DecompStrengtheningQuasi4con}. 
    Assume that $\sigma$ is not interlaced by any \trisp .
    Then $\hat\sigma$ is not interlaced by any tetra-separation, so $\hat\tau$ is quasi-5-connected by \cref{MainDecompStrengthening}~\cref{MainDecompStrengtheningQuasi5con}.
    If $\tau$ is not quasi-4-connected, this is witnessed by some separation $(X,Y)$ of $\tau$, and then $(X\cup\{\alpha\},Y\cup\{\alpha\})$ witnesses that $\hat\tau$ is not quasi-5-connected.
    Hence $\tau$ is quasi-4-connected.

    So we may assume next that $\sigma$ is interlaced by a \trisp\ $(A,B)$, which is crossed by another \trisp\ $(C,D)$ with $Z$ denoting the centre.
    Then $\varphi(A,B)$ interlaces $\sigma$ and is crossed by $\varphi(C,D)$ with vertex-centre $\hat Z:=Z\cup\{\alpha\}$.

    \cref{3DecompStrengtheningWheel}.
    Assume that $Z$ consists of exactly one vertex~$v$.
    Then $\hat\tau$ is a generalised double-wheel with centre $\{v,\alpha\}$ by \cref{MainDecompStrengthening}~\cref{MainDecompStrengtheningDoubleWheel}.
    Hence $\hat\tau$ is a generalised wheel with centre~$v$.

    \cref{3DecompStrengtheningK3m}. The vertex-centre $Z$ cannot have sizes 0 or 2, since then $\hat Z$ would have size 1 or 3, which is excluded in \cref{MainDecompStrengthening}.
    Thus $|Z|=3$, so $|\hat Z|=4$.
    Since $\alpha\in\hat Z$, both outcomes of \cref{MainDecompStrengthening}~\cref{MainDecompStrengtheningK4m} translate to the two outcomes of~\cref{3DecompStrengtheningK3m}.
\end{proof}

\begin{proof}[Proof of \cref{3Decomp}]
    Immediate from \cref{3DecompStrengthening}.
\end{proof}

\begin{example}
    The outcome in \cref{MainDecompStrengthening} is not in general unique.
    For example, consider the graph $G$ depicted in \cref{fig:TwoWayCrossingLaugen} on the left.
    Every tetra-separation of $G$ is crossed, hence $\sigma=\emptyset$ is a splitting star of the set $N(G)$ of all totally-nested tetra-separations of~$G$.
    On the left of \cref{fig:TwoWayCrossingLaugen}, we see two tetra-separations that interlace $\sigma$ and cross with empty vertex-centre.
    From their perspective, $G$ is a cycle of $K_4$-torsos and $4$-wheel torsos, see \cref{MainDecompStrengthening}~\cref{MainDecompStrengtheningBagel}.
    On the right of \cref{fig:TwoWayCrossingLaugen}, we see two tetra-separations that interlace $\sigma$ and cross with a vertex-centre of size~2.
    From their perspective, $G$ is a generalised wheel, see \cref{MainDecompStrengthening}~\cref{MainDecompStrengtheningDoubleWheel}.

    For another example, consider the graph $G$ depicted in \cref{fig:TwoWayCrossingK4m} on the left.
    Again, $\sigma=\emptyset$ is a splitting star of~$N(G)$.
    On the left of \cref{fig:TwoWayCrossingK4m}, we see two tetra-separations that interlace $\sigma$ and cross with empty vertex-centre.
    From their perspective, $G$ is a cycle of $K_4$-torsos, see \cref{MainDecompStrengthening}~\cref{MainDecompStrengtheningBagel}.
    On the right of \cref{fig:TwoWayCrossingK4m}, we see two tetra-separations that interlace $\sigma$ and cross with a vertex-centre of size~4.
    From their perspective, $G$ is a $K_{4,4}$, see \cref{MainDecompStrengthening}~\cref{MainDecompStrengtheningK4m}.
\end{example}

\begin{figure}[ht]
    \centering
    \includegraphics[height=9\baselineskip]{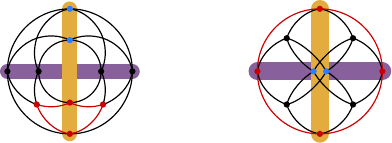}
    \caption{The drawings show the same graph, which can be seen with the help of the red cycle and the blue vertices. Two tetra-separations cross with empty vertex-centre (left) and with vertex-centre of size~2 (right).}
    \label{fig:TwoWayCrossingLaugen}
\end{figure}

\begin{figure}[ht]
    \centering
    \includegraphics[height=9\baselineskip]{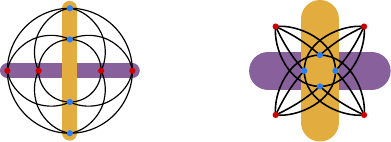}
    \caption{The drawings show the same graph. Two tetra-separations cross with empty vertex-centre (left) and with vertex-centre of size~4 (right).}
    \label{fig:TwoWayCrossingK4m}
\end{figure}

\section{How tetra-separations cross}\label{sec:crossingAnalysis}

In this section, we characterise the ways in which tetra-separations can cross.
The main result here is the \allref{keylem:crossing}.

We require the following terminology for corner diagrams from \cite[§1.3]{Tridecomp}, which is supported by \cref{fig:crossingDiagram}.
Let $(A,B)$ and $(C,D)$ be mixed-separations\pl{} of a graph~$G$.
All of the following definitions depend on $(A,B)$ and $(C,D)$ being clear from context.
The \defn{corner for $\{A,C\}$}, or the \defn{corner $AC$} for short, is the vertex $(A\sm B)\cap (C\sm D)$.
The corners for the other three combinations of one side from $(A,B)$ and one side from $(C,D)$ are defined analogously.
Two corners are \defn{adjacent} when their pair share a side, and \defn{opposite} otherwise.
An edge of $G$ is \defn{diagonal} if its endvertices lie in opposite corners.
The \defn{centre} consists of all vertices in the vertex-centre and all diagonal edges.

The \defn{link for~$C$}, or \defn{$C$-link} for short, consists of
\begin{itemize}
    \item all non-diagonal edges $f$ that have an end in the corner for a pair containing $C$ (these being the corners $AC$ or $BC$) and satisfy $f\in S(A,B)$, and
    \item all vertices in $(C\sm D)\cap S(A,B)$.
\end{itemize}
The links for the other three sides $A,B,D$ are defined analogously.
Each $L$-link is \defn{adjacent} to the two corners for the pairs containing $L$.
Two links are \defn{adjacent} if they are adjacent to the same corner, and \defn{opposite} otherwise.
An edge of $G$ is \defn{jumping} if its endvertices lie in opposite links.

Let $X \in \{A,B\}$ and $Y \in \{C,D\}$.
An edge $e$ of $G$ \defn{dangles from the $Y$-link through the $X$-link} if $e$ has an endvertex in the $Y$-link and $e$ is contained in the $X$-link.
We then call $e$ \defn{dangling} for short.

\begin{figure}[ht]
    \centering
    \includegraphics[height=8\baselineskip]{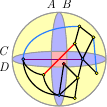}
    \caption{A crossing-diagram. Corners are yellow. Links are blue. The centre is red. The red edge is diagonal. The purple edge is jumping. The left blue edge dangles from the $A$-link through the $C$-link.}
    \label{fig:crossingDiagram}
\end{figure}

The following definition is supported by \cref{fig:PotterDef}.
The corner for $AC$ is \defn{potter} if it is empty and its adjacent links have size two and there exists a path $u_1u_2u_3u_4$ in $G$ such that $u_1u_2$ dangles from the vertex~$u_2$ in the $C$-link through the $A$-link, and $u_3u_4$ dangles from the vertex~$u_3$ in the $A$-link through the $C$-link.

\begin{figure}[ht]
    \centering
    \includegraphics[height=8\baselineskip]{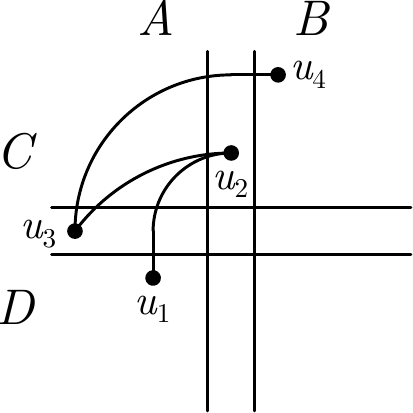}
    \caption{The corner $AC$ is potter.}
    \label{fig:PotterDef}
\end{figure}

In the remainder of this section, we prove the following:

\begin{keylemma}[Crossing Lemma]
\label{keylem:crossing}
    Let $(A,B)$ and $(C,D)$ be crossing tetra-separations in a 4-connected graph~$G$.
    Then exactly one of the following assertions holds.
    \begin{enumerate}
        \item\label{itm:crossing-0} All links are empty and the centre has size four.
        \item\label{itm:crossing-1} All links have size one and the centre has size two.
        \item\label{itm:crossing-2} All links have size two and the centre is empty.
    \end{enumerate}
    Moreover, there are no diagonal edges and no jumping edges. 
    If there is an edge that dangles from one link through another, then we are in case~\ref{itm:crossing-2} and the corner in between the two links is potter.
\end{keylemma}

A mixed-separation $(A,B)$ is an \defn{edge-cut} if $S(A,B)\se E(G)$. An edge-cut $(A,B)$ is \defn{atomic} if $|A|=1$ or $|B|=1$.
\begin{lemma}
\label{lem:trivial-tetra-separations}
    Let $(A,B)$ be a mixed-4-separation of a 4-connected graph~$G$ that satisfies the degree-condition. Then the following assertions are equivalent:
    \begin{enumerate}
        \item\label{itm:trivial-tetra-separations-1} $(A,B)$ is not a tetra-separation;
        \item\label{itm:trivial-tetra-separations-2} $(A,B)$ is an atomic cut in~$G$;
        \item\label{itm:trivial-tetra-separations-3} $|A \sm B| = 1$ or $|B \sm A| = 1$.
    \end{enumerate}
\end{lemma}
\begin{proof}
    \ref{itm:trivial-tetra-separations-2} $\Rightarrow$ \ref{itm:trivial-tetra-separations-1} is trivial.
    
    \ref{itm:trivial-tetra-separations-1} $\Rightarrow$ \ref{itm:trivial-tetra-separations-3}. If $(A,B)$ is not a tetra-separation, then there exist two edges in $S(A,B)$ that share an endvertex~$u$, say $u \in A \sm B$. Since $G$ is 4-connected and since $(A,B \cup \{u\})$ is a mixed-separation\pl{} of order at most~3, $(A,B)$ cannot be proper. Hence, $A \sm (B \cup \{u\}) = \emptyset$ and $A \sm B = \{u\}$.

    \ref{itm:trivial-tetra-separations-3} $\Rightarrow$ \ref{itm:trivial-tetra-separations-2}. If say $|A \sm B| = 1$ then $A \cap B = \emptyset$ by the degree-condition. It follows that $(A,B)$ is an atomic cut in $G$.
\end{proof}

Let $(A,B)$ and $(C,D)$ be mixed-separations\pl{} of a graph~$G$.
The \defn{corner-separator $L(A, C)$} for the corner $AC$ is the union of the $A$-link, the $C$-link, the vertex-centre, and the diagonal edges that have an end in the corner $AC$.
The corner-separators for the other three corners are defined analogously.

\begin{lemma}\emph{\cite[Lemma 1.3.4]{Tridecomp}}
\label{lem:submodularity-crossing}
    For two mixed-separations $(A, B)$ and $(C, D)$ of $G$, we have
    \[|L(A, C)| + |L(B,D)| \le |S(A, B)| + |S(C, D)|~.\]
    Moreover, if we have equality, there are no jumping edges, and every diagonal edge has its endvertices in the corners for~$AC$ and~$BD$.
\end{lemma}

\begin{lemma}
\label{lem:crossing-opposite-corners}
    Let two tetra-separations $(A,B)$ and $(C,D)$ cross in a 4-connected graph~$G$.
    Then there exist two opposite corner-separators of size at least~4.
\end{lemma}
\begin{proof}
    Otherwise, two adjacent corner-separators have size at most~3, say the corner-separator for~$AC$ and the corner-separator for~$BC$. By 4-connectivity, the corners for~$AC$ and~$BC$ are empty. Since $C \sm D \neq \emptyset$, the link~$L$ for~$C$ is nonempty. 
    The link $L$ only contains vertices since its adjacent corners are empty, so $L = (A \cap B) \sm D$. If $|L| \geq 2$ then the vertices in~$L$ have at most one neighbour in $A \sm B$ (and $B \sm A$), contradicting the degree-condition for $(A,B)$. Otherwise, if $|L| = 1$ then $(C,D)$ is not a tetra-separation by \cref{lem:trivial-tetra-separations}, a contradiction.
\end{proof}

The following lemma may be regarded as an analogue of \cite[Lemma 1.3.5]{Tridecomp}.

\begin{lemma}\label{lem:crossing-mixed-separations}
    Let $(A,B)$ and $(C,D)$ be two mixed-$k$-separations in a $k$-connected graph $G$ that cross so that two opposite corner-separators have size at least~$k$. Then either
    \begin{enumerate}
        \item\label{itm:crossing-mixed-separations-1} all links have the same size $\ell$ and the centre has size $k-2\ell$ for some $\ell \le \lfloor k/2 \rfloor$; or
        \item\label{itm:crossing-mixed-separations-2} two adjacent links have size~$i$, the other two links have size~$j$, and the centre has size $k-i-j$, where $1\le i \le \lfloor (k-1)/2 \rfloor$ and $i < j \leq k-i$.
    \end{enumerate}
\end{lemma}
\begin{proof}
    By \cref{lem:submodularity-crossing}, two opposite corner-separators have size exactly~$k$, say the corner-separators for~$AC$ and~$BD$.
    Let $a, b, c, d$ denote the sizes of the links for $A, B, C, D$, respectively. Let $x$ denote the size of the centre. 
    By \cref{lem:submodularity-crossing}, every diagonal edge has its endvertices in the corners for~$AC$ and~$BD$; and there are no jumping edges. 
    We get $a + c + x = k$ and $b + d + x = k$ and $a + b + x = k$ and $c + d + x = k$. It follows that $b=c$ and $a=d$. Without loss of generality $a=d \leq b=c$.

    Since $a+c \leq k$ and $a \leq c$ we have $a=d \in \{0,1,\ldots,\lfloor k/2 \rfloor\}$. If $a=d=0$ then the corner for~$AD$ must be nonempty, and hence $x=k$ since $G$ is $k$-connected. Then $b=c=0$ and we get outcome~\ref{itm:crossing-mixed-separations-1}.

    Otherwise $1 \leq a=d \leq \lfloor k/2 \rfloor$. If $a=b=c=d=:\ell$ then we get $x=k-a-c = k-2 \ell$ and so get outcome~\ref{itm:crossing-mixed-separations-1}. Otherwise $i := a=d < b=c =: j$. Note that $i \leq \lfloor (k-1)/2 \rfloor$. With $x=k-a-c=k-i-j$, we get outcome~\ref{itm:crossing-mixed-separations-2}.
\end{proof}

\begin{lemma}
\label{lem:links-of-crossing-tetra-separations}
    Let $(A,B)$ and $(C,D)$ be two crossing tetra-separations in a 4-connected graph~$G$. Then all links have the same size~$\ell$ and the centre has size $4-2\ell$, for some $\ell \in \{0,1,2\}$.
\end{lemma}
\begin{proof}
    Suppose that not all links have size $\ell$, for some $\ell \in \{0,1,2\}$. 
    By \cref{lem:crossing-opposite-corners} and \cref{lem:crossing-mixed-separations}, we get that two adjacent links have size one, the other two links have size at least two and the centre has size at most one. 
    Without loss of generality, the corner~$K$ of~$AC$ is the corner whose adjacent links have size one. 
    The corner-separator of $K$ has size at most~3, so~$K$ is empty since~$G$ is 4-connected. 
    Then one of the links of size one contains a vertex~$u$, say $u \in (A \cap B) \sm D$. 
    But then $u$ has at most one neighbour in $A \sm B$, contradicting the degree-condition for~$(A,B)$.
    Hence all links have the same size~$\ell$ for some $\ell\in\{0,1,2\}$, and then the centre has size $4-2\ell$ by \cref{lem:crossing-opposite-corners} and \cref{lem:crossing-mixed-separations}.
\end{proof}

\begin{observation}
\label{obs:potter-adjacent-corner}
    Let $(A,B)$ and $(C,D)$ be crossing tetra-separations in a 4-connected graph~$G$.
    If a corner is potter, then its two adjacent corners are nonempty.
    In particular, no two adjacent corners are potter. \qed
\end{observation}

\begin{lemma}
\label{lem:crossing-edges}
    Let $(A,B)$ and $(C,D)$ be two crossing tetra-separations in a 4-connected graph~$G$. 
    \begin{enumerate}
        \item\label{itm:crossing-edges-1} There are no jumping edges and no diagonal edges.
        \item\label{itm:crossing-edges-2} Let $X,Y \in \{A,B,C,D\}$. 
        If some edge dangles from the $Y$-link through the $X$-link, then the corner for $XY$ is potter.
    \end{enumerate}
\end{lemma}
\begin{proof}
    \ref{itm:crossing-edges-1} By \cref{lem:links-of-crossing-tetra-separations}, all links have the same size $\ell$ and the centre has size $4-2\ell$. 
    By \cref{lem:submodularity-crossing}, there do not exist jumping edges nor diagonal edges.

    \ref{itm:crossing-edges-2} 
    Without loss of generality, $X=A$ and $Y=C$.
    Let $e=u_1u_2$ be a dangling edge such that $u_2$ lies in the $C$-link and $e$ lies in the $A$-link.
    We first show that the corner~$K$ for~$AC$ is empty. 
    Assume for contradiction that $K \neq \emptyset$. 
    By \cref{lem:links-of-crossing-tetra-separations} and 4-connectivity, $|L(A,C)| = 4$. 
    Then $L(A,C)-e$ is a mixed-3-separator separating the vertices in the corner for~$AC$ from~$u_1$.
    This contradicts 4-connectivity by \cref{kConMixed}.

    By the degree-condition of $(A,B)$ and since $K$ is empty, the vertex $u_2$ in the $C$-link has a neighbour~$u_3$ in the $A$-link. 
    By the degree-condition of $(C,D)$ and since $K$ is empty, the vertex~$u_3$ has a neighbour~$u_4$ in $(C \sm D) \cap B$. 
    It remains to show that~$u_4$ is in the corner for~$BC$. 
    Otherwise, $u_4$ is in the $C$-link.
    By the degree-condition of $(A,B)$ and since $K$ is empty, the vertex $u_4$ has a neighbour in $A\cap D$ besides $u_3$, and therefore the $A$-link contains a third vertex or edge besides $e$ and~$u_3$.
    This contradicts \cref{lem:links-of-crossing-tetra-separations}.
\end{proof}

\begin{proof}[Proof of the \allref{keylem:crossing}]
    We combine \cref{lem:links-of-crossing-tetra-separations} and \cref{lem:crossing-edges}.
\end{proof}

\section{Reductions: a tool for finding tetra-separations}\label{sec:shifts}

This section provides a `reduction' method for obtaining tetra-separations from genuine 4-separations, and from mixed-4-separations more generally.
There are `small' mixed-4-separations that fail to yield tetra-separations via the reduction method.
We characterise these in \cref{leftrightTetra}.

\subsection{Left-reductions and right-reductions}

\begin{definition}[Left-reduction, right-reduction]
    Let $(A,B)$ be a mixed-separation\pl\ in a graph. Let $L$ be the set of vertices in $A \cap B$ with at most one neighbour in $A \sm B$, and let $R$ be the set of vertices in $A \cap B$ with at most one neighbour in $B \sm A$.
    The \defn{left-reduction} of $(A,B)$ is $(A \sm L, B)$.
    The \defn{right-reduction} of $(A,B)$ is $(A,B \sm R)$.
\end{definition}

In the following, we prove lemmas for the left-reduction.
All results have analogues for right-reductions by symmetry.

\begin{observation}\label{obs:shift-maintains-proper-sides}
    The left-reduction $(A',B')$ of a mixed-separation\pl\ $(A,B)$ satisfies $A'\sm B'=A\sm B$ and $B'\sm A'\supseteq B\sm A$.\qed
\end{observation}

\begin{lemma}
\label{lem:left-shift}
    Let $(A,B)$ be a mixed-$k$-separation\pl\ of a graph~$G$ with left-reduction $(A',B')$.
    \begin{enumerate}
        \item\label{lem:left-shift1} $(A',B')$ is a mixed-$\ell$-separation\pl\ with $\ell\le k$. We have $\ell=k$ if every vertex in $S(A,B)$ has a neighbour in $A\sm B$.
        \item\label{lem:left-shift2} Every vertex in $S(A',B')$ has $\ge 2$ neighbours in $A'\sm B'$.
    \end{enumerate}
\end{lemma}
\begin{proof}
    \cref{lem:left-shift1}.
    By definition, the left-reduction $(A',B')$ is a mixed-separation\pl{}. 
    The separator of $(A',B')$ is obtained from the separator of $(A,B)$ by replacing every vertex $v\in S(A,B)$ that has exactly one neighbour $u\in A\sm B$ with the edge $uv$ and by removing every vertex in $S(A,B)$ that has no neighbour in $A\sm B$.
    Hence the order of $(A',B')$ is at most that of $(A,B)$, with equality if every vertex in $S(A,B)$ has a neighbour in $A\sm B$.
    
    \cref{lem:left-shift2}.
    Assume for a contradiction that there is a vertex $v \in S(A',B')$ with at most one neighbour in $A' \sm B'$.
    Since $A'\sm B'=A\sm B$ by \cref{obs:shift-maintains-proper-sides}, the vertex $v$ also has at most one neighbour in $A \sm B$. 
    But then, by the definition of left-reduction, we have $v \notin A'$, contradicting $v \in S(A',B')$.
\end{proof}

\begin{lemma}\label{leftMatching}
    Let $(A,B)$ be a mixed-4-separation\pl\ in a 4-connected graph~$G$.
    Then the following assertions are equivalent:
    \begin{enumerate}
        \item\label{leftMatching1} no two edges in $S(A,B)$ share an endvertex in $A\sm B$;
        \item\label{leftMatching2} $|A\sm B|\ge 2$ or $|S(A,B)\cap E(G)|\le 1$.
    \end{enumerate} 
\end{lemma}
\begin{proof}
    \cref{leftMatching1} $\Rightarrow$ \cref{leftMatching2} is immediate.

    \cref{leftMatching2} $\Rightarrow$ \cref{leftMatching1}.
    Assume that \cref{leftMatching1} fails, so there are two edges $ab_1,ab_2\in S(A,B)$ with $a\in A\sm B$ and both $b_i\in B\sm A$.
    Let $X$ be the vertex-set that consists of $a$ together with all vertices in $S(A,B)$ and all endvertices in $B$ of edges in $S(A,B)$ other than the $ab_i$.
    Since $|A\sm B|\ge 2$ by \cref{leftMatching2}, there is $\alpha\in A\sm B$ distinct from~$a$.
    If some $b_i$ is not in $X$, we let $\beta:=b_i$.
    Otherwise both $b_i$ are in $X$, each has exactly two neighbours in $A\sm B$, and $S(A,B)$ contains only edges.
    Since $b_1$ has degree four, it has a neighbour $\beta$ in $B\sm A$ that is distinct from~$b_2$.
    In either case, $X$ is a 3-separator that separates $\alpha$ from~$\beta$, contradicting that $G$ is 4-connected.
\end{proof}

\begin{lemma}\label{leftDegreeStrongerMatching}
    Let $(A,B)$ be a mixed-4-separation\pl\ in a 4-connected graph~$G$ such that every vertex in $S(A,B)$ has $\ge 2$ neighbours in $A\sm B$.
    Then the following assertions are equivalent:
    \begin{enumerate}
        \item\label{leftDegreeStrongerMatching1} no two edges in $S(A,B)$ share an end in $A\sm B$;
        \item\label{leftDegreeStrongerMatching2} $(A,B)$ is not an atomic edge-cut with $|A|=1$;
        \item\label{leftDegreeStrongerMatching3} $|A\sm B|\ge 2$.
    \end{enumerate} 
\end{lemma}
\begin{proof}
    \cref{leftDegreeStrongerMatching1} $\Rightarrow$ \cref{leftDegreeStrongerMatching2} is immediate.

    \cref{leftDegreeStrongerMatching2} $\Rightarrow$ \cref{leftDegreeStrongerMatching3}.
    Assume $|A\sm B|\le 1$.
    Since every vertex in $S(A,B)$ has $\ge 2$ neighbours in $A\sm B$, we get that $S(A,B)$ is an edge-cut.
    Since $S(A,B)$ is nonempty, $|A|=1$ follows.

    \cref{leftDegreeStrongerMatching3} $\Rightarrow$ \cref{leftDegreeStrongerMatching1}.
    This follows from \cref{leftMatching} \cref{leftMatching2} $\Rightarrow$ \cref{leftMatching1}.
\end{proof}

\begin{lemma}\label{leftMatchingMedium}
    Let $(A,B)$ be a mixed-4-separation\pl\ in a 4-connected graph $G$
    with $A\sm B\neq\emptyset$ and
    such that no two edges in $S(A,B)$ share an end in $B\sm A$.
    Let $(A',B')$ denote the left-reduction of $(A,B)$.
    Then the following assertions are equivalent:
    \begin{enumerate}
        \item\label{leftMatchingMedium1} $(A',B')$ satisfies the matching-condition;
        \item\label{leftMatchingMedium2} $|A\sm B|\ge 2$.
    \end{enumerate} 
\end{lemma}
\begin{proof}
    By \cref{lem:left-shift}, every vertex in $S(A',B')$ has $\ge 2$ neighbours in $A'\sm B'$.

    \cref{leftMatchingMedium1} $\Rightarrow$ \cref{leftMatching2}.
    Suppose $|A\sm B|\le 1$.
    So $|A\sm B|=1$ since $A\sm B\neq\emptyset$.
    Then $|A'\sm B'|=1$ since $A\sm B=A'\sm B'$.
    Applying \cref{leftDegreeStrongerMatching} to $(A',B')$ yields that $(A',B')$ is an atomic edge-cut.

    \cref{leftMatching2} $\Rightarrow$ \cref{leftMatching1}.
    From $A\sm B= A'\sm B'$ we get $|A'\sm B'|\ge 2$.
    So no two edges in $S(A',B')$ share an end in $A'\sm B'$ by \cref{leftDegreeStrongerMatching} applied to $(A',B')$.
    So suppose for a contradiction that there are two distinct edges $a_1 b,a_2 b$ in $S(A',B')$ with $a_i\in A'\sm B'$ and $b \in B' \sm A'$.
    Then $a_i\in A'\sm B'= A\sm B$ and $b\in B\cap A$.
    Hence $b$ has $\ge 2$ neighbours in $A\sm B$, contradicting that~$b$ was replaced with edges when taking the left-reduction.
\end{proof}

\begin{lemma}\label{Righty}
    Let $G$ be a 4-connected graph.
    Let $(A,B)$ be a mixed-4-separation of~$G$.
    Let $(C,D)$ be a 4-separation of $G$ with right-reduction $(C',D')$.
    Assume $A\se C$.
    Then $(A,B)\le (C',D')$.
\end{lemma}
\begin{proof}
    We obtain $A\se C'$ from $A\se C=C'$.
    Let $X$ consist of all endvertices in $A$ of edges in $S(A,B)$.
    Then $N(B\sm A)=(A\cap B)\cup X$.
    From $A\se C$ we obtain $D\sm C\se B\sm A$.
    Hence $D\se (B\sm A)\cup N(B\sm A)=B\cup X$.
    Recall that $D'\se D$.
    Let $v\in D'$.
    It suffices to show that $v\notin X$.
    If $v\in X$, then $v$ has exactly one neighbour in $B\sm A$.
    Since $B\sm A\supseteq D\sm C$, the definition of right-reduction yields $v\notin D'$, a contradiction.
    So $v\notin X$.
\end{proof}

\subsection{Left-right-reductions and right-left-reductions}\label{sec:LeftRightShift}

\begin{definition}
    Let $(A,B)$ be a mixed-separation\pl . 
    The \defn{left-right-reduction} of $(A,B)$ is obtained from $(A,B)$ by first taking the left-reduction and then the right-reduction.
    The \defn{right-left-reduction} of $(A,B)$ is obtained from $(A,B)$ by first taking the right-reduction and then the left-reduction.
\end{definition}

\begin{observation}
\label{obs:reduction-maintains-proper-sides}
    Let $(A',B')$ be the left-right-reduction of a mixed-4-separation\pl\ $(A,B)$ in a graph~$G$. 
    Then $A' \sm B' \supseteq A \sm B$ and $B' \sm A' \supseteq B \sm A$.
\end{observation}
\begin{proof}
    We perform the two reductions one after the other, applying \cref{obs:shift-maintains-proper-sides} after each reduction.
\end{proof}

\begin{keylemma}\label{leftrightTetra}
    Let $(A,B)$ be a mixed-4-separation\pl\ in a 4-connected graph $G$.
    The left-right-reduction $(A'',B'')$ of $(A,B)$ is a tetra-separation if and only if the following three conditions are satisfied:
    \begin{enumerate}
        \item\label{leftrightTetra1} $|A\sm B|\ge 2$;
        \item\label{leftrightTetra2} the number of vertices in $S(A,B)$ with $\le 1$ neighbour in $A\sm B$ is $\ge 2- |B\sm A|$;
        \item\label{leftrightTetra3} $|B\sm A|\ge 2$ or $|S(A,B)\cap E(G)|\le 1$.
    \end{enumerate}
\end{keylemma} 
\begin{proof}
    Let $(A',B')$ denote the left-reduction of $(A,B)$.

    \casen{Backward implication.}
    Assume \cref{leftrightTetra1}--\cref{leftrightTetra3}.
    By \cref{leftrightTetra3} and \cref{leftMatching} (applied to $(B,A)$ instead of $(A,B)$), no two edges in $S(A,B)$ share an endvertex in $B\sm A$.
    By \cref{leftrightTetra1} and \cref{leftMatchingMedium}, $(A',B')$ satisfies the matching-condition.
    By \cref{lem:left-shift}, every vertex in $S(A',B')$ has $\ge 2$ neighbours in $A'\sm B'$.
    By \cref{leftrightTetra1} and since $G$ is 4-connected, every vertex in $S(A,B)$ has a neighbour in $A\sm B$.
    Hence $(A',B')$ has order four by \cref{lem:left-shift}.
    We further have $|A'\sm B'|=|A\sm B|\ge 2$ by \cref{obs:reduction-maintains-proper-sides} and we have $|B'\sm A'|\ge 2$ by~\cref{leftrightTetra2}.

    By \cref{leftMatchingMedium} (applied for the right-reduction), $(A'',B'')$ satisfies the matching-condition.
    By \cref{lem:left-shift} (applied for the right-reduction), every vertex in $S(A'',B'')$ has $\ge 2$ neighbours in $B''\sm A''$.
    Every vertex in $S(A'',B'')$ has $\ge 2$ neighbours in $A''\sm B''$ since $A''\sm B''=A'\sm B'$ by \cref{obs:reduction-maintains-proper-sides} (applied for the right-reduction).
    Hence $(A'',B'')$ satisfies the degree-condition.
    Since $B'\sm A'$ is nonempty and $G$ is 4-connected, every vertex in $S(A',B')$ has a neighbour in $B'\sm A'$, so $(A'',B'')$ has order four by \cref{lem:left-shift} (applied for the right-reduction).
    Finally, $A''\sm B''=A'\sm B'$ and $B''\sm A''\supseteq B'\sm A'$ by \cref{obs:reduction-maintains-proper-sides}, so $(A'',B'')$ is proper.

    \casen{Forward implication.}
    Suppose first that \cref{leftrightTetra3} fails.
    Then, by \cref{leftMatching} for $(B,A)$, there are two edges $a b_1, a b_2$ in $S(A,B)$ with $a\in A\sm B$ and $b_i\in B\sm A$.
    Then $a\in A''\sm B''$ and $b_i\in B''\sm A''$ by \cref{obs:reduction-maintains-proper-sides}.
    Hence $(A'',B'')$ fails the matching-condition.

    Suppose next that \cref{leftrightTetra3} holds, but \cref{leftrightTetra1} or \cref{leftrightTetra2} fails.
    By \cref{leftrightTetra3} and \cref{leftMatching} for $(B,A)$, no two edges in $S(A',B')$ share an endvertex in~$B'\sm A'$.
    If \cref{leftrightTetra1} fails, we have $|A\sm B|\le 1$, so \cref{leftMatchingMedium} gives that $(A',B')$ violates the matching-condition, and hence $(A'',B'')$ also violates the matching-condition.
    Next, assume that \cref{leftrightTetra1} holds, but \cref{leftrightTetra2} fails.
    As $|A\sm B|\ge 2$, \cref{leftMatchingMedium} gives that $(A',B')$ satisfies the matching-condition.
    The failure of \cref{leftrightTetra2} yields $|B'\sm A'|\le 1$.
    Hence applying \cref{leftMatchingMedium} to $(B',A')$ (instead of $(A,B)$) yields that $(A'',B'')$ violates the matching-condition.
\end{proof}

\begin{lemma}\label{qkcNoMixedKsep}
    Let $k\in\N$.
    Let $G$ be a quasi-$(k+1)$-connected graph with $\ge k+4$ vertices.
    Then $G$ has no mixed-$k$-separation $(A,B)$ with $|A\sm B|\ge 2$ and $|B\sm A|\ge 2$.
\end{lemma}
\begin{proof}
    Assume that $G$ has a mixed-$k$-separation $(A,B)$ with $|A\sm B|\ge 2$ and $|B\sm A|\ge 2$.
    Let $X:=A\sm B$ and $Y:=B\sm A$.
    Let $F$ denote the set of edges in $S(A,B)$.
    We have
    \[
        k+4\le |G| =|X|+|Y|+k-|F|
    \]
    and therefore
    \[
        |F| \le (\,|X|-2)+(\,|Y|-2).
    \]
    Both $X$ and $Y$ have size $\ge 2$ by assumption.
    Hence both terms $|X|-2$ and $|Y|-2$ are $\ge 0$.
    We may therefore bipartition $F$ into classes $F_X$ and $F_Y$ of $\le |X|-2$ and $\le |Y|-2$ edges, respectively.
    We obtain $X'$ from $X$ by adding for each edge in $F_X$ its endvertex in $A$ to $X$, and we obtain $Y'$ similarly.
    Then $(X',Y')$ is a $k$-separation of $G$.
    We have $|X'\sm Y'|\ge |X|-|F_X|\ge 2$, and similarly $|Y'\sm X'|\ge 2$.
    Thus $G$ is not quasi-$(k+1)$-connected.
\end{proof}

\begin{proof}[Proof of \cref{quasi5conVtetra}.]
    \cref{quasi5conVtetra1}.
    If $G$ is not quasi-5-connected, it has a 4-separation $(A,B)$ with $|A\sm B|\ge 2$ and $|B\sm A|\ge 2$.
    Hence the right-left-reduction of $(A,B)$ is a tetra-separation of $G$ by \cref{leftrightTetra}.

    \cref{quasi5conVtetra2}.
    Let $(A,B)$ be a tetra-separation of $G$.
    Then $|A\sm B|\ge 2$ and $|B\sm A|\ge 2$ by \cref{lem:trivial-tetra-separations}.
    Hence $G$ is not quasi-5-connected by \cref{qkcNoMixedKsep}.
\end{proof}

\section{Understanding nestedness through connectivity}\label{sec:ExternalConnectivity}

The purpose of this section is to characterise the property of total-nestedness for tetra-separations $(A,B)$ in terms of connectivity-properties of their separators~$S(A,B)$; see \cref{keylem:nestedness-external-connectivity}.
This translates the problem of proving total-nestedness of $(A,B)$ to a problem about constructing various families of paths that link up the vertices and edges in the separator $S(A,B)$ in different ways.

For an element $x\in V(G)\cup E(G)$ we let \defnMath{$\hat x$} denote the singleton $\{x\}$ if $x$ is a vertex, and the set of endvertices of $x$ if $x$ is an edge.
For a set $X\se V(G)\cup E(G)$ we write $\defnMath{\hat X}:=\bigcup_{x\in X}\hat x$.

A bipartition of a set is \defn{balanced} if its two classes have the same size.

\begin{lemma}\label{kConMixed}
    Let $G$ be a $k$-connected graph.
    Then $G$ has no mixed-$\ell$-separation for any $\ell<k$.
\end{lemma}
\begin{proof}
    Assume for a contradiction that $(A,B)$ is a mixed-$\ell$-separation of $G$ with $\ell<k$.
    Then $S(A,B)$ contains an edge $ab$ with $a\in A\sm B$ and $b\in B\sm A$.
    We claim that $A\sm (B\cup\{a\})$ or $B\sm (A\cup\{b\})$ is non-empty.
    Indeed, otherwise $V(G)=(A\cap B)\cup\{a,b\}$, so $|G|=\ell+1\le k$, contradicting that $k$-connected graphs such as $G$ have $>k$ vertices by definition.
    So suppose that there is a vertex $v\in A\sm (B\cup\{a\})$.
    Let $X$ be obtained from $S(A,B)$ by keeping all vertices, replacing $ab$ with $a$,  and replacing each other edge with its endvertex in $B\sm A$.
    Then $X$ is an $\ell$-separator of $G$ that separates $v$ from~$b$, contradicting that $G$ is $k$-connected.
\end{proof}

Let $X,Y$ be two vertex-sets in a graph~$G$.
A mixed-separation $(A,B)$ of $G$ \defn{strongly separates} $X$ and $Y$ if $X\se A\sm B$ and $Y\se B\sm A$.
Note that if $(A,B)$ strongly separates $X$ and $Y$, then $X$ and $Y$ are disjoint.
The following lemma is folklore; we include a proof for convenience.

\begin{lemma}[Menger for independent paths]\label{Menger}
Let $X,Y$ be two disjoint vertex-sets in a graph~$G$. Then the following two numbers are equal:
\begin{enumerate}
    \item\label{MengerPaths} the maximum number of independent $X$--$Y$ paths in~$G$;
    \item\label{MengerSep} the minimum order of a mixed-separation in $G$ that strongly separates $X$ and~$Y$.
\end{enumerate}
Moreover, if $G$ is $k$-connected, then \ref{MengerPaths}, \ref{MengerSep}${}\ge k$.
\end{lemma}

\begin{proof}
    As usual, \ref{MengerPaths} $\le$ \ref{MengerSep}, so it remains to show \ref{MengerPaths} $\ge$ \ref{MengerSep}.
    Since every $X$--$Y$ edge forms a path in every maximum set of independent $X$--$Y$ paths in $G$, we may assume without loss of generality that $G$ contains no $X$--$Y$ edges.
    Let $X':=N_G(X)$ and $Y':=N_G(Y)$.
    Note that $X'\cup Y'$ is disjoint from $X\cup Y$.
    Let $G':=G-(X\cup Y)$.

    Now let $\cP$ be a maximum-size set of independent $X$--$Y$ paths in~$G$.
    Then $\cP$ determines a set $\cP'$ of $|\cP|$ pairwise disjoint $X'$--$Y'$ paths in~$G'$.
    Hence $G'$ contains an $X'$--$Y'$ separator $S'$ of size $|\cP|$, by Menger's theorem applied in $G'$ to $X',Y'$ and $\cP'$.
    Then $S'$ is the separator of a mixed-separation of $G$ that strongly separates $X$ and~$Y$, and which has order $|S'|=|\cP|$.

    The `moreover'-part follows from \cref{kConMixed}.
\end{proof}

\subsection{Half-connected}

A mixed-separation $(A,B)$ of a graph $G$ is \defn{half-connected} if $G[A\sm B]$ or $G[B\sm A]$ is connected~\cite{Tridecomp}.

\begin{lemma}
\label{lem:nestedness-half-connected}
    Let $(A,B)$ be a tetra-separation of a 4-connected graph~$G$. Then the following assertions are equivalent:
    \begin{enumerate}
        \item\label{itm:nestedness-half-connected-1} $(A,B)$ is half-connected;
        \item\label{itm:nestedness-half-connected-2} no tetra-separation of $G$ crosses $(A,B)$ so that all links are empty.
    \end{enumerate}
\end{lemma}
\begin{proof}
    $\neg$ \ref{itm:nestedness-half-connected-2}~$\Rightarrow$~$\neg$ \ref{itm:nestedness-half-connected-1}. 
    If a tetra-separation of $G$ crosses $(A,B)$ so that all links are empty, then all corners are nonempty. Hence, neither $G[A \sm B]$ nor $G[B \sm A]$ is connected.

    $\neg$ \ref{itm:nestedness-half-connected-1}~$\Rightarrow$~$\neg$ \ref{itm:nestedness-half-connected-2}. Assume that $(A,B)$ is not half-connected; that is, $G[A \sm B]$ and $G[B \sm A]$ are not connected.
    Then $S(A,B)$ consists of vertices only.
    Let $C_A$ be a component of $G[A \sm B]$ and let $C_B$ be a component of $G[B \sm A]$. 
    Let $C := V(C_A)\cup V(C_B) \cup S(A,B)$ and $D := V(G-C_A-C_B)$. 
    Then $(C,D)$ is a genuine 4-separation with $S(C,D)=S(A,B)$ that crosses $(A,B)$ with all links empty.
    To see that $(C,D)$ is a tetra-separation, note first that both $G[C \sm D]$ and $G[D \sm C]$ are disconnected. 
    Since $G$ is 4-connected, every vertex in $S(A,B)$ has a neighbour in every component of $G-S(A,B)$.
    So $(C,D)$ satisfies the degree-condition, and it trivially satisfies the matching-condition.
\end{proof}

\subsection{3-linked}

\begin{definition}[3-linked]
\label{dfn:3-linked}
Let $X=\{x_1,x_2,x_3,x_4\}$ be the separator of a tetra-separation $(A,B)$ of a graph~$G$.
We say that $\{x_1,x_2\}$ is \defn{3-linked around~$X$} if at least one of the following holds:
\begin{defenum}
    \item\label{itm:3-linked-1} $x_3$ or $x_4$ is an edge;
    \item\label{itm:3-linked-2} $x_1$ or $x_2$ is a vertex and there is an $\hat x_1$--$\hat x_2$ edge;
    \item\label{itm:3-linked-3} there are three independent $\hat x_1$--$\hat x_2$ paths in $G-x_3-x_4$.
\end{defenum}
Similarly, we define that $\{x_i,x_j\}$ is \defn{3-linked around~$X$} for other choices of distinct indices $i,j\in [4]$.
We say that $(A,B)$ is \defn{3-linked} if every pair of distinct elements in $S(A,B)$ is 3-linked around~$X$.
\end{definition}

\begin{lemma}
\label{lem:nestedness-3-linked}
    Let $X=\{x_1,x_2,x_3,x_4\}$ be the separator of a tetra-separation $(A,B)$ in a 4-connected graph~$G$. 
    The following assertions are equivalent:
    \begin{enumerate}
        \item\label{itm:nestedness-3-linked-1} $\{x_1,x_2\}$ is 3-linked around~$X$;
        \item\label{itm:nestedness-3-linked-2} no tetra-separation of $G$ crosses $(A,B)$ so that $x_1$ and $x_2$ are in opposite links while the centre has size~2.
    \end{enumerate}
\end{lemma}
\begin{proof}
    $\neg$ \ref{itm:nestedness-3-linked-2} $\Rightarrow$ $\neg$ \ref{itm:nestedness-3-linked-1}. 
    Assume that a tetra-separation $(C,D)$ crosses $(A,B)$ so that~$x_1$ and~$x_2$ are in opposite links while the centre has size~2. 
    By the \allref{keylem:crossing}, all links have size~1 and there are no jumping edges, no diagonal edges and no dangling edges.
    So,~$x_3$ and~$x_4$ are both vertices, which shows that~\ref{itm:3-linked-1} does not hold. 
    Since there are no dangling edges and no jumping edges, if there exists an $\hat x_1$--$\hat x_2$ edge then both~$x_1$ and~$x_2$ are edges. 
    This shows that~\ref{itm:3-linked-2} does not hold.
    Since there are no diagonal edges, no jumping edges and no dangling edges, every $\hat x_1$--$\hat x_2$ path that avoids~$x_3$ and~$x_4$ has an inner vertex or an edge in the link for~$A$ or the link for~$B$. Hence, there are at most two independent $\hat x_1$--$\hat x_2$ paths in $G-x_3-x_4$. This shows that~\ref{itm:3-linked-3} does not hold. So, $\{x_1,x_2\}$ is not 3-linked around~$X$.

    $\neg$ \ref{itm:nestedness-3-linked-1} $\Rightarrow$ $\neg$ \ref{itm:nestedness-3-linked-2}. 
    Assume that $\{x_1,x_2\}$ is not 3-linked around~$X$. 
    Then $x_3$ and $x_4$ are both vertices and there are at most two independent $\hat x_1$--$\hat x_2$ paths in $G-x_3-x_4$.
    By \allref{Menger}, there is a mixed-2-separation~$(C,D)$ of~$G-x_3-x_4$ that strongly separates~$\hat x_1$ from~$\hat x_2$, say with~$\hat x_1 \subseteq C \sm D$ and $\hat x_2 \subseteq D \sm C$.
    Let~$(C',D'):=(C \cup \{x_3,x_4\}, D \cup \{x_3,x_4\})$, which is a mixed-4-separation.
    Observe that~$(C',D')$ crosses~$(A,B)$ since the $C'$-link and the $D'$-link are non-empty.

    We claim that both~$C' \sm D'$ and~$D' \sm C'$ have size~$\geq 2$.
    Assume otherwise, say~$x_1$ is a vertex and it is the only vertex in~$C' \sm D'$.
    By the degree-condition for~$x_1 \in S(A,B)$, it has at least four neighbours in $V(G) \sm (A \cap B)$. This, however, contradicts that~$(C,D)$ has order two.

    Let~$(C'',D'')$ be the left-right-reduction of~$(C',D')$, which is a tetra-separation by \cref{leftrightTetra}.
    First, note that~$\hat x_1 \subseteq C' \sm D' \subseteq C'' \sm D''$ and $\hat x_2 \subseteq D' \sm C' \subseteq D'' \sm C''$ by \cref{obs:reduction-maintains-proper-sides}.
    So,~$x_1$ and~$x_2$ are in the $C''$-link and the $D''$-link, respectively.
    In particular, $(C'',D'')$ crosses $(A,B)$.

    We claim that $x_3,x_4 \in S(C'',D'')$.
    Assume not, say~$x_3 \in C'' \sm D''$. 
    Let $e\in S(C'',D'')$ be the edge that $x_3$ sends to $D''\sm C''$.
    By the \allref{keylem:crossing},~$e$ is not a jumping edge. 
    So there is an~$X \in \{A,B\}$ such that~$e$ dangles from the $C''$-link through the $X$-link, say $X=A$.
    By the \allref{keylem:crossing}, the corner~$AC''$ is potter. 
    In particular,~$x_1$ is an edge that dangles from the $A$-link through the $C''$-link. 
    Then, however, $\hat x_1 \nsubseteq C'' \sm D''$, a contradiction.
    Hence,~$(C'',D'')$ crosses~$(A,B)$ so that~$x_1$ and~$x_2$ are in opposite links while the centre has size~2.
\end{proof}

\subsection{Potter-linked}

Let $(A,B)$ be a mixed-separation of a graph~$G$.
Let $\{x_1,x_2\}\se S(A,B)$ consist of an edge and a vertex, say $x_1$ is the edge and $x_2$ is the vertex.
Write $x_1=ab$ with $a\in A\sm B$ and $b\in B\sm A$.
In the context of $(A,B)$, we say that $\{x_1,x_2\}$ is \defn{$A$-weird} if $x_2$ has exactly one neighbour in $A \sm \{a\}$.
If $(A,B)$ is a tetra-separation, this is equivalent to saying that $x_2$ has exactly two neighbours in $A$ and $a$ is one of them.

\begin{example}
\label{obs:potter-implies-weird}
    Let $(A,B)$ and $(C,D)$ be crossing tetra-separations in a graph~$G$.
    Assume that the corner for $AC$ is potter, witnessed by an edge $u_1 u_2$ that dangles from the vertex $u_2$ in the $C$-link through the $A$-link, and by an edge $u_3 u_4$ that dangles from the vertex $u_3$ in the $A$-link through the $C$-link; see \cref{fig:PotterDef}.
    Then $\{u_2,u_3u_4\}$ is $A$-weird and $\{u_3,u_1u_2\}$ is $C$-weird.
\end{example}

Let $X,Y\se V(G)\cup E(G)$.
An \defn{$X$--$Y$ edge} means an edge $xy$ with $x\in X$ and $y\in Y$.

\begin{definition}[0-potter-linked]
\label{dfn:0-potter-linked}
Let~$\{\pi_1,\pi_2\}$ be a balanced bipartition of the separator of a tetra-separation~$(A,B)$ in a graph~$G$.
We say that $\{\pi_1,\pi_2\}$ is \defn{0-potter-linked} with respect to~$(A,B)$ if at least one of the following assertions holds:
\begin{defenum}
    \item\label{itm:0-potter-linked-1} there is a $\pi_1$--$\hat \pi_2$ edge or a $\hat \pi_1$--$\pi_2$ edge;
    \item\label{itm:0-potter-linked-2} there are five independent $\hat \pi_1$--$\hat \pi_2$ paths in~$G$.
\end{defenum}
Note that this is symmetric in~$\pi_1$ and $\pi_2$.
A~tetra-separation $(A,B)$ is \defn{0-potter-linked} if $\{\pi_1,\pi_2\}$ is 0-potter-linked for every balanced bipartition $\{\pi_1,\pi_2\}$ of $S(A,B)$. 
\end{definition}

\begin{lemma}
\label{lem:nestedness-0-potter-linked}
    Let $(A,B)$ be a tetra-separation of a 4-connected graph~$G$.
    Let $\{\pi_1,\pi_2\}$ be a balanced bipartition of $S(A,B)$. 
    Then the following assertions are equivalent:
    \begin{enumerate}
        \item\label{itm:nestedness-0-potter-linked-1} $\{\pi_1,\pi_2\}$ is 0-potter-linked;
        \item\label{itm:nestedness-0-potter-linked-2} no tetra-separation of $G$ crosses $(A,B)$ so that $\pi_1$ and $\pi_2$ form opposite links and no corner is potter.
    \end{enumerate}
\end{lemma}
\begin{proof}
    $\neg$ \ref{itm:nestedness-0-potter-linked-2} $\Rightarrow$ $\neg$ \ref{itm:nestedness-0-potter-linked-1}.
    Assume that a tetra-separation $(C,D)$ of $G$ crosses $(A,B)$ so that $\pi_1$ and $\pi_2$ form opposite links and no corner is potter.
    Hence by the \allref{keylem:crossing}, no edge is dangling, jumping, or diagonal.
    It follows that if there exists a $\hat \pi_1$--$\hat \pi_2$ edge~$uv$, then both~$u$ and~$v$ are endvertices of edges in~$S(A,B)$. This shows that~\ref{itm:0-potter-linked-1} fails.

    It remains to show that there are at most four $\hat\pi_1$--$\hat\pi_2$ paths in~$G$.
    As no edge is dangling, jumping, or diagonal, every $\hat\pi_1$--$\hat\pi_2$ path in $G$ must have an internal vertex or an edge in the $A$-link or in the $B$-link. 
    The sum of the sizes of the $A$-link and the $B$-link is at most~4 since the size of $S(C,D)$ is at most~4.
    Hence \ref{itm:0-potter-linked-2} fails.

    $\neg$ \ref{itm:nestedness-0-potter-linked-1} $\Rightarrow$ $\neg$ \ref{itm:nestedness-0-potter-linked-2}. 
    Assume that $\{\pi_1,\pi_2\}$ is not 0-potter-linked.
    Then there are at most four independent $\hat\pi_1$--$\hat\pi_2$ paths in $G$.
    By \allref{Menger}, there is a mixed-4-separation $(C,D)$ of $G$ that strongly separates $\hat\pi_1$ from $\hat\pi_2$, say with~$\hat \pi_1 \subseteq C \sm D$ and~$\hat \pi_2 \subseteq D \sm C$.
    In particular, $|C \sm D| \geq 2$ and $|D \sm C| \geq 2$.
    Observe that $(C,D)$ crosses $(A,B)$ since the $C$-link and the $D$-link are non-empty. 
    Let $(C',D')$ be the left-right-reduction of $(C,D)$, which is a tetra-separation by \cref{leftrightTetra}.
    
    We claim that~$(C',D')$ crosses~$(A,B)$ so that~$\pi_1$ and~$\pi_2$ form opposite links and no corner is potter.
    Since $\hat \pi_1 \subseteq C \sm D \subseteq C' \sm D'$ and $\hat \pi_2 \subseteq D \sm C \subseteq D' \sm C'$ by \cref{obs:reduction-maintains-proper-sides},~$\pi_1$ forms the $C'$-link while~$\pi_2$ forms the $D'$-link.
    In particular,~$(C',D')$ crosses~$(A,B)$.
    Assume for contradiction that some corner is potter, say the corner $AC'$. 
    Then~$\pi_1$ is $A$-weird and hence contains an edge~$e$.
    Let~$a$ be the endvertex of~$e$ in~$A \sm B$. So,~$a \in \hat \pi_1 \subseteq C' \sm D'$. Therefore,~$a$ is contained in the corner~$AC'$.
    In particular, the corner $AC'$ is not empty, and so it is not potter, a contradiction.
\end{proof}

The following definition is supported by \cref{fig:1potterLinkedDef}.

\begin{definition}[1-potter-linked]
\label{dfn:1-potter-linked}
    Let~$\{\pi_1,\pi_2\}$ be a balanced bipartition of the separator of a tetra-separation~$(A,B)$ of a graph~$G$.
    Let~$(X,Y)=(A,B)$ or~$(X,Y)=(B,A)$.
    We say that \defn{$\pi_1$ is 1-potter-linked to~$\pi_2$ around~$X$} with respect to $(A,B)$ if at least one of the following holds:
    \begin{defenum}
        \item\label{itm:1-potter-linked-1} $\pi_1$ is not $X$-weird;
        \item\label{itm:1-potter-linked-2} there is a $\pi_1$--$(\hat \pi_2 \cap Y)$ edge or a $(\hat \pi_1 \cap Y)$--$\pi_2$ edge;
        \item\label{itm:1-potter-linked-3} there are three independent $\hat \pi_1$--$\hat \pi_2$ paths in~$G[Y]$.
    \end{defenum}
    Note that this is \emph{not} symmetric in~$\pi_1,\pi_2$.
    We say that~$\{\pi_1,\pi_2\}$ is \defn{1-potter-linked} with respect to~$(A,B)$ if for both $X \in \{A,B\}$,~$\pi_1$ is 1-potter-linked to~$\pi_2$ around~$X$ and~$\pi_2$ is 1-potter-linked to~$\pi_1$ around~$X$ with respect to~$(A,B)$.
    A~tetra-separation $(A,B)$ is \defn{1-potter-linked} if $\{\pi_1,\pi_2\}$ is 1-potter-linked for every balanced bipartition $\{\pi_1,\pi_2\}$ of $S(A,B)$.
\end{definition}

\begin{figure}[ht]
    \centering
    \includegraphics[height=8\baselineskip]{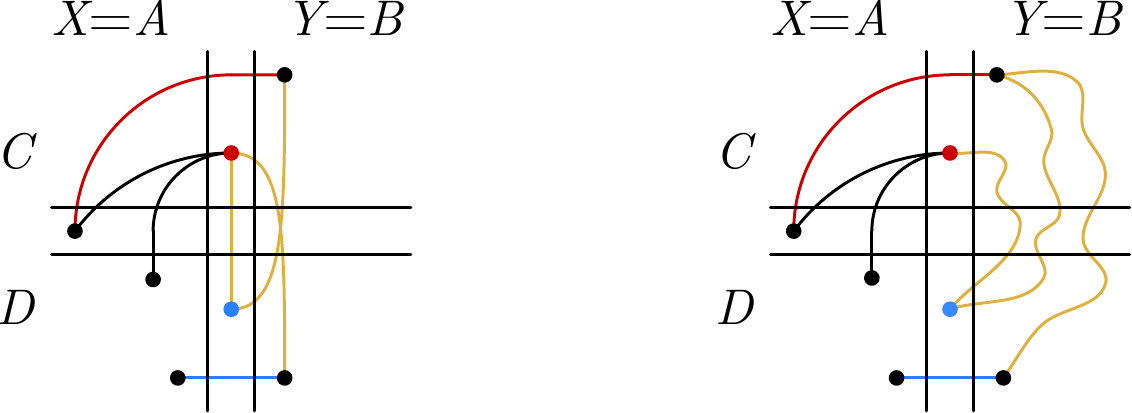}
    \caption{$\pi_1$ is red and $X$-weird. $\pi_2$ is blue.
    Both pictures include how a mixed-separation $(C,D)$ might exploit the weirdness of $\pi_1$ to cross $(A,B)$.
    Left: the yellow edges satisfy \ref{itm:1-potter-linked-2}. Right: the yellow paths satisfy \ref{itm:1-potter-linked-3}.}
    \label{fig:1potterLinkedDef}
\end{figure}

\begin{lemma}\label{lem:nestedness-1-potter-linked}
    Let $(A,B)$ be a tetra-separation of a 4-connected graph~$G$ and let $\{\pi_1,\pi_2\}$ be a balanced bipartition of $S(A,B)$. 
    Then the following assertions are equivalent:
    \begin{enumerate}
        \item\label{itm:nestedness-1-potter-linked-1} $\{\pi_1,\pi_2\}$ is 1-potter-linked with respect to $(A,B)$;
        \item\label{itm:nestedness-1-potter-linked-2} no tetra-separation of $G$ crosses $(A,B)$ so that $\pi_1$ and $\pi_2$ form opposite links and exactly one corner is potter.
    \end{enumerate}
\end{lemma}
\begin{proof}
    $\neg$ \ref{itm:nestedness-1-potter-linked-2} $\Rightarrow$ $\neg$ \ref{itm:nestedness-1-potter-linked-1}. 
    Assume that a tetra-separation $(C,D)$ crosses $(A,B)$ so that $\pi_1$ and $\pi_2$ form opposite links, say $\pi_1$ is the $C$-link and~$\pi_2$ is the $D$-link, and so that exactly one corner is potter, say the $AC$-corner.
    We show that~$\pi_1$ is \emph{not} 1-potter-linked to~$\pi_2$ with respect to~$(A,B)$.
    For this, we show that for~$(X,Y)=(A,B)$ in \cref{dfn:1-potter-linked}, all of \ref{itm:1-potter-linked-1}--\ref{itm:1-potter-linked-3} fail.
    
    Observe that $\pi_1$ is $A$-weird by \cref{obs:potter-implies-weird}, so \ref{itm:1-potter-linked-1} fails.
    By assumption, the corner $AC$ is the only corner that is potter.
    By the \allref{keylem:crossing}, every link has size two while the centre is empty.
    Moreover, there do not exist jumping or diagonal edges, and the $B$-link contains no dangling edge.
    It follows that there is no $\pi_1$--$(\hat \pi_2 \cap B)$ edge and no $(\hat \pi_1 \cap B)$--$\pi_2$ edge, so \ref{itm:1-potter-linked-2} fails.
    Moreover, every $\hat\pi_1$--$\hat\pi_2$ path in $G[B]$ has an internal vertex or an edge in the $B$-link. 
    Thus, there are at most two independent $\hat\pi_1$--$\hat\pi_2$ paths in $G[B]$, so \ref{itm:1-potter-linked-3} fails.

    $\neg$ \ref{itm:nestedness-1-potter-linked-1} $\Rightarrow$ $\neg$ \ref{itm:nestedness-1-potter-linked-2}.
    Assume that $\{\pi_1,\pi_2\}$ is not 1-potter-linked with respect to $(A,B)$. 
    Without loss of generality,~$\pi_1$ is not 1-potter-linked to~$\pi_2$ and \ref{itm:1-potter-linked-1}--\ref{itm:1-potter-linked-3} fail for~$(X,Y)=(A,B)$.
    By $\neg$~\ref{itm:1-potter-linked-1}, $\pi_1$ is $A$-weird.
    Write $\pi_1=:\{v_1,e_1\}$ where $v_1$ is a vertex and $e_1=a_1 b_1$ is an edge with $a_1\in A\sm B$ and $b_1\in B\sm A$.
    Let $a'_1$ be the unique neighbour of $v_1$ in $A$ besides~$a_1$.
    Note that $a'_1$ lies in $A\sm B$ by the degree-condition for~$v_1$.

    By $\neg$~\ref{itm:1-potter-linked-3}, there are at most two independent $\hat\pi_1$--$\hat\pi_2$ paths in $G[B]$. 
    By \allref{Menger}, there is a mixed-separation $(C_1,D_1)$ of $G[B]$ of order $\le 2$ that strongly separates $\hat\pi_1$ from $\hat\pi_2$, say with $\hat\pi_1 \cap B \subseteq C_1 \sm D_1$ and $\hat\pi_2 \cap B \subseteq D_1\sm C_1$.
    In particular, $|C_1\sm D_1|\ge 2$ and $|D_1\sm C_1|\ge 2$.
    See \cref{fig:PotterMenger} for the next definition.
    We define $(C_2,D_2):=(C_1 \cup a_1, (D_1 \cup A) \sm \{v_1\})$. 
    We have $S(C_2,D_2) = S(C_1,D_1) \cup \{a'_1 v_1,a_1\}$ since $v_1$ does not have a neighbour in~$A$ besides~$a_1'$. 
    By 4-connectivity of~$G$ and \cref{kConMixed}, $(C_2,D_2)$ has order exactly~4.
    Let $(C,D)$ be the left-right-reduction of $(C_2,D_2)$.
    Since $C_1\sm D_1\se C_2\sm D_2$ we have $|C_2\sm D_2|\ge 2$, and similarly $|D_2\sm C_2|\ge 2$.
    Hence $(C,D)$ is a tetra-separation by \cref{leftrightTetra}.

    Consider the crossing-diagram for $(A,B)$ and $(C,D)$.
    We claim that the $C$-link includes~$\pi_1$.
    The $C$-link clearly contains~$v_1$.
    The vertex $a_1\in S(C_2,D_2)$ has two neighbours $v_1,b_1$ in $C_2\sm D_2$, so $a_1\in S(C,D)$.
    Hence $e_1\notin S(C,D)$.
    Since the end $b_1$ of $e_1$ lies in $C_1\sm D_1\se C\sm D$ by \cref{obs:reduction-maintains-proper-sides}, the edge $e_1$ lies in the $C$-link.
    Next, we claim that the $D$-link includes~$\pi_2$.
    Every vertex in $\pi_2$ clearly lies in the $D$-link.
    Now let $f$ be an edge in $\pi_2$.
    The endvertex of $f$ in $B\sm A$ lies in $\hat \pi_2\cap B\se D_1\sm C_1\se D\sm C$ by \cref{obs:reduction-maintains-proper-sides}.
    The endvertex of $f$ in $A\sm B$ must be distinct from $a_1$ by the matching-condition, and hence is contained in $D_2\sm C_2\se D\sm C$ by \cref{obs:reduction-maintains-proper-sides}.
    Therefore, both ends of $f$ lie in $D\sm C$ while $f\in S(A,B)$, so $f$ lies in the $D$-link.

    Since the $C$-link includes $\pi_1$ and the $D$-link includes $\pi_2$, the two tetra-separations $(A,B)$ and $(C,D)$ cross, and the inclusions are equalities by the \allref{keylem:crossing}.
    It remains to show that the $AC$-corner is the only corner that is potter.
    The edge~$v_1 a_1'$ dangles from the $C$-link through the $A$-link, so the corner~$AC$ is potter by the \allref{keylem:crossing}.
    By \cref{obs:potter-adjacent-corner}, the corners for~$AD$ and~$BC$ are not potter. 
    Assume for contradiction that the $BD$-corner is potter.
    Then~$\pi_2$ is $B$-weird and hence contains an edge~$f$.
    But then the endvertex of $f$ in $B\sm A$ is contained in $D\sm C$ as shown above.
    In particular, the corner for $BD$ is not empty, contradicting that $BD$ is potter.
\end{proof}

\begin{figure}[ht]
    \centering
    \includegraphics[height=8\baselineskip]{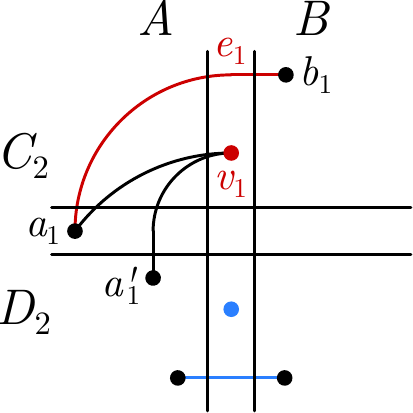}
    \caption{The situation in the second half of the proof of \cref{lem:nestedness-1-potter-linked}. $\pi_1$ is red and $\pi_2$ is blue.}
    \label{fig:PotterMenger}
\end{figure}

\begin{definition}[2-potter-linked]
\label{dfn:2-potter-linked}
    Let~$\{\pi_1,\pi_2\}$ be a balanced bipartition of the separator of a tetra-separation~$(A,B)$ of a graph~$G$.
    We say that $\{\pi_1,\pi_2\}$ is \defn{2-potter-linked} with respect to $(A,B)$ if for both~$(X,Y)=(A,B)$ and~$(X,Y)=(B,A)$, at least one of the following holds:
    \begin{defenum}
        \item\label{itm:2-potter-linked-1} $\pi_1$ is not $X$-weird;
        \item\label{itm:2-potter-linked-2} $\pi_2$ is not $Y$-weird.
    \end{defenum}
    A~tetra-separation $(A,B)$ is \defn{2-potter-linked} if $\{\pi_1,\pi_2\}$ is 2-potter-linked for every balanced bipartition $\{\pi_1,\pi_2\}$ of $S(A,B)$.
\end{definition}

\begin{observation}
\label{obs:weird-no-edges-in-separator}
    Let~$(A,B)$ be a tetra-separation of a graph~$G$.
    Let $\{\pi_1,\pi_2\}$ be a balanced bipartition of $S(A,B)$.
    Assume that~$\pi_1$ is $A$-weird.
    Then there is no $\pi_1$--$\pi_2$ edge.
\end{observation}
\begin{proof}
    Assume that there is a $\pi_1$--$\pi_2$ edge~$uv$ with $u \in \pi_1$ and $v \in \pi_2$.
    Let~$u'$ be the endvertex in $A\sm B$ of the edge in~$\pi_1$.
    Then~$v$ is the only neighbour of~$u$ in~$A \sm \{u'\}$ by definition of $A$-weird.
    Hence,~$u \in S(A,B)$ has at most one neighbour in~$A \sm B$, contradicting the degree-condition.
\end{proof}

\begin{lemma}\label{lem:nestedness-2-potter-linked}
    Let $(A,B)$ be a tetra-separation of a 4-connected graph~$G$ and let $\{\pi_1,\pi_2\}$ be a balanced bipartition of $S(A,B)$. 
    Then the following assertions are equivalent:
    \begin{enumerate}
        \item\label{itm:nestedness-2-potter-linked-1} $\{\pi_1,\pi_2\}$ is 2-potter-linked with respect to $(A,B)$;
        \item\label{itm:nestedness-2-potter-linked-2} no tetra-separation of $G$ crosses $(A,B)$ so that $\pi_1$ and $\pi_2$ form opposite links and exactly two corners are potter.
    \end{enumerate}
\end{lemma}
\begin{proof}
    $\neg$~\ref{itm:nestedness-2-potter-linked-2} $\Rightarrow$ $\neg$~\ref{itm:nestedness-2-potter-linked-1}. 
    Assume that a tetra-separation $(C,D)$ crosses $(A,B)$ so that $\pi_1$ and $\pi_2$ form opposite links, say $\pi_1$ is the $C$-link and~$\pi_2$ is the $D$-link, and so that exactly two corners are potter.
    By \cref{obs:potter-adjacent-corner}, no two adjacent corners are potter. 
    So we can assume without loss of generality that the $AC$-corner and the $BD$-corner are potter while the other two corners are not potter.
    Then~$\pi_1$ is $A$-weird and~$\pi_2$ is $B$-weird by \cref{obs:potter-implies-weird}. So,~\ref{itm:2-potter-linked-1} and~\ref{itm:2-potter-linked-2} fail for $(X,Y)=(A,B)$.

    $\neg$~\ref{itm:nestedness-2-potter-linked-1} $\Rightarrow$ $\neg$~\ref{itm:nestedness-2-potter-linked-2}.
    Assume that $\{\pi_1,\pi_2\}$ is not 2-potter-linked with respect to $(A,B)$.
    Say~\ref{itm:2-potter-linked-1} and~\ref{itm:2-potter-linked-2} fail for~$(X,Y)=(A,B)$; that is to say, $\pi_1$ is $A$-weird and~$\pi_2$ is $B$-weird.
    For $i \in \{1,2\}$, write $\pi_i=:\{v_i,e_i\}$ where~$v_i$ is a vertex and $e_i=a_i b_i$ is an edge with $a_i\in A\sm B$ and $b_i\in B\sm A$; see \cref{fig:PotterWiggle}.
    Let~$a'_1$ be the unique neighbour of~$v_1$ in~$A \sm \{a_1\}$, and let~$b_2'$ be the unique neighbour of~$v_2$ in~$B \sm \{b_1\}$.
    By the degree-condition, we have $a'_1\in A\sm B$ and $b'_1\in B\sm A$.

    Let $C:=B \sm \{v_2\}) \cup \{a_1\}$ and $D:=(A\sm \{v_1\})\cup \{b_2\}$.
    Since $\pi_1$ is $A$-weird, $\pi_2$ is $B$-weird, and $v_1 v_2\notin E(G)$ by \cref{obs:weird-no-edges-in-separator}, we have
    $S(C,D)=\{a'_1 v_1,a_1,b_2,v_2 b'_2\}$.
    Note that~$(C,D)$ crosses~$(A,B)$ since~$v_1$ is in the $C$-link and~$v_2$ is in the $D$-link.
    The vertex $a_1$ has two neighbours in $C \sm D$, namely~$v_1$ and~$b_1$.
    Since~$a_1$ has degree at least four and it is not adjacent to~$b_2$, the only other vertex in $S(C,D)$, it follows that~$a_1$ also has at least two neighbours in~$D \sm C$.
    Similarly,~$b_2$ also has two neighbours in both~$C \sm D$ and~$D \sm C$. 
    Hence $(C,D)$ satisfies the degree-condition.
    It satisfies the matching-condition since $v_1\neq v_2$ means that $a'_1 v_1$ and $v_2 b'_2$ do not share endvertices. 
    Hence,~$(C,D)$ is a tetra-separation that crosses $(A,B)$ with $\pi_1$ and $\pi_2$ as opposite links.
    The corners $AC$ and $BD$ are potter by construction.
    By \cref{obs:potter-adjacent-corner}, no other corner is potter.
\end{proof}

\begin{figure}[ht]
    \centering
    \includegraphics[height=8\baselineskip]{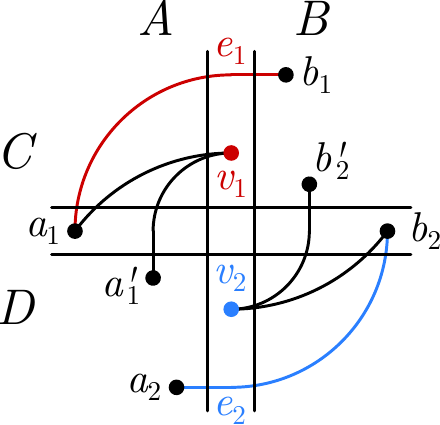}
    \caption{The situation in the second half of the proof of \cref{lem:nestedness-2-potter-linked}. $\pi_1$ is red and $\pi_2$ is blue.}
    \label{fig:PotterWiggle}
\end{figure}

\begin{corollary}
\label{cor:nestedness-potter-linked}
    Let $(A,B)$ be a tetra-separation of a 4-connected graph~$G$ and let $\{\pi_1,\pi_2\}$ be a balanced bipartition of $S(A,B)$. 
    Then the following assertions are equivalent:
    \begin{enumerate}
        \item\label{itm:nestedness-potter-linked-1} $\{\pi_1,\pi_2\}$ is $h$-potter-linked with respect to $(A,B)$ for every $h \in \{0,1,2\}$;
        \item\label{itm:nestedness-potter-linked-2} no tetra-separation of $G$ crosses $(A,B)$ so that $\pi_1$ and $\pi_2$ form opposite links.
    \end{enumerate}
\end{corollary}
\begin{proof}
    By \cref{obs:potter-adjacent-corner}, every crossing pair of tetra-separations has at most two corners that are potter.
    So the corollary immediately follows from \cref{lem:nestedness-0-potter-linked}, \cref{lem:nestedness-1-potter-linked} and \cref{lem:nestedness-2-potter-linked}.
\end{proof}

\noindent\textbf{Summary.}

\begin{definition}
\label{dfn:externally-connected}
    A tetra-separation $(A,B)$ is \defn{externally 5-connected} if it is half-connected, 3-linked and $h$-potter-linked for every $h \in \{0,1,2\}$.
\end{definition}

\begin{keylemma}
\label{keylem:nestedness-external-connectivity}
    For every tetra-separation $(A,B)$ of a 4-connected graph~$G$, the following assertions are equivalent:
    \begin{enumerate}
        \item $(A,B)$ is totally-nested;
        \item $(A,B)$ is externally 5-connected.
    \end{enumerate}
\end{keylemma}

\begin{proof}
    We combine \cref{lem:nestedness-half-connected}, \cref{lem:nestedness-3-linked} and \cref{cor:nestedness-potter-linked} with the \allref{keylem:crossing}.
\end{proof}

\section{Thickened \texorpdfstring{$K_{4,m}$}{K4m}'s}
\label{sec:K4m}

This section deals with the thickened $K_{4,m}$ outcome of \cref{MainDecomp}; see \cref{keylem:K4m}.

A \defn{sprinkled $K_{4,m}$} is a graph obtained from a $K_{4,m}$ with a side $X$ of size four by possibly adding edges between vertices in~$X$.
We refer to $X$ as a \defn{sprinkled side}.

Let $Z\se V(G)$ and $K$ a component of $G-Z$.
When $Z$ and $K$ are clear from context, we write $(U'_K,W'_K):=(V(K)\cup Z,V(G-K))$.
Let $Z$ be a set of four vertices in a 4-connected graph $G$ such that $G-Z$ has $\ge 4$ components.
For each component $K$ of $G-Z$, let $(U_K,W_K)$ denote the left-right-reduction of $(U'_K,W'_K)$.

\begin{lemma}\label{K4mMonotoneShift}
    Let $Z$ be a set of four vertices in a 4-connected graph~$G$, such that $G-Z$ has $\ge 4$ components.
    Let $K$ be a component of $G-Z$.
    Then $(U_K,W_K)$ equals the left-reduction of $(U'_K,W'_K)$.
    In particular, $(U_K,W_K)\le (U'_K,W'_K)$ and $W_K=W'_K$.
\end{lemma}
\begin{proof}
    Let $(A,B)$ denote the left-reduction of $(U'_K,W'_K)$.
    Since every vertex in $S(U'_K,W'_K)$ has $\ge 2$ neighbours in $W'_K\sm U'_K$, and since $W'_K\sm U'_K\se B\sm A$ by \cref{obs:shift-maintains-proper-sides}, it follows that every vertex in $S(A,B)$ has $\ge 2$ neighbours in $B\sm A$.
    Hence $(A,B)=(U_K,W_K)$ and $W'_K=B=W_K$.
\end{proof}

\begin{lemma}\label{K4mTotallyNestedChar}
    Let $Z$ be a set of four vertices in a 4-connected graph $G$ such that $G-Z$ has $\ge 4$ components.
    Let $K$ be a component of $G-Z$.
    Then the following assertions are equivalent:
    \begin{enumerate}
        \item\label{K4mTotallyNestedChar1} $K$ has $\ge 2$ vertices;
        \item\label{K4mTotallyNestedChar2} the left-right-reduction of $(U'_K,W'_K)$ is a totally-nested tetra-separation.
    \end{enumerate}
\end{lemma}

\begin{proof}
    $\neg$~\cref{K4mTotallyNestedChar1} $\Rightarrow$ $\neg$~\cref{K4mTotallyNestedChar2}.
    If $K$ has only one vertex, then $|U_K\sm W_K|=1$, so $(U'_K,W'_K)$ fails to be a tetra-separation by \cref{leftrightTetra}.

    \cref{K4mTotallyNestedChar1} $\Rightarrow$ \cref{K4mTotallyNestedChar2}.
    We abbreviate $(U,W):=(U_K,W_K)$ and $(U',W'):=(U'_K,W'_K)$.
    Since $K$ has $\ge 2$ vertices, $(U,W)$ is a tetra-separation by \cref{leftrightTetra}.
    To show that $(U,W)$ is totally-nested, by \cref{keylem:nestedness-external-connectivity} it suffices to show that $(U,W)$ is externally 5-connected.

    \casen{Half-connected:} 
    We have $G[U\sm W]=K$ by \cref{K4mMonotoneShift}.

    To check the remaining aspects of external 5-connectivity,
    for each $z\in Z$ we denote by $z^*$ the element of $S(U,W)$ that equal or incident to~$z$.
    By \cref{K4mMonotoneShift}, $z^*$ is either equal to $z$ or $z^*$ is an edge between $K$ and~$z$.
    Denote by $z^U$ the vertex in $\hat z\cap U$.
    We can use the three components of $G-Z$ besides~$K$ to find paths:
    \begin{enumerate}[label={($\dagger$)}]
        \item\label{lem:K4m-three-paths} There are three independent $z_1$--$z_2$ paths in $G$ with all internal vertices outside~$K$, for every two distinct vertices $z_1,z_2\in Z$.
    \end{enumerate}
    
    \casen{3-linked:} follows from \cref{lem:K4m-three-paths}.

    \casen{0-potter-linked:}
    Let $\{\pi_1,\pi_2\}$ be a balanced bipartition of~$S(U,W)$. 
    We find three independent $\hat\pi_1$--$\hat\pi_2$ paths through $W\sm U$ by~\cref{lem:K4m-three-paths}.
    We claim that there are two independent $\hat\pi_1$--$\hat\pi_2$ paths through~$K$.
    Let $\{\pi'_1,\pi'_2\}$ denote the bipartition of~$Z$ that corresponds to $\{\pi_1,\pi_2\}$.
    For both~$i$, let $\pi_i^K$ denote the set of neighbours of $\pi'_i$ in~$K$.
    Then $|\pi_i^K|\ge 2$, since otherwise $\pi_i^K\cup \pi'_{3-i}$ is a 3-separator of~$G$ as $|K|\ge 2$, contradicting 4-connectivity.
    So it suffices to find two disjoint $\pi_1^K$--$\pi_2^K$ paths in~$K$.
    If these are missing, then by Menger's theorem $\pi_1^K$ and $\pi_2^K$ are separated by a single vertex~$v$, and then $\pi'_1\cup\{v\}$ is a 3-separator of~$G$, contradicting 4-connectivity.

    \casen{1-potter-linked:} Let $\{\pi_1,\pi_2\}$ be a balanced bipartition of~$S(U,W)$. 
    Assume that~$\pi_1$ is weird, say.
    Then~$\pi_1$ is $U$-weird since each vertex in~$Z$ has $\ge 3$ neighbours in~$W \sm U$. 
    By \cref{lem:K4m-three-paths}, there exist three independent $\hat \pi_1$--$\hat \pi_2$ paths in~$G[W]$.
    Hence $\{\pi_1,\pi_2\}$ is 1-potter-linked with respect to~$(U,W)$.

    \casen{2-potter-linked:} Let $\{\pi_1,\pi_2\}$ be a balanced bipartition of~$S(U,W)$.
    No~$\pi_i$ is $W$-weird, since each vertex in~$Z$ has $\ge 3$ neighbours in~$W \sm U$. So,~\ref{itm:2-potter-linked-1} or \ref{itm:2-potter-linked-2} always holds.
\end{proof}

We write $\defnMath{\sigma(Z)}$ for the star $\{\,(U_K,W_K):K$ is a component of $G-Z$ with $|K|\ge 2\,\}$.

\begin{corollary}\label{ZgivesStarInN}
    Let $Z$ be a set of four vertices in a 4-connected graph $G$ such that $G-Z$ has $\ge 4$ components.
    Then $\sigma(Z)$ is a star of totally-nested tetra-separations.
\end{corollary}
\begin{proof}
    By \cref{K4mTotallyNestedChar}, every element in~$\sigma(Z)$ is a totally-nested  tetra-separation. 
    Let~$K \neq L$ be components of $G-Z$, each of size~$\geq 2$.
    By \cref{K4mMonotoneShift}, $(U_K,W_K) \leq (U_K',W_K')$ and $(U_L,W_L) \leq (U_L',W_L')$. With $(U_K',W_K') \leq (W_L',U_L')$, it follows that $(U_K,W_K) \leq (W_L,U_L)$.
    So,~$\sigma(Z)$ is a star.
\end{proof}

\begin{lemma}\label{squeezingTetra}
    Let $(A,B)$ be a 4-separation in a 4-connected graph~$G$ such that $|A\sm B|\ge 2$ and every vertex in $S(A,B)$ has $\ge 2$ neighbours in $B\sm A$.
    Let $(A'',B'')$ be the left-right-reduction of $(A,B)$.
    Then no tetra-separation $(C,D)$ of $G$ satisfies $(A'',B'')<(C,D)\le (A,B)$.
\end{lemma}

\begin{proof}
    Suppose for a contradiction that there is a tetra-separation $(C,D)$ with $(A'',B'')<(C,D)\le (A,B)$.
    By \cref{leftrightTetra}, $(A'',B'')$ is a tetra-separation.
    By \cref{K4mMonotoneShift}, we have $B=B''$.
    So both inclusions $B''\supseteq D\supseteq B$ are actually equalities.
    We therefore have $A''\subsetneq C\se A$.
    Let $v\in C\sm A''$.
    Then $v$ is a vertex in $S(A,B)$ with at most one neighbour in $A\sm B$.
    But since $D=B$ we have $C\sm D\se A\sm B$, so $v$ also has at most one neighbour in $C\sm D$, contradicting that $(C,D)$ is a tetra-separation.
\end{proof}

\begin{lemma}
\label{lem:not-half-conntected-is-tetra}
    Let~$G$ be a 4-connected graph.
    Let~$(A,B)$ be a mixed-4-separation in~$G$ that is not half-connected.
    Then~$(A,B)$ is a tetra-separation.
\end{lemma}
\begin{proof}
    By assumption, both $G[A\sm B]$ and $G[B\sm A]$ have two components.
    In particular, $|A \sm B|\ge 2$ and $|B \sm A| \geq 2$.
    Then by \cref{leftMatching},~$(A,B)$ satisfies the matching-condition.
    Let~$v$ be a vertex in~$S(A,B)$.
    By 4-connectivity of~$G$ and by \cref{kConMixed}, $v$~has a neighbour in every component of $G-S(A,B)$.
    In particular, $v$~satisfies the degree-condition.
    Hence,~$(A,B)$ is a tetra-separation.
\end{proof}

\begin{lemma}\label{K4mInterlacingVsComponent}
    Let~$Z$ be a set of four vertices in a 4-connected graph~$G$ such that $G-Z$ has~$\ge 4$ components.
    Let~$(E,F)$ be a totally-nested tetra-separation of~$G$.
    Then there exists a component~$K$ of $G-Z$ with $|K| \geq 2$ such that
    $(E,F)\le (U'_K,W'_K)$ or $(F,E)\le (U'_K,W'_K)$.
\end{lemma}
\begin{proof}
    We denote by $\cK$ the set of components of $G-Z$.
    First, we show that there is a component~$K \in \cK$ such that either $E \sm Z \subseteq V(K)$ or $F \sm Z \subseteq V(K)$.
    Assume otherwise.
    Then we can bipartition~$\cK$ into two sets~$\cK_1$ and~$\cK_2$ such that for both $i \in \{1,2\}$, $|\cK_i| \geq 2$ and $\bigcup_{K \in \cK_i} K$ intersects both~$E$ and~$F$.
    We define $A_i := Z \cup \bigcup_{K \in \cK_i} V(K)$.
    Then~$(A_1,A_2)$ is a tetra-separation by \cref{lem:not-half-conntected-is-tetra}.
    Now $E \nsubseteq A_i$ and $F \nsubseteq A_i$ for both $i \in \{1,2\}$ implies that $(A_1,A_2)$ crosses~$(E,F)$, contradicting that~$(E,F)$ is totally-nested.

    Hence we have~$E \subseteq V(K) \cup Z$, say, for a component~$K \in \cK$.
    Since every vertex in $Z$ has $\ge 2$ neighbours outside $V(K)\cup Z$, it follows that $(U'_K,W'_K)$ equals its right-reduction.
    Hence $(E,F)\le (U'_K,W'_K)$ by \cref{Righty}.
    Since $(E,F)$ is a tetra-separation, we have $|E\sm F|\ge 2$ by \cref{leftDegreeStrongerMatching}.
    Therefore, $|K|\ge 2$.
\end{proof}

\begin{keylemma}\label{SprinkledK4mSummary}
    Let $G$ be a 4-connected graph, and let $N$ denote the set of totally-nested tetra-separations of~$G$.
    Assume that two tetra-separations $(A,B)$ and $(C,D)$ of $G$ cross so that the vertex-centre $Z$ has size four.
    Let $m$ denote the number of trivial components of $G-Z$.
    Then:
    \begin{enumerate}
        \item\label{SprinkledK4mSummary1} $\sigma(Z)$ is a splitting star of $N$.
        \item\label{SprinkledK4mSummary2} $\sigma(Z)$ is interlaced by both $(A,B)$ and $(C,D)$.
        \item\label{SprinkledK4mSummary4} If $\sigma(Z)\neq\emptyset$, then $G-Z$ has a component with $\ge 2$ vertices, and the torso of $\sigma(Z)$ is a thickened $K_{4,m}$ where $Z$ is the left side.
        \item\label{SprinkledK4mSummary5} If $\sigma(Z)=\emptyset$, then every component of $G-Z$ is trivial, and $G=\tau$ is a sprinkled $K_{4,m}$ with $Z$ as left side and $m\ge 4$.
    \end{enumerate}
\end{keylemma}
\begin{proof}
    \cref{SprinkledK4mSummary1}.
    By \cref{ZgivesStarInN}, $\sigma(Z)$ is a star in~$N$.
    To show that $\sigma(Z)$ is splitting, assume for a contradiction that it is interlaced by a totally-nested tetra-separation $(E,F)$.
    By \cref{K4mInterlacingVsComponent}, there is a component $K$ of $G-Z$ with $|K| \geq 2$ and $(E,F)\le (U'_K,W'_K)$, say.
    Since $K$ has at least two vertices, we have $(U_K,W_K)\in\sigma(Z)$.
    Note that $(U_K,W_K)\le (U'_K,W'_K)$ by \cref{K4mMonotoneShift}, which implies $(U_K,W_K)\not < (F,E)$.
    By \cref{squeezingTetra}, we cannot have $(U_K,W_K)<(E,F)$.
    Since $(U_K,W_K)\in N$, we therefore have $(E,F)\le (U_K,W_K)$ or $(F,E)\le (U_K,W_K)$, contradicting that $(E,F)$ interlaces~$\sigma(Z)$.

    \cref{SprinkledK4mSummary2}. We denote by $\sigma'(Z)$ the star $\{\,(U'_K,W'_K):K$ is a component of $G-Z$ with $|K|\geq 2\,\}$. Both $(A,B)$ and $(C,D)$ interlace $\sigma'(Z)$, hence they also interlace $\sigma(Z)$ by \cref{K4mMonotoneShift}.

    \cref{SprinkledK4mSummary4} and \cref{SprinkledK4mSummary5} are straightforward.
\end{proof}

A graph $G$ is \defn{4-angry} if it is 4-connected and every tetra-separation of $G$ is crossed by a tetra-separation.

\begin{keylemma}
\label{keylem:K4m}
    Let $G$ be a 4-angry graph. Then the following assertions are equivalent:
    \begin{enumerate}
        \item\label{itm:keylem-K4m-1} two tetra-separations of~$G$ cross with all links empty;
        \item\label{itm:keylem-K4m-2} $G$ is a sprinkled $K_{4,m}$ with $m \geq 4$.
    \end{enumerate}
\end{keylemma}

\begin{proof}
    \ref{itm:keylem-K4m-2} $\Rightarrow$ \ref{itm:keylem-K4m-1}.
    Let $X$ be a sprinkled side of~$G$.
    Let $\{P_1,\ldots,P_4\}$ be a partition of the other side $V(G-X)$.
    Let $(A,B)$ be the mixed-4-separation $(P_1\cup P_2\cup X,X\cup P_3\cup P_4)$ and let $(C,D)$ be the mixed-4-separation $(P_1\cup P_3\cup X,X\cup P_2\cup P_4)$.
    Then $(A,B)$ and $(C,D)$ are tetra-separations with separators equal to~$X$, and they cross.

    \ref{itm:keylem-K4m-1} $\Rightarrow$ \ref{itm:keylem-K4m-2} follows from \cref{SprinkledK4mSummary} \cref{SprinkledK4mSummary4} and \cref{SprinkledK4mSummary5}.
\end{proof}

\section{Generalised double-wheels}\label{sec:DoubleWheel}

To deal with the generalised double-wheel outcome of \cref{MainDecomp}, we first introduce a particular cycle-decomposition in \cref{sec:WheelBlockBagel} for analysing the wheel-structure obtained from two tetra-separations that cross with vertex-centre of size two.
\cref{sec:WheelNpathsLemmata} introduces machinery that exploits the structure of this cycle-decomposition to find independent paths.
\cref{sec:WheelTotallyNested} employs this machinery to show that all non-trivial bags of the cycle-decomposition give rise to totally-nested tetra-separations (\cref{cor:double-wheel-totally-nested}).
\cref{sec:WheelSplitting} shows that these totally-nested tetra-separations form a splitting star as in \cref{MainDecomp} and deduces that its torso is a generalised double-wheel (\cref{wheelTorso}).
Finally, \cref{sec:WheelAngry} characterises the special case where additionally all tetra-separations of the graph are crossed (\cref{keylem:double-wheel-angry}).

\subsection{Block-bagel}\label{sec:WheelBlockBagel}

Let $G$ be a 2-connected graph.
A \defn{Block-bagel} of $G$ is a cycle-decomposition $\cO = (O,\cG)$ of~$G$ with $\cG=(G_t:t \in V(O))$ such that
\begin{enumerate}
    \item every bag $G_t$ of $\cO$ is either 2-connected or a $K_2$, and
    \item all adhesion sets have size one.
\end{enumerate}
The vertex in the adhesion-set of an edge $e$ of $O$ is the \defn{adhesion-vertex} of~$e$.
When $\cO$ as introduced as $(O,\cG)$, we tacitly assume that $\cG=(G_t:t\in V(O))$.

\begin{observation}
\label{obs:block-bagel-disjoint-adhesion}
    Let~$G$ be a 2-connected graph with a block-bagel $\cO=(O,\cG)$.
    Then every two edges of $O$ have distinct adhesion-vertices.
\end{observation}
\begin{proof}
    Otherwise, two edges~$rs$ and~$st$ have the same adhesion-vertex~$v$. But then $v$ is a cutvertex of~$G$ that separates $V(G_s - v)$ from $V(G-G_s)$.
\end{proof}

Assume now that $G$ is a 4-connected graph with two tetra-separations $(A,B)$ and $(C,D)$ that cross with all links of size one.
Then the centre $Z'$ consists of two vertices by the \allref{keylem:crossing}.
We will construct a block-bagel of $G-Z'$.
The \defn{$AC$-shrimp} is obtained from $G[A\cap C]-Z'$ by adding any edges in the $A$-link or the $C$-link (including the endvertices of these edges).
Note that the $AC$-shrimp is connected because $G-Z'$ is 2-connected.
The \defn{$AC$-block-decomposition} means the block-decomposition of the $AC$-shrimp.
Note that the $AC$-block-decomposition is a path-decomposition because $G$ is 4-connected.
Its decomposition-path is called the \defn{$AC$-block-path}.
We obtain similarly definitions for the other three corners.

\begin{lemma}\label{wheelLinksVsCorners}
    Let $G$ be a 4-connected graph with two tetra-separations $(A,B)$ and $(C,D)$ that cross with all links of size one.
    Let $K$ be one of the four corners, and consider an adjacent $L$-link~$\Lambda$.
    \begin{enumerate}
        \item\label{wheelLinksVsCorners1} If $\Lambda$ consists of a vertex, then this vertex is contained in a unique bag $V_t$ of the $K$-block-decomposition, and $G[V_t]$ is 2-connected.
        \item\label{wheelLinksVsCorners2} If $\Lambda$ consists of an edge $e=uv$, then there is a unique bag $V_t$ of the $K$-block-decomposition with $V_t=\{u,v\}$ and $G[V_t]=K_2$. Moreover, $u$ and $v$ lie in the two corners adjacent to the $L$-link.
    \end{enumerate}
\end{lemma}
\begin{proof}
    \cref{wheelLinksVsCorners1}. The bag $V_t$ is unique since $G$ is 4-connected. It follows from the degree-condition that $G[V_t]$ must be 2-connected.

    \cref{wheelLinksVsCorners2}. The first part is immediate. For the `moreover'-part we use that $e$ is not dangling by the \allref{keylem:crossing}.
\end{proof}

\begin{corollary}\label{wheelCornersNonempty}
    Let $G$ be a 4-connected graph with two tetra-separations $(A,B)$ and $(C,D)$ that cross with all links of size one.
    Then all four corners are non-empty.\qed
\end{corollary}

The bag $V_t$ and node $t$ in \cref{wheelLinksVsCorners} are said to \defn{represent} the $L$-link in the $K$-block-decomposition and in the $K$-block-path.

For each corner $K$ let $\cT^K=(T^K,\cV^K)$ denote the $K$-block-decomposition.
We obtain $O$ from the disjoint union $\bigsqcup_K T^K$, where $K$ ranges over the four corners, by doing the following for each link $L$:
for both corners $K$ and $K'$ adjacent to $L$ we consider the nodes $t$ and $t'$ that represent the $L$-link in $T^K$ and $T^{K'}$, respectively, and
\begin{itemize}
    \item we join the nodes $t$ and $t'$ by an edge $\defnMath{o(L)}$ in $O$ if their bags $G[V_t]$ and $G[V_{t'}]$ are 2-connected;
    \item we otherwise identify the nodes $t$ and $t'$ to a single node $\defnMath{o(L)}$ of~$O$.
\end{itemize}
We say that $o(L)$ \defn{represents} the $L$-link in~$O$.
For each node $t$ of $O$ that is not of the form $o(L)$, we define $G_t:=G[V^K_t]$ where $K$ is the unique corner with $t\in T^K$.
For each node $t$ of $O$ with $t=o(L)$ for some link~$L$,
we let $G_t$ be the $K_2$-subgraph of $G$ from \cref{wheelLinksVsCorners}~\cref{wheelLinksVsCorners2} that represents the $L$-link in the $K$-block-decompositions for the two corners $K$ adjacent to~$L$.
The pair $(O,(G_t)_{t\in V(O)})$ is the \defn{block-bagel induced by $(A,B)$ and $(C,D)$.}
It really is a block-bagel of $G-Z'$.

The following setting is supported by \cref{fig:BlockBagelBag}.

\begin{setting}[Block-bagel setting]
\label{set:block-bagel}
    Let $(A,B)$ and~$(C,D)$ be two tetra-separations of a 4-connected graph~$G$ that cross with all links of size one.
    Let~$Z'$ be the centre, which consists of two vertices $z'_1,z'_2$.
    Let~$\cO = (O, (G_t)_{t \in V(O)})$ be the Block-bagel induced by~$(A,B)$ and~$(C,D)$.
    Write~$O = t(0),t(1),\ldots ,t(m),t(0)$ where indices are treated as elements of $\Z_{m+1}$.
    For every node $t \in V(O)$, we let~$U_t' := V(G_t)\cup Z'$ and $W_t' := Z' \cup \bigcup_{s \in V(O -t)} V(G_s)$, and we let $(U_t,W_t)$ be the right-left-reduction of~$(U_t',W_t')$.
\end{setting}

Note that $(U_t',W_t')$ is a mixed-4-separation\pl{}.

\begin{figure}[ht]
    \centering
    \includegraphics[height=9\baselineskip]{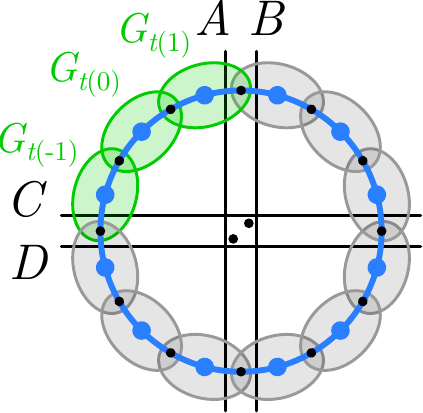}\hfill%
    \includegraphics[height=9\baselineskip]{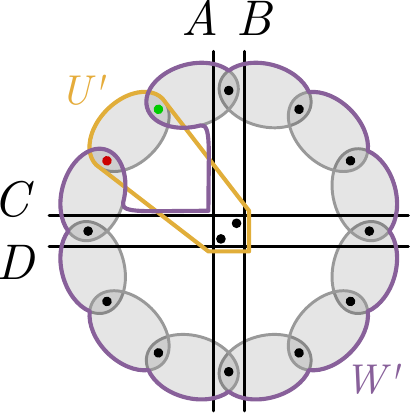}\hfill%
    \includegraphics[height=9\baselineskip]{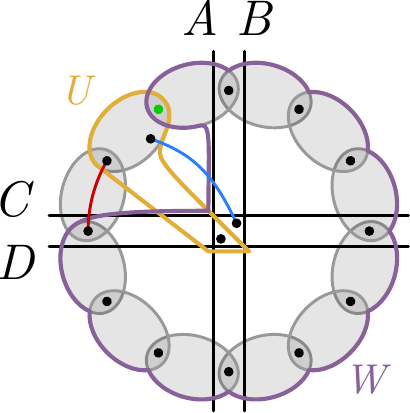}
    \caption{The \allref{set:block-bagel}. The vertices $z'_1$ and $z'_2$ are in the centre.
    Left: the arrangement of the bags.
    Center: The 4-separation\pl\ $(U',W')$. 
    The vertex $x'$ is red, the vertex $y'$ is green. 
    Right: The right-left-reduction $(U,W)$ of $(U',W')$. 
    Here $x$ is the red edge and $y$ is the green vertex.
    One of the $z_i$ is a vertex, the other is the blue edge.}
    \label{fig:BlockBagelBag}
\end{figure}

\subsection{N-Paths lemmata}\label{sec:WheelNpathsLemmata}

\begin{definition}[Helpful]
    Let~$\cO=(O,\cG)$ be a block-bagel of a graph~$G$.
    An element $x \in V(O) \cup E(O)$ is \defn{helpful} if
    \begin{itemize}
        \item $x$ is an edge with $K_2$-bags at both ends, or
        \item $x$ is a node with a 2-connected bag.
    \end{itemize}
    We define $\defnMath{h(\cO)}$ to be the number of helpful elements of the decomposition-cycle~$O$.

    For every helpful $x$ there is a unique corner $K$ such that $x$ stems from the $K$-block-path.
    We then say that $x$ \defn{belongs} to the corner~$K$.
\end{definition}

\begin{lemma}[Helpful Lemma]
\label{lem:helpful-lemma}
    Assume the \allref{set:block-bagel}.
    Then for every corner there is a helpful element that belongs to that corner. 
    In particular,~$h(\cO) \geq 4$. 
\end{lemma}
\begin{proof}
    Let $K$ be an arbitrary corner.
    If the $K$-block-decomposition has a 2-connected bag, then we are done. 
    Otherwise, the $K$-shrimp is a path.
    The corner $K$ contains a vertex $v$ by \cref{wheelCornersNonempty}.
    Then~$v$ is the adhesion-vertex of a helpful edge that belongs to the corner~$K$.
\end{proof}

\begin{lemma}
\label{lem:path-from-helpful-edge}
    Assume the~\allref{set:block-bagel}.
    The adhesion-vertex of a helpful edge of $O$ sends edges to both vertices $z'_1,z'_2$ in the centre.
\end{lemma}
\begin{proof}
    Otherwise the adhesion-vertex has degree $\le 3$, contradicting 4-connectivity.
\end{proof}

The following two lemmata are supported by \cref{fig:path-from-helpful-vertex}.

\begin{lemma}
    \label{lem:path-from-helpful-vertex}
    Assume the~\allref{set:block-bagel}.
    Assume that $t$ is a helpful node of~$O$.
    Then each vertex $z'_i$ in the centre has a neighbour $w_i$ in the interior of the bag~$G_t$.
    Moreover, if $|G_t|\ge 4$, then we can find the $w_i$ so that $w_1\neq w_2$.
\end{lemma}
\begin{proof}   
    Since $t$ is helpful, its bag $G_t$ is 2-connected.
    Hence $G_t$ has non-empty interior.
    Each $z'_i$ in the centre has a neighbour $w_i$ in this interior, since otherwise $z'_{3-i}$ together with the two adhesion-vertices of $G_t$ forms a 3-separator of $G$, contradicting 4-connectivity.

    Assume now that $G_{t_1}$ contains at least four vertices.
    If $w_1=w_2$ is the only possibility, then $w_1$ together with the two adhesion-vertices of $G_t$ forms a 3-separator, contradicting 4-connectivity.
\end{proof}

\begin{figure}[ht]
    \centering
    \includegraphics[height=8\baselineskip]{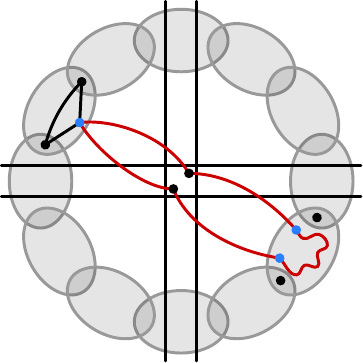}
    \caption{The neighbours $w_i$ found by \cref{lem:path-from-helpful-vertex} are blue.
    The paths found by \cref{lem:bag-path} are red.}
    \label{fig:path-from-helpful-vertex}
\end{figure}

\begin{lemma}[Bag-Path Lemma]
\label{lem:bag-path}
    Assume the~\allref{set:block-bagel}.
    Let $t$ be a helpful node of~$O$.
    Then there is a $z'_1$--$z'_2$ path in $G-z'_1 z'_2$ with interior in $\mathring{G}_t$.
\end{lemma}
\begin{proof}
    Otherwise the two adhesion-vertices of $G_t$ together with one of the $z'_i$ forms a 3-separator of~$G$, contradicting 4-connectivity.
\end{proof}

\begin{lemma}[One-Path Lemma]
\label{lem:double-wheel-one-path}
    Assume the~\allref{set:block-bagel}.
    Let~$P \subseteq O$ be a path of length~$\ge 2$ with end-nodes~$r$ and~$s$.
    Let $u \in V(G_r)$ and $v \in V(G_s)$.
    Then there is a $u$--$v$ path in $\bigcup_{t \in V(P)} G_{t}$ that internally avoids the adhesion-vertices of the edges not in~$P$.
\end{lemma}
\begin{proof}
    Let~$u_i$ denote the adhesion-vertex of $t(i)\, t(i+1)$.
    We assume without loss of generality that $P=t(1)\, t(2) \ldots t(k)$. 
    For each~$i \in \{1,2,\ldots,k\}$, we construct a path~$P_i$ in~$G_{t(i)}$, as follows.
    For $i=1$ we claim that $P_i$ can be chosen to be a $u$--$u_1$ path with interior in~$G_{t(1)} - u_0$. 
    This is clearly possible if the bag $G_{t(1)}$ is a~$K_2$.
    Otherwise~$G_{t(1)}$ is 2-connected.
    Then we find two independent $u$--$u_1$ paths in~$G_{t(1)}$, and at least one of them internally avoids~$u_0$.
    Similarly, if~$i=k$ then we let~$P_i$ be a $u_{k-1}$--$v$ path with interior in~$G_{t(k)}-u_k$.
    Otherwise, if $2 \leq i \leq k$, then we let~$P_i$ be a $u_{i-1}$--$u_i$ path in~$G_{t(i)}$.
    Then the concatenation $P:=P_1 P_2 \ldots P_k$ is a desired $u$--$v$ path.
\end{proof}

\begin{lemma}[Two-Fan Lemma]
\label{lem:double-wheel-two-fan}
    Assume the~\allref{set:block-bagel}.
    Let~$u$ and~$v$ be the two adhesion-vertices in a bag~$G_t$.
    Assume that either~$t$ is helpful or~$u$ is the adhesion-vertex of a helpful edge.
    Then there exist two independent paths~$P_1, P_2$ with all internal vertices in $\mathring{G}_{t}$ such that~$P_1$ is a $u$--$z_1'$ path and~$P_2$ is a $u$--$v$ path.
\end{lemma}
\begin{proof}
    If~$u$ is the adhesion-vertex of a helpful edge, then~$u z_1' \in E(G)$ by \cref{lem:path-from-helpful-edge}, and~$u v \in E(G)$ since~$G_t$ is a~$K_2$. Then we let~$P_1 := u z_1'$ and~$P_2 := u v$.
    Otherwise,~$t$ is helpful. 
    By \cref{lem:path-from-helpful-vertex}, $z'_1$ has a neighbour~$w$ in $\mathring{G}_{t}$.
    Since~$G_{t}$ is 2-connected, it contains two independent paths~$Q_1,Q_2$ such that~$Q_1$ is a $u$--$w$ path and~$Q_2$ is a $u$--$v$ path. 
    Then we obtain $P_1$ from $Q_1$ by appending~$z'_1$, and we let $P_2 := Q_2$.
\end{proof}

The following lemma is supported by \cref{fig:block-bagel-2-paths}.

\begin{lemma}[Two-Paths Lemma]
\label{lem:double-wheel-two-paths}
    Assume the~\allref{set:block-bagel}.
    Let~$P$ be an $r$--$s$ path in~$O$. 
    Assume that either the first edge $e$ or the second node of~$P$ is helpful.
    Assume that~$P$ contains at least two helpful elements distinct from~$r$.
    Let~$u$ be the adhesion-vertex of the edge~$e$.
    Then there exist two independent $u$--$z_1'$ paths in $G$ with all interior vertices in~$\bigcup_{t \in V(\mathring P)} G_t\cup \mathring G_s$.
\end{lemma}
\begin{proof}
    Without loss of generality,~$P=t(0)\, t(1) \ldots t(k)$ with $r=t(0)$ and $s=t(k)$.
    By assumption, either the edge $e=t(0)\,t(1)$ or the second node~$t(1)$ is helpful.
    Let~$u_i$ denote the adhesion-vertex of $t(i)\, t(i+1)$. In particular, $u=u_0$.
    By the \allref{lem:double-wheel-two-fan}, there are two independent paths~$P_1,Q_1$ with interior in~$\mathring{G}_{t(1)}$ such that~$P_1$ is a $u_0$--$z_1'$ path and~$Q_1$ is a $u_0$--$u_1$ path.

    If~$t(1) t(2)$ or~$t(2)$ is helpful, then~$k=2$. 
    In this case, we find a $u_1$--$z_1'$ path~$Q_2$ with interior in~$\mathring{G}_{t(2)}$ by the \allref{lem:double-wheel-two-fan}. 
    Then the two paths $P_1$ and~$u_0 Q_1 u_1 Q_2 z_1'$ are as desired.
    Otherwise,~$t(1)$ is helpful,~$t(2)$ has a $K_2$-bag and at least one of~$t(2) t(3)$ or~$t(3)$ is helpful. 
    So~$k=3$ in this case. 
    Then~$u_1 u_2 \in E(G)$ since~$t(2)$ has a $K_2$-bag. 
    We find a $u_2$--$z_1'$ path~$Q_2$ with interior in~$\mathring{G}_{t(3)}$ by the \allref{lem:double-wheel-two-fan}. 
    Then the paths~$P_1$ and~$u_0 Q_1 u_1 u_2 Q_2 z_1'$ are as desired.
\end{proof}

\begin{figure}[ht]
    \centering
    \includegraphics[height=7\baselineskip]{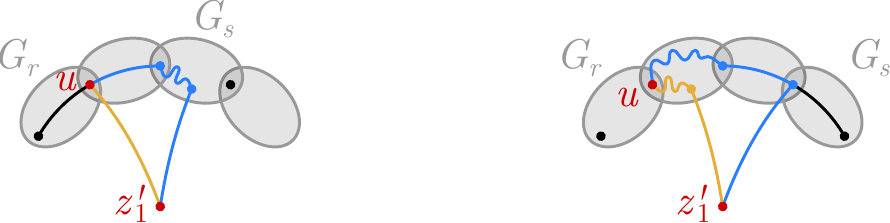}
    \caption{The paths constructed in the \allref{lem:double-wheel-two-paths}. Left: the first edge on~$P$ is helpful. Right: the second node on $P$ is helpful.}
    \label{fig:block-bagel-2-paths}
\end{figure}

\subsection{Totally-nested tetra-separations from good bags}\label{sec:WheelTotallyNested}

\begin{definition}[Good bag, bad bag]
    Assume the \allref{set:block-bagel}.
    A bag $G_t$ of $\cO$ is \defn{good} if $G_t$ is 2-connected and either
    \begin{itemize}
        \item $|G_t|\ge 4$, or
        \item $G_t$ is a triangle and a neighbouring bag is a~$K_2$.
    \end{itemize}
    Otherwise, $G_t$ is \defn{bad}.
    We also call $t$ \defn{good} or \defn{bad}, respectively.
\end{definition}

Note that $G_t$ is bad if and only if either $G_t$ is a $K_2$ or $G_t$ is a triangle and both neighbouring bags are 2-connected.

The following setting is supported by \cref{fig:bag-setting}.

\begin{setting}[Block-bagel bag setting]
\label{set:bag-setting}
    Assume the \allref{set:block-bagel}.
    Let~$U' := U_{t(0)}'$ and~$W' := W_{t(0)}'$.
    Let $x'$ be the adhesion-vertex of~$t(-1)\,t(0)$ and let $y'$ be the adhesion-vertex of~$t(0)\, t(1)$.
    Observe that $S(U',W')=\{x',y',z_1',z_2'\}$.
    Let $(U,W)$ be the right-left-reduction of~$(U',W')$.
    Let $x,y,z_1,z_2\in S(U,W)$ correspond to $x',y',z'_1,z'_2$, respectively.\footnote{So $x$ is incident or equal to $x'$, and so on.}
    Let $x^W$ denote the unique vertex in $\hat x\cap W$, and define $y^W$ and $z_i^W$ similarly.
\end{setting}

\begin{figure}[ht]
    \centering
    \includegraphics[height=12\baselineskip]{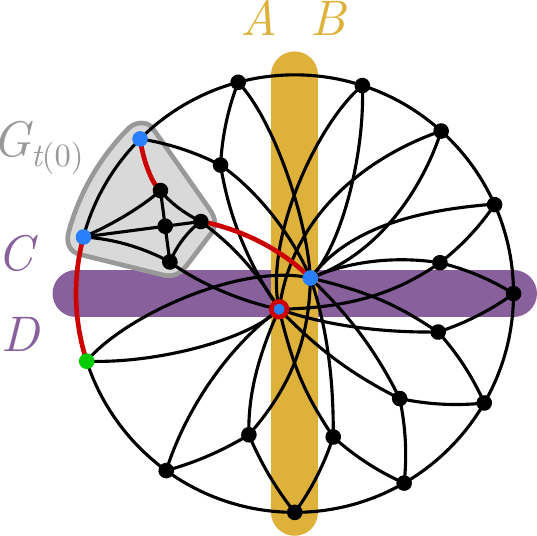}
    \caption{The \allref{set:block-bagel}. $S(U',W')$ is blue, $S(U,W)$ is red. We have $y^W=y'$ and $z^W_i=z'_i$, so these three vertices are blue. The vertex $x^W$ is green.}
    \label{fig:bag-setting}
\end{figure}

\begin{observation}
\label{obs:double-wheel-out-reduction}
    Assume the~\allref{set:bag-setting}. 
    If $G_{t(1)}$ is a~$K_2$, then~$y$ is the edge in~$G_{t(1)}$.
    Symmetrically, if~$G_{t(-1)}$ is a~$K_2$, then~$x$ is the edge in~$G_{t(-1)}$.\qed
\end{observation}

\begin{lemma}
\label{cor:double-wheel-reduced-edges}
    Assume the~\allref{set:bag-setting}. 
    Then each vertex $z'_i$ in the centre has $\ge 3$ neighbours outside $Z'$ contained in $W'\sm U'$, and in particular in $W\sm U$.
    In particular, $z_i^W = z_i'$.
\end{lemma}
\begin{proof}
    We combine the \allref{lem:helpful-lemma} with \cref{lem:path-from-helpful-edge} and \cref{lem:path-from-helpful-vertex}, and we use \cref{obs:reduction-maintains-proper-sides} for $W' \sm U' \subseteq W \sm U$.
\end{proof}

\begin{lemma}
\label{lem:double-wheel-UW-is-tetra}
    Assume the~\allref{set:bag-setting}.
    Then the following assertions are equivalent:
    \begin{enumerate}
        \item\label{lem:double-wheel-UW-is-tetra1} $G_{t(0)}$ is good, or all three bags $G_{t(-1)},G_{t(0)},G_{t(1)}$ are $K_2$'s;
        \item\label{lem:double-wheel-UW-is-tetra2} $(U,W)$ is a tetra-separation.
    \end{enumerate}
\end{lemma}
\begin{proof}
    \cref{lem:double-wheel-UW-is-tetra1} $\Rightarrow$ \cref{lem:double-wheel-UW-is-tetra2}.
    Suppose first that $G_{t(0)}$ is good.
    By \cref{wheelCornersNonempty}, each of the four corners contains a vertex.
    Since $G_{t(0)}$ is 2-connected, it is included in $A\cap C$, say, so every corner except $AC$ has its vertex contained in $W'\sm U'$.
    Hence $|W'\sm U'|\ge 3$.
    The 4-separation $(U',W')$ trivially satisfies the matching-condition.
    If the bag~$G_{t(0)}$ is not a triangle, then it has $\ge 4$ vertices, so $|U'\sm W'|\ge 2$.
    In this case, $(U,W)$ is a tetra-separation by \cref{leftrightTetra}.
    Otherwise, $G_{t(0)}$ is a triangle.
    Then some neighbouring bag, say $G_{t(1)}$, is a $K_2$.
    Hence $|U' \sm W'| = 1$ and the vertex~$y'$ has exactly one neighbour in~$W' \sm U'$. 
    Again by \cref{leftrightTetra}, $(U,W)$ is a tetra-separation.

    Suppose now that all three bags $G_{t(-1)},G_{t(0)},G_{t(1)}$ are $K_2$'s.
    Then $G_{t(0)}$ meets at most two corners, so the vertices in the other two corners lie in $W'\sm U'$, giving $|U'\sm W'|\ge 2$.
    Both $x'$ and $y'$ have exactly one neighbour in $W'\sm U'$ since both $G_{t(-1)}$ and $G_{t(1)}$ are $K_2$'s.
    Hence $(U,W)$ is a tetra-separation by \cref{leftrightTetra}.

    $\neg$ \cref{lem:double-wheel-UW-is-tetra1} $\Rightarrow$ $\neg$ \cref{lem:double-wheel-UW-is-tetra2}.
    As \cref{lem:double-wheel-UW-is-tetra1} fails, $G_{t(0)}$ either is a $K_2$ neighboured by at least one 2-connected bag, or $G_{t(0)}$ is a triangle neighboured by 2-connected bags.
    Suppose first that $G_{t(0)}$ is a $K_2$ neighboured by a 2-connected bag, say $G_{t(1)}$.
    Then $y'$ has two neighbours in $G_{t(1)}$, which are also contained in $W'\sm U'$.
    The two vertices $z'_i$ in the centre have $\ge 3$ neighbours in $W'\sm U'$ by \cref{cor:double-wheel-reduced-edges}.
    So only $x'$ might have $\le 1$ neighbours in $W'\sm U'$.
    But $|U'\sm W'|=0$, so $(U,W)$ is no tetra-separation by \cref{leftrightTetra}.

    Suppose now that $G_{t(0)}$ is a triangle with both neighbouring bags 2-connected.
    Then, as above, all vertices in $S(U',W')$ have $\ge 2$ neighbours in $W'\sm U'$.
    But $|U'\sm W'|=1$, so $(U,W)$ is no tetra-separation by \cref{leftrightTetra}.
\end{proof}

\begin{observation}
    \label{obs:double-wheel-helpful-is-next}
    Assume the~\allref{set:bag-setting}. 
    Let~$t(i)\, t(i+1)$ be the edge of~$O$ that has $y^W$ as adhesion-vertex.
    Then either the edge $t(i)\, t(i+1)$ or the node $t(i+1)$ is helpful.\qed
\end{observation}

\begin{lemma}
    \label{lem:three-y1-y2-paths}
    Assume the~\allref{set:bag-setting}. 
    Then there exist three independent $z_1'$--$z_2'$ paths in~$G[W]-\{x',y'\}-z'_1 z'_2$.
\end{lemma}
\begin{proof}
    For each helpful $h\in V(O)\cup E(O)$ we construct a $z'_1$--$z'_2$ path~$P_h$, as follows.
    If $h$ is an edge, then its adhesion-vertex send edges to both $z'_i$ by \cref{lem:path-from-helpful-edge}, and we let $P_h$ be the path defined by these two edges; see \cref{fig:three-y1-y2-paths}.
    If $h$ is a helpful node, then we use the \allref{lem:bag-path} to let $P_h$ by a $z'_1$--$z'_2$ path avoiding the edge $z'_1 z'_2$ (if it exists) and with interior in $\mathring{G}_h$.

    All paths $P_h$ are independent.
    By the \allref{lem:helpful-lemma}, there are at least four of them.
    At most one of them meets the bag $G_{t(0)}$.
    By \cref{cor:double-wheel-reduced-edges}, both vertices $z'_i$ in the centre are contained in~$W$.
    Hence at least three of the paths $P_h$ are independent $z'_1$--$z'_2$ paths in $G[W]-\{x',y'\}-z'_1 z'_2$.
\end{proof}

\begin{figure}[ht]
    \centering
    \includegraphics[height=8\baselineskip]{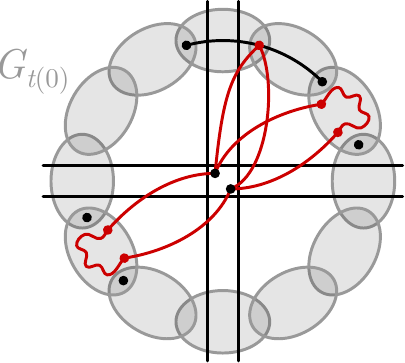}
    \caption{The red paths are constructed in \cref{lem:three-y1-y2-paths}.}
    \label{fig:three-y1-y2-paths}
\end{figure}

The following lemma is supported by \cref{fig:double-wheel-three-paths}.
\begin{lemma}
\label{lem:double-wheel-applied-three-paths}
    Assume the~\allref{set:bag-setting} with $G_{t(0)}$ 2-connected. 
    Then there are three independent $\{x^W,y^W\}$--$z_1'$ paths $P_1,P_2,P_3$ in~$G[W]$, such that:
    \begin{enumerate}
        \item both $P_1$ and $P_2$ are $y^W$--$z_1'$ paths, and
        \item $P_3$ is an $x^W$--$z_1'$ path.
    \end{enumerate}
\end{lemma}

\begin{proof}
    Since $t(0)$ is helpful, \allref{lem:helpful-lemma} implies that
    \begin{equation}\label{lem:double-wheel-applied-three-paths-eq}
        \text{$\ge 3$ elements of $O-t(0)$ are helpful.}
    \end{equation}

    Let $f=t(j)\,t(j+1)$ be the edge on $O$ that has $y^W$ as adhesion-vertex.
    Then $f$ or $t(j+1)$ is helpful by \cref{obs:double-wheel-helpful-is-next}.
    If $t(j)$ is helpful, then $j=0$, and $j=1$ otherwise.
    Let $P$ be the path in $O$ that starts with $t(j),t(j+1)$ and has two helpful elements distinct from $t(j)$ and is minimal with these properties.
    The path $P$ is well-defined since $O-t(0)$ has $\ge 3$ helpful elements by \cref{lem:double-wheel-applied-three-paths-eq}.
    Moreover, the end $s$ of~$P$ is distinct from~$t(0)$.
    By the \allref{lem:double-wheel-two-paths}, we find two independent $y^W$--$z_1'$ paths~$Q_1, Q_2$ with all interior vertices in $\bigcup_{t \in V(\mathring P)} G_{t} \cup \mathring G_{s}$.
    By construction, the paths $Q_1$ and $Q_2$ are contained in~$G[W]$.

    Now let $e=t(i)\,t(i-1)$ be the edge on $O$ that has $x^W$ as adhesion-vertex.
    Then $e$ or $t(i-1)$ is helpful by \cref{obs:double-wheel-helpful-is-next}.
    And if $t(i)$ is helpful, then $i=0$.
    By the \allref{lem:double-wheel-one-path}, we find an $x^W$--$z_1'$ path~$Q_3$ with interior in~$\mathring{G}_{t(i-1)}$.
    Once more, $Q_3\se G[W]$ by construction.

    It remains to show that the three paths $Q_i$ are independent.
    We already know that $Q_1$ and $Q_2$ are independent, so it remains to show that $Q_3$ is independent from $Q_1$ and~$Q_2$.
    If the edge $e$ is helpful, then it is the first helpful element on $O$ following $t(0)$ in the direction of~$e$.
    In particular, $e$ does not lie on~$P$ by~\cref{lem:double-wheel-applied-three-paths-eq}.
    Since the bag $G_{t(i-1)}$ is a~$K_2$, it has empty interior.
    So $Q_3=x^W z'_1$ avoids $\mathring G_s$, even if $s=t(i-1)$, so $Q_3$ is independent from $Q_1$ and $Q_2$.
    Otherwise $t(i-1)$ is helpful.
    Then $t(i-1)$ is the first helpful element on $O$ following $t(0)$ in the direction of~$e$.
    In particular, both edges on $O$ incident to $t(i-1)$ are unhelpful.
    So $t(i-1)$ does not lie on $P$ by \cref{lem:double-wheel-applied-three-paths-eq}.
    Hence $Q_3$ is independent from $Q_1$ and~$Q_2$.
\end{proof}

\begin{figure}[ht]
    \centering
    \includegraphics[height=10\baselineskip]{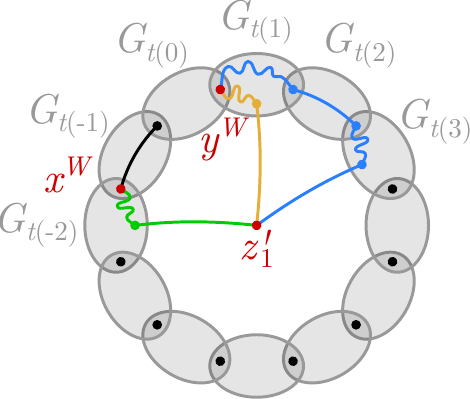}
    \caption{The paths constructed in \cref{lem:double-wheel-applied-three-paths}. $P_1$ is yellow, $P_2$ is blue, $P_3$ is green.}
    \label{fig:double-wheel-three-paths}
\end{figure}

\begin{lemma}
\label{lem:double-wheel-path-shortening}
    Let~$G$ be a graph, let~$u \in U \subseteq V(G)$ and $v \in V \subseteq V(G)$.
    Every $u$--$v$ path in~$G$ contains a $U$--$V$ path.\qed
\end{lemma}

In the following, we characterise when $(U,W)$ is a totally-nested tetra-separation.\medskip

\begin{center}
    \textbf{Half-connected}
\end{center}

\begin{lemma}
    \label{lem:double-wheel-half-connected}
    Assume the~\allref{set:bag-setting}. Then $(U,W)$ is half-connected.
\end{lemma}
\begin{proof}
    We show that~$G[W \sm U]$ is connected.
    Without loss of generality, assume that the interior of~$G_{t(0)}$ is contained in the corner~$AC$. 
    Let~$v$ be a vertex in the corner~$BD$, which exists by \cref{wheelCornersNonempty}.
    Let $t(i)$ be a node of $O$ with $v\in G_{t(i)}$.

    We claim that every vertex $w$ in $W$ is connected to~$v$ in~$G[W \sm U]$.
    Since~$G_{t(i)} \subseteq G[W \sm U]$, every vertex in~$G_{t(i)}$ is connected to~$v$ in~$G[W \sm U]$.
    If $w$ does not lie in the centre~$Z'$, then we are done by the \allref{lem:double-wheel-one-path}.
    Otherwise $w=z'_1$, say.
    Then $w$ has a neighbour in $W\sm Z'$ by \cref{lem:three-y1-y2-paths}, which we already know is linked to $v$ in $G[W\sm U]$.     
    Hence $(U,W)$ is half-connected.
\end{proof}

\begin{center}
    \textbf{3-linked}
\end{center}

\begin{lemma}
\label{lem:double-wheel-3-linked}
    Assume the~\allref{set:bag-setting} with $G_{t(0)}$ 2-connected. 
    Then $(U,W)$ is 3-linked.
\end{lemma}
\begin{proof}
    Let $\pi$ consist of two elements of $S(U,W)$.
    We show that $\pi$ is 3-linked around $S(U,W)$.

    \casen{Case $\pi=\{x,y\}$.}
    We may assume that both $z_1,z_2$ are vertices, by \ref{itm:3-linked-1}.
    Since the bag $G_{t(0)}$ is 2-connected, it contains two independent $x'$--$y'$ paths.
    By \cref{lem:double-wheel-path-shortening}, these paths contain two independent $\hat x$--$\hat y$ paths which are also included in the bag~$G_{t(0)}$.
    We find a third $\hat x$--$\hat y$ path $G$ that avoids both the bag $G_{t(0)}$ and the centre, by the \allref{lem:double-wheel-one-path}.
    So we are done by \ref{itm:3-linked-3}.

    \casen{Case $\pi=\{z_1,z_2\}$.}
    We find three independent $z'_1$--$z'_2$ paths in $G-x-y$ by \cref{lem:three-y1-y2-paths}.
    These contain three independent $\hat z_1$--$\hat z_2$ paths by \cref{lem:double-wheel-path-shortening}.
    So we are done by \ref{itm:3-linked-3}.

    \casen{Case $\pi=\{x,z_1\}$.}
    We find an $x'$--$z'_1$ path with interior in $\mathring{G}_{t(0)}$, by the \allref{lem:double-wheel-two-fan}.
    This path contains a $\hat x$--$\hat z_1$ path~$P_1$, by \cref{lem:double-wheel-path-shortening}.
    Next, we find two independent $x^W$--$z'_1$ paths $P_2$ and $P_3$ in $G$ that internally avoid $G_{t(0)}$, by \cref{lem:double-wheel-applied-three-paths}.
    Both paths $P_2$ and $P_3$ are $\hat x$--$\hat z_1$ paths by \cref{cor:double-wheel-reduced-edges}.
    Hence the $P_i$ witness \ref{itm:3-linked-3}.

    All remaining cases are symmetric to the case $\pi=\{x,z_1\}$.
\end{proof}

\begin{center}
    \textbf{0-potter-linked}
\end{center}

\begin{lemma}
\label{lem:double-wheel-0-potter-linked}
    Assume the~\allref{set:bag-setting} with $G_{t(0)}$ good.
    Then $(U,W)$ is 0-potter-linked.
\end{lemma}
\begin{proof}
    Let $\{\pi_1,\pi_2\}$ be a balanced bipartition of $S(U,W)$.

    \casen{Case $\pi_1 = \{x,y\}$ and $\pi_2=\{z_1,z_2\}$}.
    Suppose first that~$G_{t(0)}$ is a triangle with a neighbouring bag that is a~$K_2$; say $G_{t(1)}$ is a $K_2$.
    Let $v$ denote the vertex of the triangle $G_{t(0)}$ besides $x'$ and~$y'$.
    Since $G$ is 4-connected, $v$ sends edges to both vertices $z'_i$ in the centre.
    Similarly, the adhesion-vertex $y'$ in between the triangle-bag $G_{t(0)}$ and the $K_2$-bag $G_{t(1)}$ sends at least one edge to some vertex in the centre, say to~$z'_1$.
    
    We claim that $z'_1=z_1$.
    On the one hand, $z'_1$ has $\ge 3$ neighbours in $W'\sm U'$ by \cref{cor:double-wheel-reduced-edges}, so $z'_1$ remains in the separator when taking the right-reduction $(U'',W'')$ of $(U',W')$.
    However, $y'$ becomes the edge $y$ of the $K_2$-bag $G_{t(1)}$ in the separator $S(U'',W'')$ by \cref{obs:double-wheel-out-reduction}.
    Hence $z'_1$ has two neighbours $y'$ and $v$ in $U''\sm W''$.
    Thus $z'_1$ also remains in the separator when taking the left-reduction $(U,W)$ of $(U'',W'')$.
    In particular, $z'_1=z_1$.

    So $z_1$ is a vertex.
    Recall that $z_1$ sends an edge to $y'\in\hat\pi_1$.
    Hence we have~\ref{itm:0-potter-linked-1}.
    
    Suppose otherwise that $G_{t(0)}$ is 2-connected and has at least four vertices.
    We will find five independent $\hat \pi_1$--$\hat \pi_2$ paths in~$G$ for \ref{itm:0-potter-linked-2}.
    There exist two disjoint $\{z'_1,z'_2\}$--$\mathring G_{t(0)}$ paths $P_1,P_2$ by \cref{lem:path-from-helpful-vertex}.
    Since $G_{t(0)}$ is 2-connected, we can extend the paths $P_1,P_2$ to two disjoint $\{z'_1,z'_2\}$--$\{x',y'\}$ paths with interior in $\mathring G_{t(0)}$.
    We find three independent $\{x_1^W,x_2^W\}$--$\{z_1',z_2'\}$ paths in $G[W]$ by \cref{lem:double-wheel-applied-three-paths}.
    The five paths are independent, and they contain five $\hat\pi_1$--$\hat\pi_2$ paths by \cref{lem:double-wheel-path-shortening}.

    \casen{Case $\pi_1=\{x,z_1\}$ and $\pi_2=\{y,z_2\}$}.
    Here it suffices to find five independent $\hat\pi_1$--$\hat\pi_2$ paths in $G$ for \ref{itm:0-potter-linked-2}.
    There are two independent $x'$--$y'$ paths in $G_{t(0)}$ by 2-connectivity.
    We find three independent $z'_1$--$z'_2$ paths in $G[W]-\{x',y'\}$ by \cref{lem:three-y1-y2-paths}.
    The latter three paths are internally disjoint from the former two.
    By \cref{lem:double-wheel-path-shortening}, the five paths contain five independent $\hat\pi_1$--$\hat\pi_2$ paths.

    All remaining cases are symmetric to one of the two cases above.
\end{proof}

\begin{center}
    \textbf{1-potter-linked and 2-potter-linked}
\end{center}

\begin{lemma}
\label{lem:double-wheel-weird}
    Assume the~\allref{set:bag-setting}. 
    Then no pair of two elements of $S(U,W)$ is $W$-weird.
\end{lemma}
\begin{proof}
    Let $\pi$ consist of two elements of $S(U,W)$.
    Assume for a contradiction that $\pi$ is $W$-weird.

    \casen{Assume $\pi=\{x,z_1\}$.}
    The vertex $z'_1$ in the centre has $\ge 3$ neighbours in $W'\sm U'$ by \cref{cor:double-wheel-reduced-edges}.
    Hence $z^W_1=z'_1$.
    Since the $\ge 3$ neighbours also lie in $W\sm U$ by \cref{obs:reduction-maintains-proper-sides}, we further get that $z_1$ must be the edge in~$\pi$.
    Since $z_1$ is an edge, $x$ is a vertex which has exactly one neighbour in $G[W]-z_1^W$.
    Hence $G_{t(-1)}$ is a~$K_2$.
    But then $x$ is the edge of $G_{t(-1)}$ by \cref{obs:double-wheel-out-reduction}, a contradiction.

    \casen{Assume $\pi=\{x,y\}$.}
    Say $x$ is a vertex and $y$ is an edge.
    Then $x=x'$, so $G_{t(-1)}$ is 2-connected by \cref{obs:double-wheel-out-reduction}.
    Hence $x$ has two neighbours $u,v$ in $G_{t(-1)}$.
    By the \allref{lem:helpful-lemma}, neither $u$ nor $v$ is an endvertex of the edge~$y$.
    Hence $\{x,y\}$ is not $W$-weird, a contradiction.

    All other cases are covered by symmetry.
\end{proof}

\begin{lemma}
    \label{lem:double-wheel-potter-linked}
    Assume the~\allref{set:bag-setting} with $G_{t(0)}$ 2-connected. Then $(U,W)$ is 1-potter-linked and 2-potter-linked.
\end{lemma}
\begin{proof}
    Let $\{\pi_1,\pi_2\}$ be a balanced bipartition of $S(U,W)$. 
    Neither~$\pi_i$ is $W$-weird by \cref{lem:double-wheel-weird}.
    In particular, $\{\pi_1,\pi_2\}$ is 2-potter-linked.
    
    To show that $\{\pi_1,\pi_2\}$ is 1-potter-linked, it suffices to find three independent $\hat \pi_1$--$\hat \pi_2$ paths in~$G[W]$.
    If one of the $\pi_i$ is of the form $\{x,y\}$, then we find three such paths by \cref{lem:double-wheel-applied-three-paths} and \cref{lem:double-wheel-path-shortening}.
    Otherwise $z_1\in\pi_1$ and $z_2\in\pi_2$, say.
    Then we find three desired paths by \cref{lem:three-y1-y2-paths} and \cref{lem:double-wheel-path-shortening}.
\end{proof}

\begin{center}
    \textbf{Summary}
\end{center}

\begin{keylemma}
    \label{cor:double-wheel-totally-nested}
    Assume the~\allref{set:block-bagel}.
    Let $t$ be a node of~$O$.
    If the bag $G_t$ is good, then $(U_t,W_t)$ is a totally-nested tetra-separation.
\end{keylemma}
\begin{proof}
    We may assume the \allref{set:bag-setting} with~$t=t(0)$.
    Then $(U_t,W_t)$ is a tetra-separation by \cref{lem:double-wheel-UW-is-tetra}.
    It suffices to show that $(U_t,W_t)$ is externally 5-connected, by \cref{keylem:nestedness-external-connectivity}.
    This follows from \cref{lem:double-wheel-half-connected} (half-connected), \cref{lem:double-wheel-3-linked} (3-linked), \cref{lem:double-wheel-0-potter-linked} (0-potter-linked) and \cref{lem:double-wheel-potter-linked} (1-potter-linked and 2-potter-linked).
\end{proof}

\subsection{Splitting stars and torsos}\label{sec:WheelSplitting}

\begin{definition}[Tetra-star $\sigma(\cO)$]
    Assume the \allref{set:block-bagel}.
    The \defn{tetra-star} of~$\cO$ is the set of all $(U_t,W_t)$ for good bags~$G_t$ of~$\cO$.
    We denote the tetra-star of~$\cO$ by $\defnMath{\sigma(\cO)}$.
\end{definition}

\begin{lemma}
\label{lem:double-wheel-star}
    Assume the \allref{set:block-bagel}.
    Then $\sigma(\cO)$ is a star of totally-nested tetra-separations.
\end{lemma}
\begin{proof}
    By \cref{cor:double-wheel-totally-nested},~$\sigma(\cO)$ is a set of totally-nested tetra-separations.
    It remains to show that for all good nodes $s\neq t$, we have $(U_s,W_s) \leq (W_t,U_t)$.
    We have $U_s \subseteq V(G_s)$ and $U_t \subseteq V(G_t)$.
    Moreover, by \cref{obs:double-wheel-out-reduction}, for all nodes~$r \in V(O-s)$ that have a 2-connected bag, we have $V(G_r) \subseteq W_s$.
    Therefore, $U_t \subseteq V(G_t) \subseteq W_s$.
    Similarly, $U_s \subseteq V(G_s) \subseteq W_t$.
    It follows that~$(U_s,W_s) \leq (W_t,U_t)$.
\end{proof}

The following two lemmata are supported by \cref{fig:double-wheel-moving-links}.

\begin{lemma}[Link-Moving Lemma for $K_2$-bags]
\label{lem:double-wheel-moving-link-1}
    Assume the \allref{set:block-bagel}.
    Let~$P$ and~$Q$ be the two paths that are the components of~$O - o(C) - o(D)$.
    Assume $o(A) \in P$ and $o(B) \in Q$.
    Let~$c$ denote the end of~$P$ that is adjacent or incident to~$o(C)$ in~$O$.
    Define~$d$ similarly.
    Let~$s$ be a node in the same component of $P-o(A)$ as~$c$ such that $G_s=K_2$.
    Let $P' := dP\mathring{s}$ and $Q':= \mathring{s} P c \cup Q$.
    \begin{enumerate}
        \item\label{itm:double-wheel-moving-link-1-1} $A' := \bigcup_{t \in V(P')} V(G_t)\cup Z'$ and $B':= \bigcup_{t \in V(Q')} V(G_t)\cup Z'$ are the two sides of a tetra-separation of~$G$.
        \item\label{itm:double-wheel-moving-link-1-2} $S(A',B') = Z' \cup (D\text{-link}) \cup \{e\}$, where~$e$ is the edge in~$G_s$.
        \item\label{itm:double-wheel-moving-link-1-3} $(A',B')$ separates some two vertices in $B \cup D$.
        \item\label{itm:double-wheel-moving-link-1-4} $(A',B')$ and~$(C,D)$ cross with all links of size one, and the block-bagel induced by~$(A',B')$ and~$(C,D)$ is equal to~$\cO$.
    \end{enumerate}
\end{lemma}
\begin{proof}
    The pair $(A',B')$ is a mixed-separation by construction.
    Moreover,~\ref{itm:double-wheel-moving-link-1-2} also holds by construction. We show~\ref{itm:double-wheel-moving-link-1-1} next. Let~$d$ be the element in the $D$-link.

    \casen{Matching-condition:} We have to show that, if~$d$ in an edge, then it does not share an endvertex with~$e$.
    The endvertices of~$d$ lie in the corners for~$AD$ and~$BD$ since there do not exist dangling edges in the crossing-diagram of~$(A,B)$ and~$(C,D)$ by the \allref{keylem:crossing}. 
    Both endvertices of~$e$ lie in the corner for~$AC$ by \cref{wheelLinksVsCorners}~\ref{wheelLinksVsCorners1}.
    Hence $d$ and $e$ share no ends.

    \casen{Degree-condition:} 
    All vertices in~$S(A',B')$ are also contained in~$S(A,B)$.
    Since $B \sm A \subseteq B' \sm A'$, every vertex in~$S(A',B')$ has at least two neighbours in~$B' \sm A'$.
    
    If~$d$ is a vertex, then we claim that every neighbour~$u$ of~$d$ in~$A \sm B$ also lies in~$A' \sm B'$.
    Indeed, since $ud$ is not a dangling edge,~$u$ is contained in the corner~$AD$ or the $A$-link with respect to~$(A,B),(C,D)$. 
    By \cref{wheelLinksVsCorners}~\ref{wheelLinksVsCorners1}, no endvertex of~$e$ lies in the $A$-link with respect to $(A,B),(C,D)$.
    So, we have $u \in A' \sm B'$.

    It remains to show that both vertices~$z_1'$ and~$z_2'$ in the centre have at least two neighbours in~$A' \sm B'$.
    Let~$s'$ be the neighbour of~$s$ on~$P'$.
    Observe that either~$s'$ or~$ss'$ is helpful and belongs to the corner~$AC$.
    Moreover, by the \allref{lem:helpful-lemma} there is a helpful element that belongs to the $AD$-corner.
    Hence, there are at least two helpful elements that belong to~$(ss') \cup T_{AD}$, where $T_{AD}$ is the $AD$-block-path.
    By \cref{lem:path-from-helpful-edge,lem:path-from-helpful-vertex}, we are done.

    \medskip

    We also have~\ref{itm:double-wheel-moving-link-1-3} since there is a vertex in $A \cap (D \sm C) \subseteq A' \sm B'$ that is separated from every vertex in $B \sm A \subseteq B' \sm A'$ by $(A',B')$.
    Next, we show~\ref{itm:double-wheel-moving-link-1-4}.
    Clearly,~$Z'$ is the centre of the crossing-diagram of~$(A',B')$ with~$(C,D)$. 
    So, $(A',B')$ and $(C,D)$ cross with all links of size one.
    It is straightforward to check that the block-bagel induced by~$(A',B')$ and~$(C,D)$ is equal to~$\cO$.
\end{proof}

Let $P$ be a path.
Let $u$ be a vertex on~$P$, and let $e$ be an edge on~$P$.
Then $uP\mathring e$ denotes the subpath of $P$ that starts in~$u$ and ends in the first end of~$e$ it encounters.
In particular, $e$ is not an edge of $uP\mathring e$.

\begin{lemma}[Link-Moving Lemma for adhesion-vertices]
\label{lem:double-wheel-moving-link-2}
    Assume the \allref{set:block-bagel}.
    Let~$P$ and~$Q$ be the two paths that are the components of~$O - o(C) - o(D)$.
    Assume $o(A) \in P$ and $o(B) \in Q$.
    Let~$c$ denote the end of~$P$ that is adjacent or incident to~$o(C)$ in~$O$.
    Define~$d$ similarly.
    Let~$e$ be an edge of~$O$ with 2-connected bags at both ends.
    Assume that~$e$ stems from the $AC$-block-path.
    Let $P':=dP\mathring{e}$ and $Q':=\mathring e P c \cup Q$.
    Then:
    \begin{enumerate}
        \item\label{itm:double-wheel-moving-link-2-1} $A' := \bigcup_{t \in V(P')} V(G_t)\cup Z'$ and $B':= \bigcup_{t \in V(Q')} V(G_t)\cup Z'$ are the two sides of a tetra-separation of~$G$.
        \item\label{itm:double-wheel-moving-link-2-2} $S(A',B') = Z' \cup (D\text{-link}) \cup \{v\}$, where~$v$ is the adhesion-vertex of the edge~$e$.
        \item\label{itm:double-wheel-moving-link-2-3} $(A',B')$ separates some two vertices in $B \cup D$.
        \item\label{itm:double-wheel-moving-link-2-4} $(A',B')$ and~$(C,D)$ cross with all links of size one, and the block-bagel induced by~$(A',B')$ and~$(C,D)$ is equal to~$\cO$.
    \end{enumerate}
\end{lemma}
\begin{proof}
    The proof is analogous to the proof of the \allref{lem:double-wheel-moving-link-1}, except that the matching-condition is trivially fulfilled, and that the vertex~$v$ satisfies the degree condition since it is contained in two 2-connected bags $G_s,G_{s'}$ with $V(G_s - v) \subseteq A' \sm B'$ and $V(G_{s'} - v) \subseteq B' \sm A'$.
\end{proof}

\begin{figure}[ht]
    \centering
    \includegraphics[height=12\baselineskip]{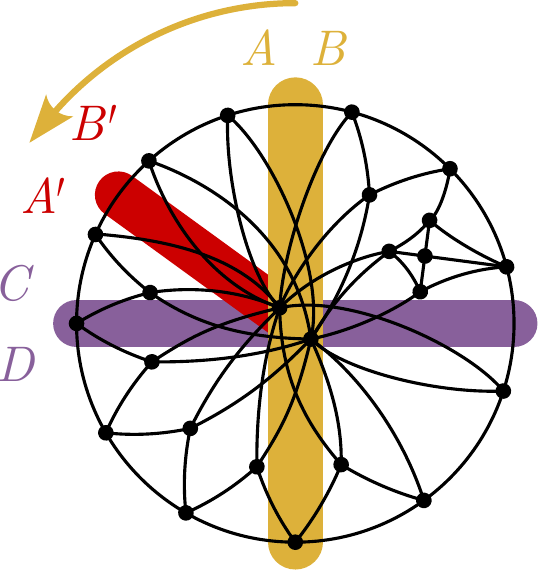}\hfill%
    \includegraphics[height=12\baselineskip]{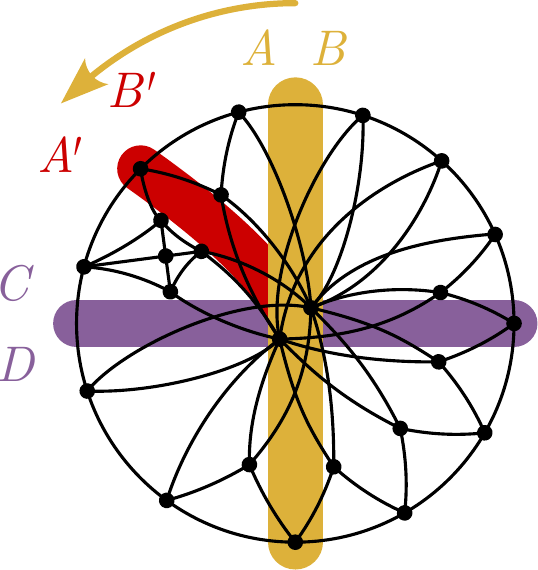}
    \caption{The tetra-separation $(A',B')$ constructed in the \allref{lem:double-wheel-moving-link-1} (left) and in the \allref{lem:double-wheel-moving-link-2} (right).}
    \label{fig:double-wheel-moving-links}
\end{figure}

\begin{lemma}\label{wheelBlockPathInteriorVertex}
    Assume the \allref{set:block-bagel}.
    Let~$T$ denote the $AC$-block-path.
    Let~$T':=T-o(A)-o(C)$.
    Then either~$T'$ is non-empty or $|A \cap C| \leq 3$.
\end{lemma}
\begin{proof}
    If the $A$-link or the $C$-link contains a vertex, then this vertex is contained in a 2-connected bag $G_t$ with $t\in V(T')$, by \cref{wheelLinksVsCorners}.
    Otherwise both the $A$-link and the $C$-link consist of edges.
    Then $G_{o(A)}$ and $G_{o(C)}$ are $K_2$-bags.
    Each of the two bags $G_{o(A)},G_{o(C)}$ has exactly one endvertex $v_A,v_C$ in the corner for~$AC$ by \cref{wheelLinksVsCorners}, respectively.
    If~$T'$ is empty, then $v_A=v_C$, and $A\cap C=\{v_A\}\cup Z'$ has size three.
\end{proof}

\begin{lemma}
\label{lem:double-wheel-E-in-bag}
    Assume the \allref{set:block-bagel}.
    Let $(E,F)$ be a tetra-separation of $G$ that interlaces the tetra-star $\sigma(\cO)$. 
    Then~$(E,F)$ is crossed by another tetra-separation of~$G$, or there is a 2-connected bag~$G_t$ of~$\cO$ such that~$E \subseteq U'_t$ or~$F \subseteq U'_t$.
\end{lemma}
\begin{proof}
    Assume that~$(E,F)$ is totally-nested. Then~$(E,F) \leq (A,B)$ and~$(E,F) \leq (C,D)$, say. 
    We choose two tetra-separations~$(A',B')$ and~$(C',D')$ with $|A' \cap C'|$ minimum such that
    \begin{enumerate}
        \item\label{itm:double-wheel-min-1} $(A',B')$ and~$(C',D')$ cross with all links of size one; and
        \item\label{itm:double-wheel-min-2} the block-bagel induced by~$(A',B')$ and~$(C',D')$ is equal to~$\cO$; and
        \item\label{itm:double-wheel-min-3} $(E,F) \leq (A',B')$ and $(E,F) \leq (C',D')$.
    \end{enumerate}
    Note that the existence of~$(A',B')$ and~$(C',D')$ follows from~$(A,B)$ and~$(C,D)$ exhibiting properties~\ref{itm:double-wheel-min-1}--\ref{itm:double-wheel-min-3}.
    Let $T$ denote the $A'C'$-block-path, and let~$T':=T-o(A')-o(C')$.
    By \cref{wheelBlockPathInteriorVertex}, $T'$ is non-empty.
    Let~$\cO' = \cO$ be the block-bagel induced by~$(A',B')$ and~$(C',D')$. 

    We claim that every bag $G_t$ with $t\in T'$ is 2-connected.
    Assume otherwise that $G_t=K_2$ for some $t\in T'$, and let $e$ denote the edge of~$G_t$.
    In the context of $(A',B')$ and $(C',D')$, we apply the \allref{lem:double-wheel-moving-link-1} to~$t$ and obtain a tetra-separation $(A'',B'')$ such that
    \begin{itemize}
        \item $S(A'',B'')=\{e\}\cup Z'\cup\text{($D'$-link)}$, and
        \item $(A'',B'')$ separates two vertices of~$B' \cup D'$. 
    \end{itemize}
    Since $B'\cup D'\se F$, the tetra-separation $(A'',B'')$ separates two vertices in~$F$.
    If $(A'',B'')$ also separates two vertices in~$E$, then it crosses~$(E,F)$, a contradiction.
    Hence we have either $(E,F) \leq (A'',B'')$ or $(E,F) \leq (B'',A'')$. 
    In the former case, we have that~$(A'',B'')$ and~$(C',D')$ satisfy all of~\ref{itm:double-wheel-min-1}-\ref{itm:double-wheel-min-3}, but have~$|A'' \cap C'| < |A' \cap C'|$, contradicting the minimality of $|A' \cap C'|$. 
    In the latter case, we apply the \allref{lem:double-wheel-moving-link-1} to~$t$ and obtain a tetra-separation $(C'',D'')$ such that
    \begin{itemize}
        \item $S(C'',D'')=\{e\}\cup Z'\cup\text{($B'$-link)}$, and
        \item $(C'',D'')$ separates two vertices of~$B' \cup D'$. 
    \end{itemize}
    Moreover, we have
    \[
        E\se B''\cap A'\cap C'\se C'' \quad\text{and}\quad F\supseteq A''\cup D'\cup B'\supseteq D''.
    \]
    Hence $(E,F)\le (C'',D'')$.
    Therefore, we have that~$(A',B')$ and~$(C'',D'')$ satisfy all of~\ref{itm:double-wheel-min-1}-\ref{itm:double-wheel-min-3}, but have~$|A' \cap C''| < |A' \cap C'|$, contradicting the minimality of $|A' \cap C'|$.

    We claim that $T'$ has no edge.
    Assume otherwise that $tt'$ is an edge of~$T'$.
    Then both bags $G_t,G_{t'}$ at its ends are 2-connected by the above claim.
    In the context of $(A',B')$ and $(C',D')$, we apply the \allref{lem:double-wheel-moving-link-2} to the edge~$tt'$ and obtain a tetra-separation $(A'',B'')$ with the properties stated in the lemma.
    Then we derive a contradiction in the same way as in the previous claim.

    By the above, every bag $G_t$ with $t\in T'$ is 2-connected, and $T'$ has no edge.
    Since $T'$ is non-empty, it has a unique node~$t$, and $G_t$ is 2-connected.
    Then~$A' \cap C' = V(G_t) \cup Z'$. 
    By \ref{itm:double-wheel-min-3}, $E\se U'_t$. 
\end{proof}

\begin{lemma}
\label{lem:double-wheel-interlaced-crossed}
    Assume the \allref{set:block-bagel}.
    Every tetra-separation of~$G$ that interlaces the tetra-star $\sigma(\cO)$ is crossed by another tetra-separation of~$G$.
\end{lemma}
\begin{proof}
    Assume for contradiction that~$\sigma(\cO)$ is interlaced by a totally-nested tetra-separation~$(E,F)$ of~$G$. 
    By \cref{lem:double-wheel-E-in-bag}, there is a 2-connected bag~$G_t$ of~$\cO$ with $E \subseteq U_t'$, say. 
    By \cref{Righty}, we have~$(E,F) \leq (U_t'',W_t'')$ where $(U''_t,W''_t)$ is the right-reduction of $(U'_t,W'_t)$.

    Assume first that~$G_t$ is good.
    Then~$(U_t,W_t) \in \sigma(\cO)$ by definition of~$\sigma(\cO)$. 
    Since~$(E,F)$ interlaces~$\sigma(\cO)$, we have $(U_t,W_t) < (E,F)$ or $(U_t,W_t) < (F,E)$. 
    We cannot have~$(U_t,W_t) < (F,E)$ since otherwise $\emptyset \neq E \sm F \subseteq (W_t \sm U_t) \cap (U_t'' \sm W_t'')$ contradicts that $U''_t\sm W''_t \se U_t \sm W_t$ by \cref{obs:shift-maintains-proper-sides}.
    So~$(U_t,W_t) < (E,F) \leq (U_t'',W_t'')$. 
    Since $(U_t,W_t)$ is the left-reduction of $(U''_t,W''_t)$, some vertex in $S(E,F)$ must have $\le 1$ neighbour in $E\sm F$, contradicting the degree-condition. 
    
    Otherwise,~$G_t$ is bad.
    Since $G_t$ is 2-connected, it is a triangle, and since $G_t$ is bad, both neighbouring bags in $\cO$ are 2-connected.
    Hence both adhesion-vertices of $G_t$ are contained in the separator of $(U''_t,W''_t)$.
    Therefore, the unique interior vertex of the triangle-bag $G_t$ is the only vertex of $U''_t\sm W''_t$.
    Since $(E,F) \leq (U_t'',W_t'')$, we have~$E \sm F \subseteq U_t'' \sm W_t''$. 
    Hence $|E \sm F| \leq 1$, which by \cref{lem:trivial-tetra-separations} contradicts that~$(E,F)$ is a tetra-separation.
\end{proof}

\begin{corollary}\label{wheelGoodChar}
    Assume the \allref{set:block-bagel}.
    Let $t$ be a node of~$O$.
    Then the following assertions are equivalent:
    \begin{enumerate}
        \item\label{wheelGoodChar1} the bag $G_t$ is good;
        \item\label{wheelGoodChar2} $(U_t,W_t)$ is a totally-nested tetra-separation of~$G$.
    \end{enumerate}
\end{corollary}
\begin{proof}
    \cref{wheelGoodChar1} $\Rightarrow$ \cref{wheelGoodChar2}: this is \cref{cor:double-wheel-totally-nested}.

    $\neg$ \cref{wheelGoodChar1} $\Rightarrow$ $\neg$ \cref{wheelGoodChar2}.
    Assume that $G_t$ is bad.
    Then $(U_t,W_t)$ fails to be a tetra-separation, unless $G_t$ and both neighbouring bags are $K_2$'s, by \cref{cor:double-wheel-totally-nested}.
    In this case, both vertices of $G_t$ lie in $U_t\sm W_t$, so $(U_t,W_t)$ is distinct from all tetra-separations in the tetra-star~$\sigma(\cO)$.
    Therefore, $(U_t,W_t)$ interlaces $\sigma(\cO)$, and so $(U_t,W_t)$ is crossed by another tetra-separation, by \cref{lem:double-wheel-interlaced-crossed}.
\end{proof}

\begin{lemma}\label{dotTorso4con}
    Let $G$ be a 4-connected graph.
    Let $\sigma$ be a star of mixed-4-separations of~$G$.
    Then the torso of $\sigma$ is 4-connected or a~$K_4$.
\end{lemma}
\begin{proof}
    We combine Lemmas 2.6.3 and 2.6.4 from~\cite{Tridecomp}.
\end{proof}

\begin{keylemma}\label{wheelTorso}
    Let $G$ be a 4-connected graph, and let $N$ denote the set of totally-nested tetra-separations of~$G$.
    Let $\sigma$ be a splitting star of~$N$ that is interlaced by two tetra-separations $(A,B)$ and $(C,D)$ of $G$ that cross with vertex-centre $Z$ of size two.
    Then the torso $\tau$ of $\sigma$ is a generalised double-wheel with centre~$Z$.
\end{keylemma}

\begin{proof}
    By the \allref{keylem:crossing}, the centre~$Z$ equals the vertex-centre and all links have size one.
    Let~$\cO$ be the Block-bagel induced by~$(A,B)$ and~$(C,D)$, which has centre~$Z$.
    By \cref{lem:double-wheel-star}, the tetra-star~$\sigma(\cO)$ is a star of totally-nested tetra-separations of~$G$.
    By \cref{lem:double-wheel-interlaced-crossed}, the tetra-star~$\sigma(\cO)$ is a splitting star of~$N$.
    Since both~$(A,B)$ and~$(C,D)$ interlace both~$\sigma$ and~$\sigma(\cO)$, we obtain $\sigma=\sigma(\cO)$.
    Hence it remains to compute the torso~$\tau$ of~$\sigma(\cO)$.
    The torso~$\tau$ is 4-connected by \cref{dotTorso4con}.
    We can obtain~$\tau$ from~$\cO$ by replacing each good bag~$G_t$ by a~$K_4$ on the vertex set $Z\cup\{u,v\}$ where $u$ and $v$ are the adhesion-vertices of~$G_t$, leaving only $K_2$-bags and triangle-bags, and contracting some edges in $K_2$-bags.
    Hence $\tau$ is a generalised wheel with centre~$Z$.
\end{proof}

\subsection{Angry block-bagels}\label{sec:WheelAngry}

Let $X$ be a graph.
We call $X$ a \defn{double-wheel} with \defn{centre $\{u,v\}$} if $X$ is 4-connected and $u,v$ are distinct vertices of $X$ such that $X-u-v$ is a cycle.
So $u$ and $v$ send edges to all vertices of $X-u-v$, and the edge $uv$ may or may not exist.
The \defn{rim-length} of $X$ then means the length of the cycle $X-u-v$.
We call $X$ a \defn{double-wheel of triangles} with \defn{centre $\{u,v\}$} if $X$ is 4-connected and $u,v$ are distinct vertices of $X$ such that $X-u-v$ has a cycle-decomposition $(O,\cG)$ with all adhesion-sets of size 1 and into triangles.
The \defn{rim-length} of $X$ means the length of the cycle~$O$.

\begin{lemma}
\label{lem:double-wheel-all-bad}
    Let~$G$ be a 4-connected graph with a block-bagel $\cO=(O,\cG)$ with centre~$Z$ such that all bags of $\cO$ are bad.
    Then either~$G$ is a double-wheel with centre~$Z$ or~$G$ is a double-wheel of triangles with centre~$Z$.
    In either case, the rim-length equals $h(\cO)$.
\end{lemma}
\begin{proof}
    Since all bags of $\cO$ are bad, all bags are either $K_2$'s or triangles.
    Moreover, if there is a bag that is a triangle, then its neighbouring bags are triangles, too.
    It follows that if some bag is a triangle, then all bags are triangles.
    In this case,~$G$ is a double-wheel of triangles with centre~$Z$.
    Then $h(\cO)$ equals the number of triangle-bags of~$\cO$, and hence equals the rim-length.
    Otherwise, every bag is a~$K_2$.
    In this case,~$G$ is a double-wheel with centre~$Z$. 
    Then $h(\cO)$ equals the number $K_2$-bags of~$\cO$, and hence equals the rim-length.
\end{proof}

\begin{lemma}
\label{lem:double-wheel-crossing-tetra}
    Let~$G$ be a double-wheel with rim-length~$\geq 4$ and centre~$Z$. 
    Then~$G$ has two tetra-separations that cross with all links of size one and centre equal to~$Z$.
\end{lemma}
\begin{proof}
    Let $O$ denote the cycle $G-Z$.
    Let $e_1,f_1,e_2,f_2$ be the first four edges on~$O$.
    Let $(A_i,B_i)$ be the tetra-separation of $G$ with separator $\{e_i,f_i\}\cup Z$.
    Then $(A_1,B_1)$ crosses $(A_2,B_2)$ with the four edges as links and with centre equal to~$Z$.
\end{proof}

\begin{lemma}
\label{lem:double-wheel-of-triangles-crossing-tetra}
    Let~$G$ be a double-wheel of triangles with rim-length~$\geq 4$ and centre~$Z$. Then~$G$ has two tetra-separations that cross with all links of size one and centre equal to~$Z$.
\end{lemma}
\begin{proof}
    Since~$G$ is a double-wheel of triangles with rim-length~$\geq 4$, the graph $G-Z$ has a cycle-decomposition $\cO=(O,(G_t)_{t \in O})$ with $|O| \geq 4$ such that all bags are triangles.
    Let $(A,B)$ and $(C,D)$ be crossing 2-separations of~$O$.
    Let $\hat A:=Z\cup\bigcup_{t\in A}V(G_t)$ and define $\hat B,\hat C,\hat D$ similarly.
    Then $(\hat A,\hat B)$ and $(\hat C,\hat D)$ are tetra-separations of $G$ that cross as desired.
\end{proof}

\begin{keylemma}
\label{keylem:double-wheel-angry}
    Let~$G$ be a 4-angry graph. Then the following assertions are equivalent:
    \begin{enumerate}
        \item\label{itm:double-wheel-angry-1} $G$ has two tetra-separations that cross with all links of size one;
        \item\label{itm:double-wheel-angry-2} $G$ has a block-bagel~$\cO$ with $h(\cO) \geq 4$ all whose bags are bad;
        \item\label{itm:double-wheel-angry-3} $G$ is a double-wheel with rim-length~$\geq 4$ or a double-wheel of triangles with rim-length~$\geq 4$.
    \end{enumerate}
\end{keylemma}
\begin{proof}
    \ref{itm:double-wheel-angry-1} $\Rightarrow$ \ref{itm:double-wheel-angry-2}. 
    Assume that~$G$ has two tetra-separations~$(A,B)$ and~$(C,D)$ that cross with all links of size one.
    Let~$Z$ be the centre of~$(A,B)$ and~$(C,D)$, which is equal to the vertex-centre and has size two, by the \allref{keylem:crossing}.
    Let~$\cO$ be the block-bagel induced by~$(A,B)$ and~$(C,D)$.
    Then we have the \allref{set:block-bagel}.
    By the \allref{lem:helpful-lemma}, we have~$h(\cO) \geq 4$.
    If $\cO$ has a good bag, then $G$ has a totally-nested tetra-separation by \cref{cor:double-wheel-totally-nested}, contradicting that~$G$ is 4-angry.
    Hence, every bag of~$\cO$ is bad.
    
    \ref{itm:double-wheel-angry-2} $\Rightarrow$ \ref{itm:double-wheel-angry-3}. Follows from \cref{lem:double-wheel-all-bad}.

    \ref{itm:double-wheel-angry-3} $\Rightarrow$ \ref{itm:double-wheel-angry-1}. Follows from \cref{lem:double-wheel-crossing-tetra} and \cref{lem:double-wheel-of-triangles-crossing-tetra}.
\end{proof}

\section{Cycles of 2-connected torsos}\label{sec:CycleOfGraphs}

One of the outcomes of \cref{MainDecomp} is a cycle of triangle-torsos and 3-connected torsos on $\le 5$ vertices.
To investigate how this outcome arises, we assume that we are given two tetra-separations $(A,B)$ and $(C,D)$ of a 4-connected graph $G$ that cross with empty vertex-centre.
In order to construct a cycle-decomposition for analysing the cyclic structure that is determined by $(A,B)$ and $(C,D)$, we first use the corner-diagram to divide $G$ into four smaller pieces (\cref{sec:BagelBananas}).
We then apply the Tutte-decomposition to all four pieces and merge them to obtain a cycle-decomposition of $G$ (\cref{sec:BagelTutteBanana}).
In \cref{sec:BagelSettings}, we summarise assumptions shared by a fair share of the lemmas to come.
\cref{sec:BagelNpaths} provides machinery that exploits the structure of the cycle-decomposition to find independent paths.
We use this machinery in \cref{sec:BagelTotallyNested} to characterise which bags of the cycle-decomposition give rise to totally-nested tetra-separations (\cref{cor:bagel-totally-nested}).
\cref{sec:BagelSplitting} shows that these totally-nested tetra-separations form a splitting star as in \cref{MainDecomp}, and determines the structure of its torso (\cref{bagelTorso}).
Finally, \cref{sec:BagelAngry} characterises the special case where additionally all tetra-separations of the graph are crossed.

\subsection{Bananas}\label{sec:BagelBananas}

Let $G$ be a 4-connected graph.
Assume that two tetra-separations $(A,B)$ and $(C,D)$ of $G$ cross with all links of size two.
Recall that the centre is empty by the \allref{keylem:crossing}.
We obtain the graph $\defnMath{G^*}$ from $G$ by adding edges, as follows.
For each $L$-link $\{x,y\}$ and each adjacent corner~$K=LS$ (where $L,S$ are sides), we join the unique vertices in $\hat x\cap S$ and $\hat y\cap S$ by an edge, which we call the \defn{$K$-edge for the $L$-link}; see \cref{fig:GStarDef}.
We will only use $G^*$ when $(A,B)$ and $(C,D)$ are clear from context.
The \defn{expanded corner for $AC$} is obtained from the corner $(A\sm B)\cap (C\sm D)$ for $AC$ by adding the endvertices of all edges contained in the two adjacent links.
The \defn{banana for $AC$}, or \defn{$AC$-banana}, is the subgraph of $G^*$ induced by the expanded corner for~$AC$.
The expanded corners and bananas are defined similarly for the other three corners.

\begin{figure}[ht]
    \centering
    \includegraphics[height=10\baselineskip]{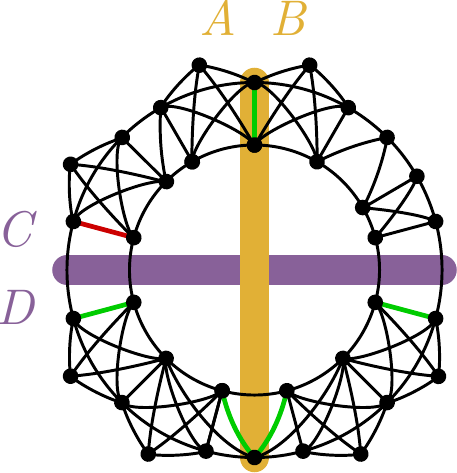}
    \caption{The edges added to $G$ to obtain $G^*$ are coloured red and green. The red edge is the $AC$-edge for the $A$-link.}
    \label{fig:GStarDef}
\end{figure}

\begin{lemma}
\label{lem:size-expanded-corner}
    Let $G$ be a 4-connected graph.
    Let $(A,B)$ and $(C,D)$ be tetra-separations of $G$ that cross with all links of size two.
    Let $K$ be a corner.
    Then the expanded corner for~$K$ has size~$\ge 4$.
\end{lemma}
\begin{proof}
    Assume $K=AC$ without loss of generality.
    Let~$L$ be the union of the $A$-link and the $C$-link, and note that $|L| = 4$.
    For each element~$x$ in~$L$, we define~$v_x:=x$ if~$x$ is a vertex, and $v_x:=u$ if~$x$ is an edge and~$u$ is the endvertex of~$x$ in $(B \sm A) \cup (D \sm C)$.
    By the matching-condition for~$(A,B)$ and~$(C,D)$, $x \mapsto v_x$ is an injection that maps elements in~$L$ to vertices of the expanded corner for~$AC$.
    So, the expanded corner for~$AC$ has size~$\ge 4$.
\end{proof}

\begin{lemma}\label{bananaEndNumberChar}
    Let $G$ be a 4-connected graph.
    Let $(A,B)$ and $(C,D)$ be tetra-separations of $G$ that cross with all links of size two.
    Let $K$ be a corner.
    Then the following assertions are equivalent:
    \begin{enumerate}
        \item\label{bananaEndNumberChar1} the $K$-edges for the links adjacent to $K$ are distinct;
        \item\label{bananaEndNumberChar2} the corner $K$ is not potter.
    \end{enumerate}
\end{lemma}
\begin{proof}
    $\neg$ \cref{bananaEndNumberChar1} $\Rightarrow$ $\neg$ \cref{bananaEndNumberChar2}.
    Say $K$ is the corner for $AC$.
    Let $\beta$ denote the $AC$-banana.
    Let $e(A)$ denote the $AC$-edge for the $A$-link, and let $e(C)$ denote the $AC$-edge for the $C$-link.
    Assume that $e(A)=e(C)=:e$.
    We will show that $AC$ is potter.
    If both endvertices of $e$ lie in the corner for $AC$, then each link adjacent to $AC$ consists of a matching of size two with the endvertices of $e$ as one side of the matching; see \cref{fig:bananaEndNumberChar}.
    As $G$ is 4-connected, the corner for $AC$ cannot contain any other vertices besides the endvertices of~$e$.
    But then both endvertices of $e$ have degree $\le 3$ in~$G$, contradicting that $G$ is 4-connected.
    So we may assume that $e$ has an endvertex $v$ that lies in the $A$-link, say.
    Then the $C$-link contains an edge $f$ that ends in~$v$.
    Hence $f$ dangles from the $A$-link through the $C$-link, which implies that the corner for $AC$ is potter by the \allref{keylem:crossing}.

    $\neg$ \cref{bananaEndNumberChar2} $\Rightarrow$ $\neg$ \cref{bananaEndNumberChar1} holds by definition.
\end{proof}

\begin{figure}[ht]
    \centering
    \includegraphics[height=8\baselineskip]{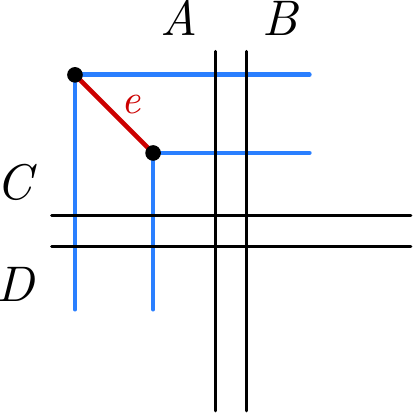}
    \caption{The situation in the proof of \cref{bananaEndNumberChar}}
    \label{fig:bananaEndNumberChar}
\end{figure}

\begin{lemma}\label{banana2con}
    Let $G$ be a 4-connected graph.
    Let $(A,B)$ and $(C,D)$ be tetra-separations of $G$ that cross with all links of size two.
    The bananas of the four corners are 2-connected.
\end{lemma}
\begin{proof}
    Consider the corner for $AC$, say, and let $\beta$ denote its banana.
    If $AC$ is potter, then $\beta$ is a union of two triangles sharing an edge, and hence $\beta$ is 2-connected.
    So we may assume that $AC$ is not potter.
    Then $\beta$ has $\ge 3$ vertices by \cref{lem:size-expanded-corner}.
    Suppose for a contradiction that $\beta$ has a $(\le 1)$-separation $(X,Y)$.
    Without loss of generality, the $AC$-edge $e(A)$ for the $A$-link lies in $\beta[X]$.
    Then the $AC$-edge $e(C)$ for the $C$-link lies in $\beta[Y]$, since $G$ is 4-connected; see \cref{fig:banana2con}.
    Let $\hat X:=X\cup D$ and $\hat Y:=Y\cup (B\cap C)$.
    Then $S(\hat X,\hat Y)=S(X,Y)\cup (B\text{-link})$.
    So $(\hat X,\hat Y)$ is a mixed-3-separation of~$G^*$.
    Hence $(\hat X,\hat Y)$ witnesses that $G^*$ is not 4-connected by \cref{kConMixed}, contradicting that $G^*$ is spanned by the 4-connected subgraph~$G$.
\end{proof}

\begin{figure}[ht]
    \centering
    \includegraphics[height=8\baselineskip]{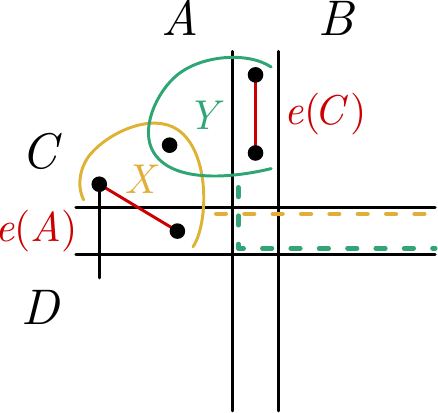}
    \caption{%
    The construction of $(\hat X,\hat Y)$ in the proof of \cref{banana2con}.
    The extensions of $X$ to $\hat X$ and of $Y$ to $\hat Y$ are indicated by the dashed lines.
    }%
    \label{fig:banana2con}
\end{figure}

\subsection{Tutte-bananas and Tutte-bagels}\label{sec:BagelTutteBanana}
A \defn{Tutte-bagel} of $G$ is a cycle-decomposition $(O,\cG)$ of $G$ with $\cG=(G_t:t\in V(O))$ that satisfies all of the following conditions:
\begin{enumerate}
    \item every torso is 3-connected, a 4-cycle or a triangle;
    \item all adhesion sets have size two;
    \item whenever two triangle-torsos are neighbours, the adhesion-graph in between them is a~$K_2$.
\end{enumerate}

\begin{lemma}\label{TutteBagelObservations}
    Let $G$ be a 4-connected graph, and let $(O,\cG)$ be a Tutte-bagel of~$G$.
    \begin{enumerate}
        \item If a torso is 3-connected or a 4-cycle, then its adhesion sets are disjoint.
        \item If a torso is a triangle, then its adhesion sets are distinct but share exactly one vertex.
        \item If a torso is a 4-cycle, then the two neighbouring torsos must be 3-connected.\qed
    \end{enumerate}
\end{lemma}

Let $G$ be a graph.
A 2-separation $(A,B)$ of $G$ is \defn{totally-nested} if it is nested with every 2-separation of~$G$.
This is equivalent to saying that $(A,B)$ is half-connected and at least one of $G[A]$ and $G[B]$ is 2-connected \cite[Corollary~3.3.3]{Tridecomp}.
Assume now that $G$ is 2-connected and $N$ denotes its set of totally-nested 2-separations.
Then the tree-decomposition $\cT=\cT(N)$ defined by $N$ is the \defn{Tutte-decomposition} of $G$.
Every torso of $\cT$ is 3-connected, a cycle or a $K_2$.
Tutte introduced his decomposition in \cite{TutteCon}.
A modern proof is included in the appendix of \cite{Tridecomp}.
The Tutte-decomposition $\cT=(T,\cV)$ has the following two properties:
\begin{enumerate}
    \item Every adhesion-set has size two.
    \item Every adhesion-set between two neighbouring cycle-torsos is spanned by an edge in~$G$.
    \item Every $K_2$-torso neighbours at least three torsos, none of which are $K_2$'s.
\end{enumerate}
These properties give an alternative characterisation of the Tutte-decomposition.

\begin{definition}[Tutte-bananas]
    Let $G$ be a 4-connected graph.
    Let $(A,B)$ and $(C,D)$ be tetra-separations of $G$ that cross with all links of size two.
    Let $K$ be one of the four corners.
    Recall that the $K$-banana is 2-connected by \cref{banana2con}.
    The \defn{$K$-Tutte-banana} is the Tutte-decomposition of the $K$-banana.
    The \defn{$K$-Tutte-path} is the decomposition-tree of the $K$-Tutte-banana.
\end{definition}

\begin{lemma}\label{TutteBananaIsBanana}
    Tutte-bananas are path-decompositions all whose torsos are either 3-connected graphs, 4-cycles or triangles.
\end{lemma}
\begin{proof}
    This follows from $G$ being 4-connected.
\end{proof}

\begin{lemma}\label{linkRepresentedByEqualTorsos}
    Let $G$ be a 4-connected graph.
    Let $(A,B)$ and $(C,D)$ be tetra-separations of $G$ that cross with all links of size two.
    Let $K$ be one of the four corners, and consider an adjacent $L$-link $\Lambda$.
    Let $(T,\cV)$ be the $K$-Tutte-banana.
    \begin{enumerate}
        \item\label{linkRepresentedByEqualTorsos1} If $\Lambda$ consists of two vertices, then there is a unique node $t$ of $T$ with $\Lambda\se V_t$.
        Moreover, the torso of $V_t$ is 3-connected, and $t$ is an end-node of the path~$T$.
        \item\label{linkRepresentedByEqualTorsos2} If $\Lambda$ contains an edge, then there is unique node $t$ of $T$ with $\hat \Lambda=V_t$.
        Moreover, the torso of $V_t$ is a 4-cycle or a triangle, and $t$ is an end-node of the path~$T$.
    \end{enumerate}
\end{lemma}
\begin{proof}
    Say $K$ is the $AC$-corner and $\Lambda$ is the $A$-link.

    \cref{linkRepresentedByEqualTorsos1}. 
    Let $u,v$ be the two vertices in $\Lambda$.
    Recall that $uv$ is an edge in $G^*$.
    So there is at least one bag $V_t\in\cV$ that contains both $u$ and~$v$.
    Let $V_s\in\cV$ contain the two ends of the $AC$-edge for the $C$-link, and choose $V_s$ to minimise the distance from $s$ to $t$ in~$T$.
    Recall that $T$ is a path by \cref{TutteBananaIsBanana}.
    If $t$ is not an end of $T$, there is $V_r\in\cV$ with $t\in \mathring{r}Ts$ and $V_r\sm V_t$ non-empty, so every vertex in $V_r\sm V_t$ is separated from a vertex outside the $AC$-banana by the adhesion-set of $rt$ which has size two, contradicting 4-connectivity.
    Hence $t$ is an end of~$T$.
    Similarly, either $u$ or $v$ avoids the adhesion-set of $V_t$, say $u$.
    If the torso of $V_t$ is a cycle, then $u\in S(C,D)$ violates the degree-condition.
    Thus the torso is 3-connected.

    \cref{linkRepresentedByEqualTorsos2}.
    Since $\Lambda$ contains an edge~$e$, the subgraph $H$ of $G^*$ induced by $\hat \Lambda$ is either a triangle or a 4-cycle, depending on whether the other element of $\Lambda$ besides $e$ is a vertex $v$ or an edge $f$, respectively.
    If $H$ is a triangle, then $v\in D\sm C$ has degree two in the $AC$-banana.
    Hence $H$ is a leaf-bag of $(T,\cV)$, and no other bag of $(T,\cV)$ includes~$\Lambda$.
\end{proof}

The node $t$ of $T$ as in \cref{linkRepresentedByEqualTorsos} is said to \defn{represent} the $L$-link in $T$, and in the $K$-Tutte-banana.

\begin{lemma}\label{TutteBananaLinksRepresentedDifferently}
    Let $G$ be a 4-connected graph.
    Let $(A,B)$ and $(C,D)$ be tetra-separations of $G$ that cross with all links of size two.
    Let $K$ be one of the four corners, and let $(T,\cV)$ be its Tutte-banana.
    The two links adjacent to $K$ are represented by distinct end-nodes of $T$, unless $T$ has only one node.
    In the latter case, both links consist of vertices only, and the unique torso of $(T,\cV)$ is 3-connected.
\end{lemma}
\begin{proof}
    Without loss of generality, $K$ is the corner $AC$.
    Both the $A$-link and the $C$-link are represented by end-nodes $a$ and $c$ of $T$, respectively, by \cref{linkRepresentedByEqualTorsos}.
    Assume now that $a=c$.
    Then $T$ has no edge, since otherwise the adhesion-set of this edge would form a 2-separator of $G$, contradicting 4-connectivity.
    Hence $|T|=1$, and in particular the corner $AC$ is not potter.

    We claim that both the $A$-link and the $C$-link consist of vertices only.
    Indeed, if the $A$-link contains an edge, then the $AC$-edge $e$ for the $A$-link is distinct from the $AC$-edge for the $C$-link by \cref{bananaEndNumberChar}, so the ends of $e$ together form a 2-separator of the $AC$-banana.
    This 2-separator gives rise to a 2-separation of the $AC$-banana that cuts off the triangle-torso or 4-cycle-torso at $a$ by \cref{linkRepresentedByEqualTorsos}, and this 2-separation is totally-nested as its separator is spanned by the edge $e$.
    Hence the 2-separation corresponds to an edge of $T$, a contradiction.
\end{proof}

\begin{definition}[Tutte-bagel induced by crossing tetra-separations]
    Let $G$ be a 4-connected graph.
    Let $(A,B)$ and $(C,D)$ be tetra-separations of $G$ that cross with all links of size two.
    For each corner $K$ let $\cT^K=(T^K,\cV^K)$ denote the $K$-Tutte-banana.
    Recall that $T^K$ is a path by \cref{TutteBananaIsBanana}.
    
    We obtain $O$ from the disjoint union $\bigsqcup_K T^K$, where $K$ ranges over the four corners, by doing the following for each link~$L$:
    for both corners $K$ and $K'$ adjacent to $L$ we consider the nodes $t$ and $t'$ that represent the $L$-link in $T^K$ and $T^{K'}$, respectively, and
    \begin{itemize}
        \item we join the nodes $t$ and $t'$ by an edge $\defnMath{o(L)}$ in $O$ if the torsos associated with these nodes in their respective Tutte-bananas are 3-connected;
        \item we otherwise identify the nodes $t$ and $t'$ to a single node $\defnMath{o(L)}$ of~$O$.
    \end{itemize}
    We say that $o(L)$ \defn{represents} the $L$-link in~$O$.
    For each node $t$ of $O$ that is not of the form $o(L)$, we define $G_t:=G[V^K_t]$ where $K$ is the unique corner with $t\in T^K$.
    For each node of $O$ of the form $o(L)$ for some link~$L$, we consider the two nodes $t\in T^K$ and $t'\in T^{K'}$ that were identified to obtain~$o(L)$, and we let $G_{o(L)}:=G[V^K_t]$ which is independent of $K$ by \cref{linkRepresentedByEqualTorsos}.
    We refer to the pair $(O,(G_t:t\in V(O))$ as the \defn{Tutte-bagel induced by $(A,B)$ and $(C,D)$}.
    It remains to show that this really is a Tutte-bagel.
\end{definition}

\begin{lemma}
\label{lem:really-is-Tutte-bagel}
    Let $G$ be a 4-connected graph.
    Let $(A,B)$ and $(C,D)$ be tetra-separations of $G$ that cross with all links of size two.
    The Tutte-bagel induced by $(A,B)$ and $(C,D)$ really is a Tutte-bagel.
\end{lemma}

\begin{proof}
Let $\cO=(O,(G_t : t\in V(O))$ denote the Tutte-bagel induced by $(A,B)$ and $(C,D)$.
First, we show that $\cO$ is a cycle-decomposition.
That $O$ is a cycle follows from \cref{TutteBananaLinksRepresentedDifferently}.

\cref{GraphDec1}.
Recall that there are no jumping edges and no diagonal edges by the \allref{keylem:crossing}.
Every vertex and edge of $G$ is included in some expanded corner, and hence is included in the bag of some Tutte-banana.
Hence every vertex and edge of $G$ appears in some~$G_t$.

\cref{GraphDec2}.
Let $v\in G$ be a vertex and consider $O_v:=O[\,\{t\in V(O):v\in G_t\}\,]$.

If $v$ avoids $\hat\Lambda$ for every link $\Lambda$, then $v$ lies in a unique banana, and so $O_v$ is a subpath of the decomposition-path of the corresponding Tutte-banana.
If $v$ lies in $\hat\Lambda$ for precisely one $L$-link $\Lambda$, then $v$ lies in the bananas for the two corners adjacent to~$L$, and so $O_v$ is a path.
Otherwise $v$ lies in $\hat\Lambda_1\cap\hat\Lambda_2$ for two adjacent $L_i$-links~$\Lambda_i$.
Then $v$ lies in the corner $K$ in between the two links.
Then $v$ lies in all bags of the $K$-Tutte-banana, and in the unique bags at the nodes representing the $L_i$-links in the Tutte-bananas of the corners adjacent to~$K$.

Let $H_t$ be a torso of~$\cO$.
If $t$ represents some link, then $H_t$ is a 4-cycle or a triangle.
Otherwise $H_t$ stems from a Tutte-banana. Then $H_t$ is 3-connected, a 4-cycle or a triangle, by \cref{TutteBananaIsBanana}.

Since every adhesion-set between two neighbouring triangle-torsos of a Tutte-banana is spanned by an edge in~$G$, the same is true for~$\cO$.
\end{proof}

\begin{lemma}
    \label{TutteBagelAdhesionSetNote}
    Let $G$ be a 4-connected graph.
    Let $(A,B)$ and $(C,D)$ be tetra-separations of $G$ that cross with all links of size two.
    Let $\cO$ be the Tutte-bagel induced by $(A,B)$ and $(C,D)$.
    Every $K$-edge for an $L$-link spans a unique adhesion-set of~$\cO$.\qed
\end{lemma}

\subsection{Settings}\label{sec:BagelSettings}

The following settings are supported by \cref{fig:BagelBag}.

\begin{setting}[Tutte-bagel setting]\label{setting:TutteBagel}
    Let $G$ be a 4-connected graph.
    Let $(A,B)$ and $(C,D)$ be tetra-separations of $G$ that cross with all links of size two.
    Let $\cO=(O,(G_t:t\in V(O)))$ be the Tutte-bagel induced by $(A,B)$ and $(C,D)$, with torsos being denoted by~$H_t$.
    Write $O=:t(0),t(1),\ldots,t(m),t(0)$ with indices in $\Z_{m+1}$.
    Assume that the representatives of the links appear in this enumeration in the order $o(C),o(B),o(D),o(A)$.
    For each node $t$ of $O$ let $U'_t:=V(G_t)$ and $W'_t:=\bigcup\,\{\,V(G_s):s\in V(O-t)\,\}$.
    Let $(U_t,W_t)$ denote the right-left-reduction of $(U'_t,W'_t)$.
\end{setting}

\begin{setting}[Tutte-bagel torso setting]\label{BagelBag}
    Assume the \allref{setting:TutteBagel}.
    Assume that the torso $H_{t(0)}$ of $\cO$ is 3-connected; so $H_{t(0)}$ stems from the $AC$-banana.    
    Let $\{x'_1,x'_2\}$ be the adhesion set of $t(-1) t(0)$ and let $\{y'_1,y'_2\}$ be the adhesion set of $t(0) t(1)$.
    Let $(U',W'):=(U'_{t(0)},W'_{t(0)})$ and $(U,W):=(U_{t(0)},W_{t(0)})$.
    Note that $S(U',W')=\{x'_1,x'_2,y'_1,y'_2\}$.
    Let $x_i\in S(U,W)$ be equal or incident to~$x'_i$, and let $x^W_i$ be the unique vertex in $\hat x_i \cap W$.
    Define $y_i$ and $y^W_i$ similarly.
\end{setting}

\begin{figure}[ht]
    \centering
    \includegraphics[height=12\baselineskip]{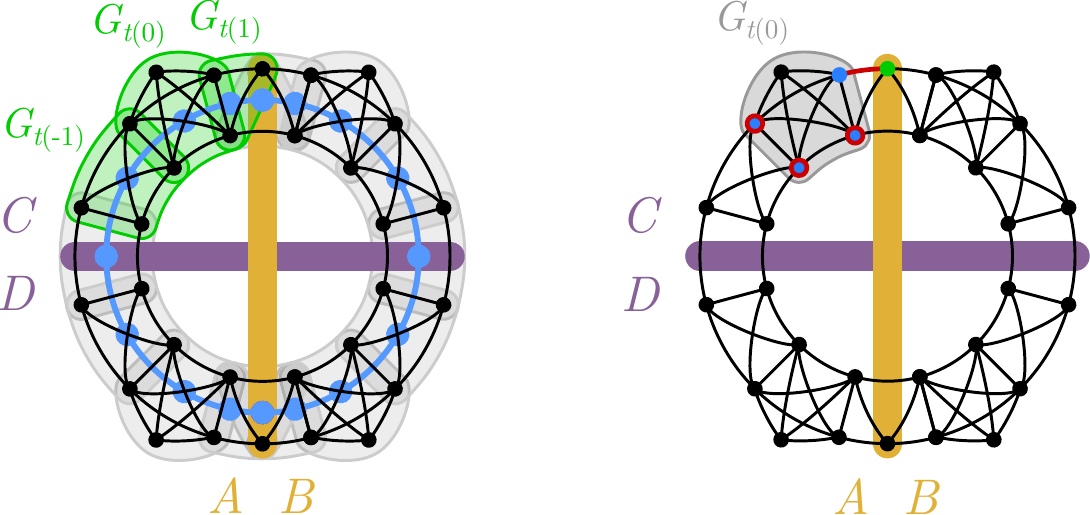}
    \caption{Left: The \allref{setting:TutteBagel}.
    Right: the \allref{BagelBag}.
    The separator $S(U',W')$ is blue, while $S(U,W)$ is red.
    We have $x^W_i=x'_i$ for both $i=1,2$ and $y^W_1=y'_1$ (say), so these three vertices are red.
    The vertex $y^W_2$ is green.}
    \label{fig:BagelBag}
\end{figure}

Let $\cO$ be a Tutte-bagel of a 4-connected graph $G$.
Let $H_t$ be a torso of $\cO$.
Assume first that $H_t$ is a 4-cycle.
Then $H_t$ has exactly two edges whose endvertices lie in distinct adhesion-sets of $H_t$.
We call these edges the \defn{free edges} of $H_t$.
Assume now that $H_t$ is a triangle.
Then the adhesion-sets of $H_t$ share exactly one vertex $v$.
We call $v$ the \defn{tip} of the triangle $H_t$.
There is a unique edge of $H_t$ whose endvertices lie in distinct adhedion-sets of $H_t$.
We call this edge the \defn{free edge} of the triangle~$H_t$.

\begin{lemma}\label{BagelBagAdhesionSet}
    Assume the \allref{BagelBag}.
    The vertex-set $\{y^W_1,y^W_2\}$ is the adhesion-set at $t(0)\, t(1)$ if the torso $H_{t(1)}$ is 3-connected, and it is the adhesion-set at $t(1)\, t(2)$ otherwise.
    In the latter case, the torso at $t(2)$ is not a 4-cycle.
\end{lemma}
\begin{proof}
    The former case is straightforward, so it remains to consider the case where $H_{t(1)}$ is not 3-connected.
    If the torso $H_{t(1)}$ is a 4-cycle, then the $y_i$ are the free edges of $H_{t(1)}$, and so $\{y^W_1,y^W_2\}$ is the adhesion-set at $t(1) t(2)$.
    Hence we may assume that the torso $H_{t(1)}$ is a triangle.
    The next torso $H_{t(2)}$ cannot be 4-cycle by \cref{TutteBagelObservations}, so it is 3-connected or a triangle.
    
    Assume first that $H_{t(2)}$ is 3-connected.
    Since $H_{t(2)}$ shares a vertex with $H_{t(1)}$, say $y'_1$, the vertex-set of $H_{t(2)}$ must be included in~$C$.
    Hence the $\ge 2$ neighbours of $y'_1$ in the bag $G_{t(2)}$ lie in~$W'\sm U'$.
    Thus $y^W_1=y'_1$.
    Since $H_{t(1)}$ is a triangle, $y_2$ is its free edge, so $\{y^W_1,y^W_2\}$ is the adhesion-set at~$t(1)\, t(2)$.
    
    Otherwise it is a triangle, so the adhesion-graph at $t(1)\, t(2)$ is a $K_2$-subgraph of~$G$.
    Then the $y'_i$ incident to the free edge of the triangle $H_{t(1)}$ corresponds to an edge~$y_i$.
    The $y'_j$ that is the tip of $H_{t(1)}$ has $y^W_i$ as a neighbour and is incident to the free edge of $H_{t(2)}$, so $y'_j=y_j$.
    Hence $\{y^W_1,y^W_2\}$ is the adhesion-set at $t(1)\, t(2)$.
\end{proof}

Recall that $O_v=O[\,\{t\in V(O):v\in G_t\}\,]$ for a vertex $v\in V(G)$ in the context of a Tutte-bagel~$\cO$ of~$G$.

\begin{lemma}\label{trianglePath}
    Let $G$ be a 4-connected graph, and let $\cO=(O,\cG)$ be a Tutte-bagel of $G$.
    Then both of the following assertions hold:
    \begin{enumerate}
        \item\label{trianglePath2} Every two triangle-torsos of $\cO$ have distinct tips.
        \item\label{trianglePath1} Let $P \subseteq O_v$ be a path with ends~$r \neq t$ such that $H_r,H_t$ are triangles. Then $P=rt$ or $P=rst$ where $H_s$ is a triangle. 
    \end{enumerate} 
\end{lemma}

\begin{proof}
    \cref{trianglePath2}. 
    Suppose for a contradiction that two triangle-torsos of $\cO$ have the same tip.
    Say the torsos are~$H_r$ and~$H_t$, and~$v$ is their shared tip.
    Let~$e$ be the free edge of~$H_r$, and let~$e'$ be the free edge of~$H_t$.
    Then $\{v,e,e'\}$ is a mixed-3-separator, contradicting 4-connectivity by \cref{kConMixed}.
    
    \cref{trianglePath1}.
    First, assume that there is a node $p$ on $P$ other than $r,t$ with $v\in G_p$ such that the torso $H_p$ is a 4-cycle or 3-connected.
    That the adhesion-sets at $G_p$ are disjoint then contradicts that $v$ lies in all three bags $G_r,G_p,G_t$.
    Otherwise, all torsos of all nodes on~$P$ are triangles.
    Let~$e$ be the edge of~$P$ incident with~$r$, and let~$e'$ be the edge of~$P$ incident with~$t$.
    By~\ref{trianglePath2}, if $|P| \geq 4$ then the adhesion sets of~$e$ and~$e'$ are disjoint, which contradicts that~$v$ lies in all bags of nodes in~$P$ by $P\se O_v$.
    Otherwise $|P| \leq 3$ and we get the statement.
\end{proof}

\subsection{N-Paths lemmata}\label{sec:BagelNpaths}

\begin{lemma}\label{twoDisjointPathsOutsideBagel}
    Let $G$ be a 4-connected graph with a Tutte-bagel $\cO=(O,\cG)$ and decomposition-cycle $O=t(0),\ldots t(m),t(0)$.
    Let $0\le i< j\le m$.
    Let $\pi_\ell$ denote the adhesion-set of $t(\ell)t(\ell+1)$ for all relevant indices~$\ell$.
    There exist two disjoint $\pi_i$--$\pi_j$ paths in $G_{t(i+1)}\cup\cdots\cup G_{t(j)}$.
\end{lemma}
\begin{proof}
    For each torso $H_{t(\ell)}$ we use that it is 2-connected to find two disjoint paths $P_\ell,Q_\ell$ between its two adhesion-sets.
    These paths use no torso-edges, so they are subgraphs of their bags and of~$G$.
    Hence combining the paths in the bags $G_{t(i+1)},\ldots, G_{t(j)}$ yields a pair of desired $\pi_i$--$\pi_j$ paths.
\end{proof}

\begin{lemma}\label{3conAdhesionNoPathK4}
    Let $\cO$ be a Tutte-bagel of a 4-connected graph~$G$.
    Let $H_t$ be a 3-connected torso of $\cO$ with adhesion sets $\pi$ and $\pi'$.
    The two vertices in $\pi$ are linked by a path in $G_t-\pi'$, unless $H_t=K_4$.
\end{lemma}
\begin{proof}
    Assume that the two vertices $x_1,x_2$ in $\pi$ lie in distinct components $C_i$ of $G_t-\pi'$.
    Then $V(C_i)=\{x_i\}$ and $G_t-\pi'=C_1\cup C_2$, as $G$ is $4$-connected.
    Hence $G_t$ has only 4 vertices.
    Since $H_t$ is 3-connected, it is a $K_4$.
\end{proof}

\begin{lemma}\label{BagelLinkForm}
    Assume the \allref{setting:TutteBagel}.
    Let $K$ be a corner.
    If the $K$-Tutte-banana has no 3-connected torsos, then all its torsos are triangles, all its adhesion-graphs are $K_2$-subgraphs of~$G$, and the two links adjacent to $K$ are represented by triangle-torsos in the Tutte-banana.
\end{lemma}
\begin{proof}
    Since the Tutte-banana $\cT$ for $K$ has no 3-connected torsos, it also has no 4-cycles as torsos by \cref{TutteBagelObservations}.
    Hence all torsos of $\cT$ are triangles; in particular, both links adjacent to $K$ are represented by triangle-torsos of~$\cT$.
    Again by \cref{TutteBagelObservations}, it follows that all adhesion-graphs of $\cT$ are $K_2$-subgraphs of~$G$.
\end{proof}

\begin{lemma}[One-Path Lemma]\label{OnePathBagel}
    Assume the \allref{setting:TutteBagel}.
    Let $a_1 a_2$ be the $AD$-edge for the $A$-link.
    Then there is an $a_1$--$a_2$ path $P$ in $G[A\cap D]$ that uses at most one vertex in the $D$-link.
    Moreover, if $AD$ is not potter, then for each vertex $d$ in the $D$-link there is such a path $P$ that avoids~$d$.
    And if both $a_1,a_2$ lie in the $A$-link, then $P$ does not use the edge $a_1 a_2$ (if it exists).
\end{lemma}
\begin{proof}
    If the corner $AD$ is potter, then $a_1 a_2$ is an edge in $G$, and the path $a_1 a_2$ is as desired.
    So we may assume that the corner $AD$ is not potter.
    Let $\cT$ denote the $AD$-Tutte-banana.
    We distinguish two cases.

    Suppose first that $\cT$ has a 3-connected torso $H_t$.
    We find two disjoint $a_i$--$H_t$ paths $P_i$ for $i=1,2$ in the $AD$-banana using \cref{twoDisjointPathsOutsideBagel}.
    We even get that both $P_i$ are subgraphs of~$G$.
    Note that the ends of the $P_i$ in $H_t$ together form an adhesion-set.
    If $H_t$ is a $K_4$, then we link the ends of the $P_i$ in the bag $G_t$ by a path that avoids any prescribed vertex $d$ in the $D$-link and that uses no torso-edges.
    Otherwise $H_t$ is not a $K_4$, so we may use \cref{3conAdhesionNoPathK4} to link the ends of the $P_i$ in the bag $G_t$ with a path that avoids all vertices in the $D$-link and that uses no torso-edges.

    Suppose second that $\cT$ has no 3-connected torsos.
    Then all torsos of $\cT$ are triangles, both the $A$-link and the $D$-link are represented by triangle-torsos, and $a_1 a_2\in E(G)$ by \cref{BagelLinkForm}.
    In particular, exactly one of $a_1,a_2$ lies in the $A$-link, say it is~$a_1$.
    Now $a_2$ cannot lie in the $D$-link, since otherwise some edge dangles from the $D$-link through the $A$-link which by the \allref{keylem:crossing} implies that the corner $AD$ is potter, a contradiction.
    Hence the path $a_1 a_2$ is as desired.
\end{proof}

\begin{lemma}[Two-Paths Lemma]\label{TwoPathsBagel}
    Assume the \allref{BagelBag}.
    Let $\ell\ge 1$ be maximum such that $t(\ell)$ stems from the BD-Tutte-banana.
    There exist two independent $y^W_1$--$y^W_2$ paths $P_1,P_2$ in $G_{t(1)}\cup\cdots\cup G_{t(\ell)}$ not using the edge $y'_1 y'_2$ (if it exists) such that:
    \begin{enumerate}
        \item $P_1\cup P_2$ uses at most one vertex of the $D$-link.
        \item If the corner $BD$ is not potter, then for every vertex $d$ in the $D$-link we can choose the two paths so that $d\notin P_1\cup P_2$.
    \end{enumerate}
\end{lemma}

\begin{proof}
    By \cref{BagelBagAdhesionSet}, there is $i\in \{1,2\}$ such that $\{y^W_1,y^W_2\}$ is the adhesion-set at $t(i-1) t(i)$.
    Moreover, the torso $H_{t(i)}$ is not a 4-cycle.
    We consider two cases, supported by \cref{fig:BagelBagTwoPaths}.

    Suppose first that the torso at $t(i)$ is 3-connected.
    Then we find two $y^W_1$--$y^W_2$ paths $Q_1,Q_2$ in the torso at $t(i)$ that avoid the torso-edge $y^W_1 y^W_2$.
    Note that $G_{t(i)}\se G[C]$.
    If neither $Q_j$ uses the other torso-edge $z_1 z_2$, then we are done.
    So we may assume that $Q_1$, say, uses $z_1 z_2$.
    Let $e=u_1 u_2$ be the $BD$-edge for the $B$-link.
    Let $k\ge i$ be maximum such that $t(k)$ stems from the $BC$-Tutte-banana.
    By \cref{twoDisjointPathsOutsideBagel}, we find two disjoint $\{z_1,z_2\}$--$\{u_1,u_2\}$ paths $R_1,R_2$ in $G_{t(i+1)}\cup\cdots\cup G_{t(k)}$, taking the trivial paths if $k=i$.
    We apply the \allref{OnePathBagel} (with $B$ in place of $A$) to the $BD$-edge $u_1 u_2$ for the $B$-link to find a $u_1$--$u_2$ path $R_3$ in $G[B\cap D]$ that uses at most one vertex in the $D$-link.
    Moreover, if $BD$ is not potter, then we find $R_3$ so that it avoids any prescribed vertex $d$ in the $D$-link.
    Then $P_1:=(Q_1- z_1 z_2)\cup R_1 \cup R_2 \cup R_3$ and $P_2:=Q_2$ are as desired.

    Suppose otherwise that the torso at $t(i)$ is a triangle.
    Then $i=2$ and the torso at $t(1)$ is not 3-connected, by \cref{BagelBagAdhesionSet}.
    Since the torso at $t(2)$ is a triangle, the torso at $t(1)$ cannot be a 4-cycle by \cref{TutteBagelObservations}, so it must be a triangle as well.
    Then the adhesion-graph of $t(1) t(2)$ is a $K_2$-subgraph of~$G$ with vertices~$y^W_1,y^W_2$.
    Then we let $P_2:=Q_2:=y^W_1 y^W_2$ and construct $P_1:=R_1\cup R_2\cup R_3$ as in the first case with the \allref{OnePathBagel}.
\end{proof}

\begin{figure}[ht]
    \centering
    \includegraphics[height=10\baselineskip]{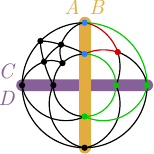}
    \caption{Two paths found by \cref{TwoPathsBagel}.
    The vertices $y^W_i$ are blue. The path $P_2=Q_2$ is red. The path $P_1$ is green. The green vertices in the $B$-link are $u_1,u_2$.}
    \label{fig:BagelBagTwoPaths}
\end{figure}

\begin{lemma}[Three-Paths Lemma]\label{ThreePathsBagelLemma}
    Assume the \allref{BagelBag}.
    There exist three independent $\{x^W_1,y^W_1\}$--$\{x^W_2,y^W_2\}$ paths $P_1,P_2$ and $P_3$ in $G[W] - x'_1 x'_2 - y'_1 y'_2$ such that two $P_i$ link $y^W_1$ to $y^W_2$ and one $P_i$ links $x^W_1$ to $x^W_2$.
\end{lemma}

\begin{proof}
    Recall that not both corners $AD,BD$ are potter by \cref{obs:potter-adjacent-corner}.

    Suppose first that the corner $AD$ is not potter.
    We apply the \allref{TwoPathsBagel} to find two independent $y^W_1$--$y^W_2$ paths $P_1,P_2$ in $G[W] - y'_1 y'_2$ such that $P_1\cup P_2$ uses at most one vertex in the $D$-link.
    Then we apply the \allref{OnePathBagel} to find an $x^W_1$--$x^W_2$ path $P_3\se G[W] - x'_1 x'_2$ that avoids $P_1\cup P_2$.

    Suppose otherwise that the corner $BD$ is not potter.
    We proceed as in the first case, except that we find $P_3$ first and then find $P_1,P_2$ afterwards.
\end{proof}

\subsection{Totally-nested tetra-separations from good bags}\label{sec:BagelTotallyNested}

\begin{definition}[Good bag, bad bag]
    Let $\cO$ be a Tutte-bagel of a 4-connected graph~$G$.
    A torso $H_t$ of $\cO$ is \defn{good} if it is 3-connected and it satisfies at least one of the following conditions:
    \begin{enumerate}
        \item $|H_t|\ge 6$;
        \item some neighbouring torso is a 4-cycle;
        \item both neighbouring torsos are triangles; or
        \item $|H_t|=5$ and some neighbouring torso is a triangle.
    \end{enumerate} 
    If $H_t$ fails to be good, it is \defn{bad}.
\end{definition}

\begin{lemma}\label{GoodBagTetra}
    Assume the \allref{BagelBag}. The following assertions are equivalent:
    \begin{enumerate}
        \item\label{GoodBagTetra1} $(U,W)$ is a tetra-separation;
        \item\label{GoodBagTetra2} $H_{t(0)}$ is good.
    \end{enumerate}
\end{lemma}
\begin{proof}
    Note that $|W'\sm U'|\ge 2$.
    
    \cref{GoodBagTetra2} $\Rightarrow$ \cref{GoodBagTetra1}.
    Assume that $H_{t(0)}$ is good.
    If $|H_{t(0)}|\ge 6$, then $U'\sm W'$ has size~$\ge 2$, so $(U,W)$ is a tetra-separation by \cref{leftrightTetra}.
    If $|H_{t(0)}|= 5$ or $|H_{t(0)}|=4$, then $H_{t(0)}$ being good ensures via the neighbouring torsos that at least one or at least two vertices in $S(U',W')$ have exactly one neighbour in $W'\sm U'$, respectively, so $(U,W)$ is a tetra-separation by \cref{leftrightTetra}.

    \cref{GoodBagTetra1} $\Rightarrow$ \cref{GoodBagTetra2}.
    Assume that $(U,W)$ fails to be a tetra-separation.
    Then by \cref{leftrightTetra}, the number of vertices in $S(U',W')$ with $\le 1$ neighbour in $W'\sm U'$ fails to be $\ge 2-|U'\sm W'|$.
    Hence all sufficient conditions for $H_{t(0)}$ being good fail.
\end{proof}

\begin{center}
    \textbf{Half-connected}
\end{center}

\begin{lemma}\label{BagelHalfconnected}
    Assume the \allref{BagelBag} with $(U,W)$ a tetra-separation.
    Then $(U,W)$ is half-connected.
\end{lemma}
\begin{proof}
    If $S(U,W)$ contains an edge, then $G-S(U,W)$ has exactly two components by 4-connectivity, so $(U,W)$ is half-connected.
    Otherwise $(U,W)=(U',W')$.
    We will show that $G[W\sm U]$ is connected.
    Let $\pi$ denote the adhesion-set of~$\cO$ that is spanned by the $BD$-edge for the $D$-link by \cref{TutteBagelAdhesionSetNote}.

    We claim that every vertex $v\in W\sm U$ not in $\pi$ sends a fan of two paths to $\pi$ in $G[W\sm U]$.
    Indeed, $\pi$ together with $\{x'_1,x'_2\}$, say, separates $v$ from $G_{t(0)}$.
    Since $G$ is 4-connected, there is a fan of four paths from $v$ to $\{x'_1,x'_2\}\cup\pi$ by \cref{Menger}.
    The two paths ending in $\pi$ are as desired.

    So it remains to show that there is at least one vertex in $W\sm U$ besides the two vertices in~$\pi$.
    The vertex $x_1$ has at least two neighbours in $W\sm U$, at least one of which must be in $A\sm B$.
    The neighbour in $A\sm B$ avoids $\pi\se B$.
\end{proof}

\begin{center}
    \textbf{3-linked}
\end{center}

\begin{lemma}\label{CrossPathInBag}
    Let $\cO$ be a Tutte-bagel of a 4-connected graph~$G$.
    Let $H_t$ be a 3-connected torso of $\cO$ with adhesion sets $\{a_1,a_2\}$ and $\{b_1,b_2\}$.
    Then there is an $a_1$--$b_1$ path in $G_t-\{a_2,b_2\}$.
\end{lemma}
\begin{proof}
    Assume for a contradiction that the two vertices $a_1,b_1$ lie in distinct components $C,D$ of $G_t-\{a_2,b_2\}$.
    Then $V(C)=\{a_1\}$ and $V(D)=\{b_1\}$, as $G$ is $4$-connected.
    Hence $G_t$ has only 4 vertices.
    Since $H_t$ is 3-connected, it is a $K_4$.
    But then $a_1 b_1$ is an edge in $G_t$, a contradiction.
\end{proof}

\begin{lemma}\label{PathShortening}
    Assume the \allref{BagelBag}.
    Let $P'$ be a path in $G_{t(0)}$ that joins two vertices $v'_1,v'_2\in S(U',W')$ and avoids the other two vertices in $S(U',W')$.
    Let $v_i\in S(U,W)$ be equal or incident to~$v'_i$.
    Then $P'$ contains a $\hat v_1$--$\hat v_2$ path~$P$ that avoids the two vertices or edges in $S(U,W)$ besides $v_1$ and~$v_2$.
\end{lemma}
\begin{proof}
    If $v_i$ is an edge, then all neighbours of $v'_i$ in $G_{t(0)}$ lie in $S(U',W')$ except the one that is joined to $v'_i$ via the edge~$v_i$, so $P'$ uses~$v_i$.
    Hence $P'$ contains the desired subpath.
\end{proof}

\begin{lemma}\label{bagelHowFar}
    Assume the \allref{BagelBag} with $(U,W)$ a tetra-separation.
    Let $1\le k\le m$ be maximum with $y^W_1\in G_{t(k)}$.
    Assume that $y_2$ is a vertex.
    Then exactly one of the following two assertions holds:
    \begin{enumerate}
        \item\label{bagelHowFar1} The torso $H_{t(k)}$ is 3-connnected, there are $\le 2$ torsos in between $H_{t(0)}$ and $H_{t(k)}$, and they are triangles.
        \item\label{bagelHowFar2} $H_{t(k)}$ is a triangle, the two edges in $H_{t(k)}$ incident to $y^W_1$ exists in~$G$ and $y^W_2\notin H_{t(k)}$.
        Moreover, there are exactly two triangle-torsos between $H_{t(0)}$ and $H_{t(k)}$.
    \end{enumerate}
\end{lemma}
\begin{proof}
    If $H_{t(1)}$ is 3-connected, we have~\cref{bagelHowFar1}.
    Otherwise $H_{t(1)}$ is a 4-cycle or a triangle.
    Since $y_2$ is a vertex by assumption, $H_{t(1)}$ is a triangle and $y_2=y^W_2$ is its tip, by \cref{BagelBagAdhesionSet}.
    Hence $y^W_1$ is incident to the free edge of $H_{t(1)}$.
    If $H_{t(2)}$ is 3-connected, then we have~\cref{bagelHowFar1}.
    Otherwise $H_{t(2)}$ is a triangle, by \cref{TutteBagelObservations}.
    By \cref{trianglePath}, $y^W_1$ is the tip of $H_{t(2)}$.
    Hence $y^W_2$ is not contained in the adhesion-set between $H_{t(2)}$ and $H_{t(3)}$, and in particular $y^W_2$ is not contained in $H_{t(3)}$.
    If $H_{t(3)}$ is 3-connected, we have~\cref{bagelHowFar1}.
    Otherwise $H_{t(3)}$ also is a triangle.
    By \cref{trianglePath}, $y^W_1$ is incident to the free edge of $H_{t(3)}$.
    By \cref{TutteBagelObservations}, the adhesion-graph between the triangle-torsos $H_{t(2)}$ and $H_{t(3)}$ is a $K_2$-subgraph of~$G$.
    Hence the two edges of $H_{t(3)}$ incident to $y^W_1$ exist in~$G$.
    Thus we have~\cref{bagelHowFar2}.
\end{proof}

\begin{lemma}\label{bagelTwoPaths}
    Let $\cO$ be a Tutte-bagel of a 4-connected graph~$G$.
    Let $H_t$ be a 3-connected torso of $\cO$ with adhesion-sets $\{a_1,a_2\}$ and $\{b_1,b_2\}$.
    Then either
    \begin{enumerate}
        \item\label{bagelTwoPaths1} there exist two independent $a_1$--$b_1$ paths in the bag $G_t$, or
        \item\label{bagelTwoPaths2} the torso $H_t$ is a 4-wheel with rim $a_1 a_2 b_1 b_2$ and both edges $a_1 a_2, b_1 b_2$ are missing from~$G$.
    \end{enumerate}
\end{lemma}
\begin{proof}
    Assume that \cref{bagelTwoPaths1} fails.
    Then there exists a cutvertex $v$ of $G_t$ that separates $a_1$ from $b_1$.
    Since $H_t$ is 3-connected and $G_t$ misses at most two of the torso-edges, $v$ is distinct from $a_2$ and~$b_2$.
    Adding the edges $a_1 a_2$ and $b_1 b_2$ makes $v$ into the separator of a mixed-3-separation $(X,Y)$ of $H_t$ with $a_1\in X\sm Y$ and $b_1\in Y\sm X$.
    Since $S(X,Y)$ contains edges, $H_t -S(X,Y)$ has exactly two components.
    These components include the vertex-sets $\{a_1,v,b_2\}$ and $\{a_2,v,b_1\}$, respectively.
    Since $G$ is 4-connected, these inclusions are equalities, and $v$ has degree four in~$H_t$.
    As $H_t$ is 3-connected, $H_t$ is a 4-wheel with rim $a_1 a_2 b_1 b_2$.
    Since \cref{bagelTwoPaths1} fails by assumption, both torso-edges $a_1 a_2$ and $b_1 b_2$ must be missing from~$G$.
\end{proof}

\begin{lemma}\label{outsidePaths}
    Assume the \allref{BagelBag} with $(U,W)$ a tetra-separation.
    Assume that both $x_2$ and $y_2$ are vertices.
    There are two independent $x^W_1$--$y^W_1$ paths in $G[W]-x_2-y_2$.
\end{lemma}
\begin{proof}
    Let $1\le i\le m$ be maximum with $y^W_1\in G_{t(i)}$.
    Let $1\le j\le m$ be minimum with $x^W_1\in G_{t(j)}$.
    Let $\pi_\ell$ denote the adhesion-set at $t(\ell)t(\ell+1)$ for all~$\ell$.
    We consider two cases.

    \casen{Case $i<j$.}
    This case is supported by \cref{fig:outsidePaths}.
    If $H_{t(i)}$ is 3-connected, we let $F_y$ by a $y^W_1$--$\pi_i$ fan of two paths in $G_{t(i)}$ avoiding the vertex in $\pi_{i-1}-y^W_1$.
    Otherwise, by \cref{bagelHowFar}, $H_{t(i)}$ is a triangle avoiding $y^W_2$ such that both edges incident to $y^W_1$ in $H_{t(i)}$ exist in~$G$,
    and we let $F_y$ denote the $y^W_1$--$\pi_i$ fan of two paths defined by these two edges.
    In both cases, $F_y$ avoids~$y_2$.
    Similarly, we obtain an $x^W_1$--$\pi_{j-1}$ fan $F_x\se G_{t(j)}$ avoiding $x_2$.
    By \cref{twoDisjointPathsOutsideBagel}, we find two disjoint $\pi_i$--$\pi_{j-1}$ paths in $G_{t(i+1)}\cup\cdots\cup G_{t(j-1)}$.
    Call these paths $P$ and~$Q$.
    Then $F_x\cup P\cup Q\cup F_y$ is a union of two independent $x^W_1$--$y^W_1$ paths in $G[W]-x_2-y_2$.

    \begin{figure}[ht]
        \centering
        \includegraphics[height=6\baselineskip]{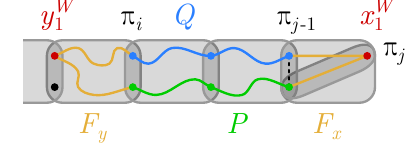}
        \caption{Case $i<j$ in the proof of \cref{outsidePaths}. 
        Here, $H_{t(i)}$ is 3-connected and $H_{t(j)}$ is a triangle. The unique vertex in $\pi_{i-1}-y^W_1$ is black.}
        \label{fig:outsidePaths}
    \end{figure}

    \casen{Case $i\ge j$.}
    This case is supported by \cref{fig:outsidePathsOther}.
    We claim that the corner $BD$ is not potter, and assume for a contradiction that it is.
    Let $\pi$ denote the unique adhesion-set of the $BD$-Tutte-banana.
    Since $BD$ is potter, neither corner $BC$ nor $AD$ is potter by \cref{obs:potter-adjacent-corner}.
    Hence if $y_1$ is an edge, it neither dangles from the $C$-link through the $B$-link nor from the $B$-link through the $C$-link, by the \allref{keylem:crossing}.
    So $y^W_1\in C\sm D$, and hence $y^W_1\notin\pi$.
    Similarly, $x^W_1\notin\pi$.
    But this contradicts $i\ge j$, so the corner $BD$ is not potter.

    We claim that $y^W_1$ lies in the $B$-link and $x^W_1$ lies in the $D$-link.
    Assume for a contradiction that $y^W_1$ avoids the $B$-link, say.
    Then $y^W_1\in C\sm D$.
    Let $\pi$ denote the adhesion-set defined by the $BD$-edge for the $B$-link.
    Since $BD$ is not potter, no edge dangles from the $D$-link through the $B$-link.
    Hence $\pi\se (B\sm A)\cap D$.
    Now $y^W_1\notin\pi$ since $y^W_1\notin D$, and $x^W_1\notin\pi$ since $x^W_1\notin B\sm A$, contradicting $i\ge j$.
    So $y^W_1$ lies in the $B$-link and $x^W_1$ lies in the $D$-link.

    By \allref{keylem:crossing}, the corners $BC$ and $AD$ are potter.
    Let $T$ denote the $BD$-Tutte-path.
    Then $o(B)$ and $o(D)$ are the ends of~$T$.
    Note that $T$ has $\ge 3$ nodes.

    Assume first that $o(B)$ or $o(D)$ has a neighbour in $T$ with a triangle-torso in~$\cO$.
    Say this is a neighbour $t$ of $o(B)$.
    Since $y^W_1$ is the tip of the triangle $H_{o(B)}$, the vertex $y^W_1$ is an end of the free edge $f$ of $H_t$ by \cref{trianglePath}.
    Let $u$ denote the other end of~$f$.
    Let $v$ denote the tip of $H_t$.
    Note that $\{u,v\}$ is an adhesion-set, so $x^W_1\in\{u,v\}$.
    Hence $x^W_1$ lies in the two triangle-torsos from the potter corner $AD$ and in the triangle-torso $H_t$.
    Thus $o(D)$ and $t$ are neighbours in~$T$ by \cref{trianglePath}.
    In particular, $T$ has exactly three nodes.
    Since $v$ lies in the two triangle-torsos $H_t$ and $H_{o(B)}$ while $x^W_1$ lies in the two triangle-torsos for the potter corner $AD$, we cannot have $v=x^W_1$ by \cref{trianglePath}.
    Hence $x^W_1=u$.
    By \cref{TutteBagelObservations}, the two adhesion-sets of $H_t$ are spanned by edges in~$G$.
    Hence $x^W_1 y^W_1$ and $x^W_1 v y^W_1$ are desired paths.

    \begin{figure}[ht]
        \centering
        \includegraphics[height=9\baselineskip]{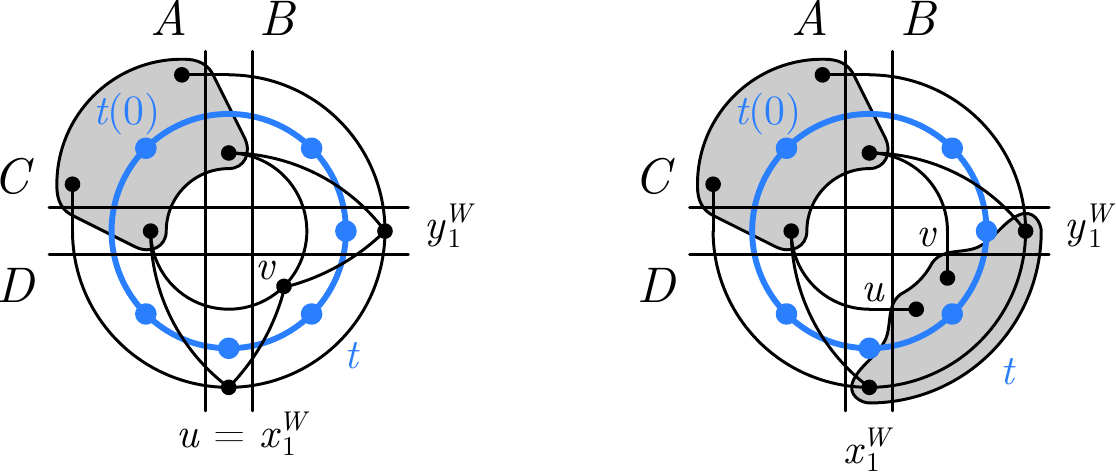}
        \caption{Case $i\ge j$ in the proof of \cref{outsidePaths}. Left: $H_t$ is a triangle.\\ Right: $H_t$ is 3-connected.}
        \label{fig:outsidePathsOther}
    \end{figure}

    Assume otherwise that neither $o(B)$ nor $o(D)$ has a neighbour in $T$ with a triangle-torso.
    Then the torsos neighbouring $o(B)$ and $o(D)$ via $T$ are 3-connected.
    Since 3-connected torsos have disjoint adhesion-sets and $i\ge j$, we have $|T|=3$.
    Let $t$ denote the internal node of~$T$.
    Let $\{x^W_1,u\}$ denote the adhesion-set of the edge $o(D)t$, and let $\{y^W_1,v\}$ denote the adhesion-set of the edge $o(B)t$.
    We find two independent $y^W_1$--$x^W_1$ paths in $G_t$, and are done, unless $H_t$ is a 4-wheel with rim $x^W_1 u y^W_1 v$ and $x^W_1 u,y^W_1 v\notin E(G)$, by \cref{bagelTwoPaths}.
    But in the latter case, $u$ and $v$ have degree three in $G$, contradicting 4-connectivity.
\end{proof}

\begin{lemma}\label{Bagel3linkedAntisym}
    Assume the \allref{BagelBag} with $(U,W)$ a tetra-separation.
    Then $\{x_i,y_j\}$ is 3-linked around $S(U,W)$ for all $i,j\in [2]$.
\end{lemma}

\begin{proof}
    Say $i=j=1$.
    If $x_2$ or $y_2$ is an edge, then we are done by \ref{itm:3-linked-1}, so assume that both are vertices.
    By \cref{CrossPathInBag}, there is an $x'_1$--$y'_1$ path $P'$ in $G_{t(0)}-\{x'_2,y'_2\}$.
    Then $P'$ contains a $\hat x_1$--$\hat y_1$ path~$P$ that avoids the vertices~$x_2$ and~$y_2$, by \cref{PathShortening}.
    We find two independent $x^W_1$--$y^W_1$ paths in $G[W]-x_2-y_2$ using \cref{outsidePaths}.
    Hence $\{x_1,y_1\}$ is 3-linked around $S(U,W)$ by \ref{itm:3-linked-3}.
\end{proof}

\begin{lemma}\label{Bagel3linkedSym}
    Assume the \allref{BagelBag} with $(U,W)$ a tetra-separation.
    Then $\{x_1,x_2\}$ is 3-linked around $S(U,W)$.
    Similarly, $\{y_1,y_2\}$ is 3-linked around $S(U,W)$.
\end{lemma}
\begin{proof}
    If $y_1$ or $y_2$ is an edge, we are done, so assume that both are vertices.
    There exist two independent $x^W_1$--$x^W_2$ paths in $G[W]-\{y_1,y_2\}$ by the \allref{ThreePathsBagelLemma}.
    So it remains to find an $x'_1$--$x'_2$ path in $G':=G_{t_0}-\{y'_1,y'_2\}$ (as any such path contains a suitable $\hat x_1$--$\hat x_2$ path by \cref{PathShortening}).
    Assume for a contradiction that such a path is missing.
    Then $H_{t(0)}$ is a $K_4$ by \cref{3conAdhesionNoPathK4}.
    Since both $y_1$ and $y_2$ are vertices, the torso $H_{t(1)}$ is 3-connected by \cref{BagelBagAdhesionSet}.
    Since $H_{t(0)}$ is good, it follows that $H_{t(-1)}$ is a 4-cycle.
    So $x_1$ and $x_2$ are the two free edges on this 4-cycle.
    By 4-connectivity of $G$, both $x'_1$ and $x'_2$ have degree at least four.
    Since $H_{t(-1)}$ is a 4-cycle and $H_{t(0)}$ is a $K_4$, it follows that the edge $x'_1 x'_2$ must exist in~$G$, contradicting that there is no $x'_1$--$x'_2$ path in $G':=G_{t_0}-\{y'_1,y'_2\}$.
\end{proof}

\begin{center}
    \textbf{Potter-linked}
\end{center}

\begin{lemma}\label{Bagel0potter}
    Assume the \allref{BagelBag} with $(U,W)$ a tetra-separation.
    Then $(U,W)$ is 0-potter-linked.
\end{lemma}

\begin{proof}
    Let $\{\pi_1,\pi_2\}$ be a balanced bipartition of $S(U,W)$.
    It suffices to find five independent $\hat\pi_1$--$\hat\pi_2$ paths in~$G$.

    Suppose first that $\pi_1=\{x_1,x_2\}$ and $\pi_2=\{y_1,y_2\}$.
    We find two independent $\hat\pi_1$--$\hat\pi_2$ paths $P_1,P_2$ in $G[W]$ using \cref{twoDisjointPathsOutsideBagel,BagelBagAdhesionSet}.
    The torso $H_{t(0)}$ is 3-connected by assumption, and hence it contains three independent $\{x'_1,x'_2\}$--$\{y'_1,y'_2\}$ paths by \allref{Menger}.
    None of these paths use torso-edges, so all three paths are subgraphs of the bag~$G_{t(0)}$.
    The three $\{x'_1,x'_2\}$--$\{y'_1,y'_2\}$ paths can be shortened to three independent $\hat\pi_1$--$\hat\pi_2$ paths $Q_1,Q_2,Q_3$ by \cref{PathShortening}.
    Hence $\pi_1$ is 0-potter-linked to~$\pi_2$.

    Suppose otherwise that $\pi_1=\{x_1,y_1\}$ and $\pi_2=\{x_2,y_2\}$, say.
    We find three independent $\hat\pi_1$--$\hat\pi_2$ paths $P_1,P_2,P_3$ in $G[W] - x'_1 x'_2 - y'_1 y'_2$ using the \allref{ThreePathsBagelLemma}.
    By \cref{bagelTwoPaths}, we find two independent $\{x'_1,y'_1\}$--$\{x'_2,y'_2\}$ paths in the bag $G_{t(0)}$, which shorten to two independent $\hat\pi_1$--$\hat\pi_2$ paths $P_4,P_5$ in $G_{t(0)}$ by \cref{PathShortening}, unless $H_{t(0)}$ is a 4-wheel and both torso-edges $x'_1 x'_2,y'_1 y'_2$ are missing in~$G$.
    But the latter is impossible: since $H_{t(0)}$ is good, some neighbouring torso would be a triangle or a 4-cycle, but then a vertex in the adhesion-set would have degree three in~$G$, contradicting 4-connectivity.
\end{proof}

Assume the \allref{BagelBag}.
We call $x_i$ an \defn{inward edge} if $x_i$ is an edge and $x^W_i=x'_i$; that is, if $x_i$ is an edge with both endvertices in $G_{t(0)}$.
We call $x_i$ an \defn{outward edge} if $x_i$ is an edge but not an inward edge; that is, if the endvertex $x^W_i$ of $x_i$ lies outside of $G_{t(0)}$.

\begin{lemma}\label{bagelInwardEdge}
    Assume the \allref{BagelBag} with $(U,W)$ a tetra-separation.
    Assume that $x_1$ is an inward edge.
    Then all of the following hold:
    \begin{enumerate}
        \item\label{bagelInwardEdge1} Some $y'_i$ is a neighbour of $x'_1$ and $y^W_i=y'_i$.
        \item\label{bagelInwardEdge2} $\{x_1,x_2\}$ is not $W$-weird.
    \end{enumerate}
\end{lemma}
\begin{proof}
    \cref{bagelInwardEdge1}.
    Since $H_{t(0)}$ is 3-connected, the vertex $x'_1$ has $\ge 3$ neighbours in $H_{t(0)}$.
    Let $u$ be a neighbour of $x'_1$ in $H_{t(0)}$ distinct from $x'_2$ and the end of the edge $x_1$ besides~$x'_1$.
    All interior vertices of $H_{t(0)}$ are contained in $U\sm W$, so $u\in\{y'_1,y'_2\}$.
    Say $u=y'_1$.
    Assume for a contradiction that $y^W_1\notin H_{t(0)}$.
    Then $x'_1 y'_1$ is a $(W\sm U)$--$(U\sm W)$ edge, so $(U,W)$ violates the matching-condition, a contradiction.

    \cref{bagelInwardEdge2}.
    Assume for a contradiction that $\{x_1,x_2\}$ is $W$-weird.
    Then $x_2=x'_2$ since $x_1$ is an edge by assumption.
    Moreover, $x'_2$ has exactly one neighbour outside $H_{t(0)}$ by weirdness.
    But then $x_2$ should be an edge in $S(U,W)$ by the definition of right-left-reduction of $(U',W')$, contradicting that $x_2$ is a vertex.
\end{proof}

\begin{lemma}\label{bagelPotterExtra}
    Assume the \allref{BagelBag} with $(U,W)$ a tetra-separation.
    Let $\pi_1:=\{x_1,x_2\}$ and $\pi_2:=\{y_1,y_2\}$.
    Assume that $x_1$ is an outward edge.
    Assume that~$\pi_1$ is $U$-weird.
    Then $x_2=x'_2$ and there is an $x_2$--$\{y'_1,y'_2\}$ edge~$x_2 y_i'$ in~$G$ with~$y_i'=y_i^W$.
\end{lemma}
\begin{proof}
    Since $\pi_1$ is $U$-weird and $x_1$ is an outward edge, we get that $x_2=x'_2$, that $x_2 x'_1$ is an edge, and that $x_2$ has at most one neighbour in $U\sm\{x'_1\}$.
    However, $H_{t(0)}$ is 3-connected, so $x_2$ has some two neighbours $u,v$ in $H_{t(0)}$ besides $x'_1$.
    Hence one of these neighbours, say $v$, is not in $U$.
    Since all interior vertices of $H_{t(0)}$ lie in $U\sm W$, the vertex $v$ must be one of the $y'_i$, say $v=y'_1$.
    Since $v\notin U$, we get that $y_1$ is an inward edge.
    In particular, $y^W_1=y'_1$.
\end{proof}

\begin{lemma}\label{bagelPotterExtra2}
    Assume the \allref{BagelBag} with $(U,W)$ a tetra-separation.
    Let $\pi_1:=\{x_1,x_2\}$ and $\pi_2:=\{y_1,y_2\}$.
    Assume that~$\pi_1$ is $W$-weird and~$x_1$ is an outward edge.
    Then no~$y_i$ is an inward edge.
\end{lemma}
\begin{proof}
    Assume otherwise, say~$y_1$ is an inward edge.
    Then by \cref{bagelInwardEdge} we have $y^W_1=y'_1$ and some $x'_j$ is a neighbour of $y^W_1$.
    Then $j\neq 1$ since otherwise $x'_1 y^W_1$ is a $(U\sm W)$--$(W\sm U)$ edge, contradicting the matching-condition.
    So $j=2$.
    Then~$y_1' = y_1^W$ is the only neighbour of~$x_2$ in~$W \sm \{x_1^W\}$.
    It follows that~$x_1^W$ is the only neighbour of~$x_2'$ in~$W' \sm U'$.
    This contradicts that~$x_2$ is a vertex by definition of right-left-reduction.
\end{proof}

\begin{lemma}\label{bagel1potterAdhVsAdh}
    Assume the \allref{BagelBag} with $(U,W)$ a tetra-separation.
    Let $\pi_1:=\{x_1,x_2\}$ and $\pi_2:=\{y_1,y_2\}$.
    Then $\{\pi_1,\pi_2\}$ is 1-potter-linked.
\end{lemma}
\begin{proof}
    By symmetry, it suffices to show that $\pi_1$ is 1-potter-linked to $\pi_2$.
    If $\pi_1$ is not weird for any side of $(U,W)$, then $\pi_1$ is 1-potter-linked to $\pi_2$ by \ref{itm:1-potter-linked-1}.
    So we may assume that $\pi_1$ is weird for some side; in particular, $x_1$ is an edge and $x_2=x'_2$ is a vertex, say.

    \casen{Case: $x_1$ is an inward edge.}
    By \cref{bagelInwardEdge}, there is a neighbour $y'_i$ of $x'_1$ with $y'_i=y^W_i$, and $\pi_1$ is not $W$-weird.
    So~$\pi_1$ is 1-potter-linked to~$\pi_2$ around~$W$ by \ref{itm:1-potter-linked-1}.
    Since we assumed earlier that $\pi_1$ is weird for some side of $(U,W)$, but showed that $\pi_1$ is not $W$-weird, it follows that $\pi_1$ is $U$-weird.
    By \cref{twoDisjointPathsOutsideBagel}, there exist two disjoint $\{x^W_1,x^W_2\}$--$\{y^W_1,y^W_2\}$ paths in $\bigcup_{t\in O-t(0)}G_t$.
    Since $x_1$ is an inward edge and $y'_i=y^W_i$, the edge $x'_1 y'_i$ can also be written as $x^W_1 y^W_i$.
    This edge together with the two paths forms a set of three independent $\hat\pi_1$--$\hat\pi_2$ paths in $G[W]$, so $\pi_1$ is 1-potter-linked to $\pi_2$ around~$U$ by \ref{itm:1-potter-linked-3}.

    \casen{Case: $x_1$ is an outward edge.}
    If $\pi_1$ is $U$-weird, then we find an $x_2$--$\{y'_1,y'_2\}$ edge $e=x_2 y_i'$ with~$y_i'=y_i^W$ by \cref{bagelPotterExtra}.
    Since $x_2\in\pi_1$ and $y_i'\in\hat\pi_2 \cap W$, the edge $e$ is a $\pi_1$--$(\hat\pi_2 \cap W)$ edge, which witnesses that $\pi_1$ is 1-potter-linked to $\pi_2$ around~$U$ by~\ref{itm:1-potter-linked-2}.
    So it remains to show that $\pi_1$ is 1-potter-linked to~$\pi_2$ around~$W$.
    For this, we assume that~$\pi_1$ is $W$-weird by~\ref{itm:1-potter-linked-1}.
    By \cref{bagelPotterExtra2}, no~$y_i$ is an inward edge.
    Since $H_{t(0)}$ is 3-connected, we find three independent $\{x'_1,x'_2\}$--$\{y'_1,y'_2\}$ paths in $H_{t(0)}$ by \allref{Menger}.
    All three paths are subgraphs of $G[U]$ since no $y_i$ is an inward edge.
    Moreover, they are three $\hat\pi_1$--$\hat\pi_2$ paths in $G[U]$.
    This means $\pi_1$ is 1-potter-linked to $\pi_2$ around~$W$ by \ref{itm:1-potter-linked-3}.
\end{proof}

\begin{lemma}\label{bagel2potterAdhVsAdh}
    Assume the \allref{BagelBag} with $(U,W)$ a tetra-separation.
    Let $\pi_1:=\{x_1,x_2\}$ and $\pi_2:=\{y_1,y_2\}$.
    Then $\{\pi_1,\pi_2\}$ is 2-potter-linked.
\end{lemma}
\begin{proof}
    By symmetry and by \cref{obs:weird-no-edges-in-separator}, it suffices to show the following: if $\pi_1$ is $U$-weird and $\pi_2$ is $W$-weird, then there is a $\pi_1$--$\pi_2$ edge in $G$.
    So we assume that $\pi_1$ is $U$-weird and $\pi_2$ is $W$-weird.
    Then no $y_i$ is an inward edge by \cref{bagelInwardEdge}.
    Without loss of generality,~$y_1$ is an outward edge.
    Then by \cref{bagelPotterExtra2}, no~$x_i$ is an inward edge.
    So we can assume that~$x_1$ is an outward edge.
    Then by \cref{bagelPotterExtra}, we find an $x_2$--$\{y'_1,y'_2\}$ edge~$x_2y_i'$ with~$y_i'=y_i^W$.
    Note that $y_i$ is a vertex, so~$y_i'=y_i \in \pi_2$. 
    Then~$x_2y_i'$ is the desired edge.
\end{proof}

\begin{lemma}\label{UWedgeEndsInA}
    Assume the \allref{BagelBag} with $(U,W)$ a tetra-separation.
    If $x_1$ is an edge, then $x^W_1\in A$.
\end{lemma}
\begin{proof}
    Suppose for a contradiction that $x^W_1\in B\sm A$.
    By the \allref{keylem:crossing}, there are no jumping edges and no diagonal edges.
    Then $x_1$ dangles from the $A$-link through the $D$-link.
    Hence the corner $AD$ is potter by the \allref{keylem:crossing}.
    But then $x'_1$ has a neighbour in $D$-link, so $x'_1$ has $\ge 2$ neighbours in $W'\sm U'$, contradicting that $x_1$ is an edge rather than the vertex $x'_1$.
\end{proof}

\begin{lemma}\label{bagelPotterCrossNotWweird}
    Assume the \allref{BagelBag} with $(U,W)$ a tetra-separation.
    Then $\{x_1,y_1\}$ is not $W$-weird.
\end{lemma}
\begin{proof}
    Assume for a contradiction that $\{x_1,y_1\}$ is $W$-weird.
    Then $x_1$ is an edge and $y_1=y'_1$ is a vertex, say.

    Assume first that $x_1$ is an inward edge.
    Then $y'_1$ has $\le 1$ neighbour in $W\sm\{x'_1\}\supseteq V(G-G_{t(0)})$.
    But then, by the definition of $(U,W)$ via right-left-reduction, $y_1$ should be an edge rather than the vertex $y'_1$, a contradiction.

    Assume otherwise that $x_1$ is an outward edge.
    By \cref{UWedgeEndsInA}, $x^W_1\in A$.
    However, since $y'_1$ is contained in $A\cap C$ and the Tutte-bagel induces a path-decomposition of $G[A]$, the vertex $y'_1=y_1$ does not neighbour $x^W_1$, contradicting that $\{x_1,y_1\}$ is $W$-weird.
\end{proof}

\begin{lemma}\label{bagelPotterCross12}
    Assume the \allref{BagelBag} with $(U,W)$ a tetra-separation.
    Let $\pi_1:=\{x_1,y_1\}$ and $\pi_2:=\{x_2,y_2\}$.
    Then $\{\pi_1,\pi_2\}$ is both 1-potter-linked and 2-potter-linked.
\end{lemma}
\begin{proof}
    By \cref{bagelPotterCrossNotWweird}, neither $\pi_1$ nor $\pi_2$ is $W$-weird.
    Hence $\{\pi_1,\pi_2\}$ is 2-potter-linked.
    It remains to show that $\pi_1$ is 1-potter-linked to $\pi_2$, say.
    Since $\pi_1$ is not $W$-weird, it remains to consider the case where $\pi_1$ is $U$-weird.
    We find three independent $\hat\pi_1$--$\hat\pi_2$ paths in $G[W]$ with the \allref{ThreePathsBagelLemma}.
    Hence $\pi_1$ is 1-potter-linked to $\pi_2$ via \ref{itm:1-potter-linked-3}.
\end{proof}

\begin{keylemma}
    \label{cor:bagel-totally-nested}
    Assume the~\allref{setting:TutteBagel}.
    Then the following two assertions are equivalent for every node $t$ of~$O$:
    \begin{enumerate}
        \item $(U_t,W_t)$ is a totally-nested tetra-separation;
        \item the torso of $\cO$ at $t$ is good.
    \end{enumerate}
\end{keylemma}
\begin{proof}
    If $H_t$ is bad, then $(U_t,W_t)$ is not even a tetra-separation by \cref{GoodBagTetra}, let alone totally-nested.
    So we may assume now that $H_t$ is good.
    In particular, we may assume the \allref{BagelBag} with $t=t(0)$.

    By \cref{GoodBagTetra}, $(U,W)$ is a tetra-separation.
    It suffices to show that $(U,W)$ is externally 5-connected by \cref{keylem:nestedness-external-connectivity}.
    By \cref{BagelHalfconnected}, $(U,W)$ is half-connected.
    To see that $(U,W)$ is 3-linked, we combine \cref{Bagel3linkedAntisym} with \cref{Bagel3linkedSym}.
    By \cref{Bagel0potter}, $(U,W)$ is 0-potter-linked.
    We get that $(U,W)$ is 1-potter-linked and 2-potter-linked from \cref{bagel1potterAdhVsAdh,bagelPotterCross12,bagel2potterAdhVsAdh}.
\end{proof}

\subsection{Splitting stars and torsos}\label{sec:BagelSplitting}

\begin{lemma}\label{BagelMovingCD}
    Assume the \allref{setting:TutteBagel}.
    Let $H_s$ be a torso of $\cO$.
    Assume that $s$ stems from the $AC$-Tutte-path~$T$.
    Assume that the $A$-link is represented by a triangle-torso, and that $s$ is neither equal nor adjacent to the node $o(A)$ in~$O$.
    Then $G_s\se G[C\sm D]$.
\end{lemma}
\begin{proof}
    Let $r$ be the unique neighbour of $o(A)$ in~$T$.
    If $H_r$ is 3-connected or a 4-cycle, then its adhesion sets are disjoint, and we get $G_s\se G[C\sm D]$.
    Otherwise $H_r$ is a triangle.
    The tips of $H_{o(A)}$ and $H_r$ are the two vertices in the adhesion-set of the edge $o(A)r$.
    It follows that the other adhesion-set of $H_r$ is included in $C\sm D$, so $G_s\se G[C\sm D]$.
\end{proof}

\begin{lemma}\label{BagelMovingTip}
    Assume the \allref{setting:TutteBagel}.
    Let $H_s$ be a triangle-torso with tip~$u$.
    Assume that $s$ stems from the $AC$-Tutte-path.
    Assume that $s\neq o(A)$, and that if the $A$-link is represented by a triangle-torso then $s$ is not adjacent to~$o(A)$.
    Then $u$ has no neighbour in the $D$-link.
\end{lemma}
\begin{proof}
    Assume not for a contradiction, and let $v$ be a neighbour of $u$ in the $D$-link.
    Recall that $u\notin D$ by \cref{BagelMovingCD}.
    So the edge $uv$ dangles from the $D$-link through the $A$-link.
    Hence the corner $AD$ is potter by the \allref{keylem:crossing}, and the triangle-torso $H_{o(A)}$ has $uv$ as free edge.
    By assumption, $o(A)$ and $s$ are not adjacent in~$O$.
    By \cref{trianglePath}, there is precisely one torso $H_r$ in between the two triangle-torsos $H_{o(A)}$ and $H_s$, which also contains~$u$.
    But $H_r$ has $u$ as tip since $uv$ is free in $H_{o(A)}$, which by \cref{trianglePath} contradicts that $H_s$ also has $u$ as tip.
\end{proof}

\begin{figure}[ht]
    \centering
    \includegraphics[height=12\baselineskip]{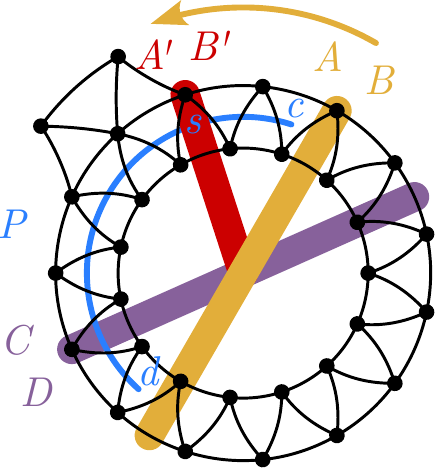}
    \caption{The situation in \cref{BagelMovingLemmaC3C4}.}
    \label{fig:BagelMovingC3C4}
\end{figure}

Let $P$ be a path, and let $u,v$ be vertices on~$P$.
Recall that $uP\mathring v$ denotes the subpath of $P$ that starts at $u$ and ends just before the vertex~$v$.
The statement of the following lemma is supported by \cref{fig:BagelMovingC3C4}.

\begin{lemma}[Link-Moving Lemma for vertices]\label{BagelMovingLemmaC3C4}
    Assume the \allref{setting:TutteBagel}.
    Let $P$ and $Q$ be the two paths that are the two components of $O-o(C)-o(D)$.
    Assume that $o(A)\in P$ and $o(B)\in Q$.
    Let $c$ denote the end of $P$ that is adjacent or incident to $o(C)$ in~$O$.
    Define $d$ similarly.
    Let $s$ be a node in the same component of $P-o(A)$ as $c$ such that $H_s$ is a 4-cycle or a triangle.
    Assume that if the $A$-link is represented by a triangle-torso, then $s$ is not adjacent to $o(A)$ in~$P$.
    Let $P':=dP\mathring{s}$ and $Q':=\mathring{s}Pc\cup Q$.
    Then:
    \begin{enumerate}
        \item\label{BagelMovingLemmaC3C4_1} $A':=\bigcup_{t\in V(P')}V(G_t)$ and $B':=\bigcup_{t\in V(Q')}V(G_t)$ are the two sides of a tetra-separation of~$G$.
        \item\label{BagelMovingLemmaC3C4_2} $S(A',B')=S\cup (D$-link$)$, where $S$ consists of the two free edges of $H_s$ if $H_s$ is a 4-cycle and of the free edge and the tip of $H_s$ if $H_s$ is a triangle.
        \item\label{BagelMovingLemmaC3C4_3} $(A',B')$ separates some two vertices in $B\cup D$.
    \end{enumerate}
\end{lemma}

\begin{proof}
    \cref{BagelMovingLemmaC3C4_2} holds by construction, and gives that the mixed-separation $(A',B')$ has order four.
    We show \cref{BagelMovingLemmaC3C4_1} next.

    \casen{Matching-condition.}
    Let $ab$ be an edge in the $D$-link with $a\in A\sm B$ and $b\in B\sm A$.
    Let $f$ be an edge in~$S$.
    Assume for a contradiction that $ab$ and $f$ share an endvertex, which can only be~$a$.
    Since $a\in (A\sm B)\cap D$ and $V(G_s)\se A\cap C$, we find that $a$ is in the $A$-link.
    So $ab$ dangles from the $A$-link through the $D$-link.
    Then the corner~$AD$ is potter by the \allref{keylem:crossing}.
    Thus the $A$-link
    is represented by a triangle-torso, so $a\notin D$ by \cref{BagelMovingCD}, a contradiction.

    \casen{Degree-condition.} 
    It is straightforward to check that every vertex in $S(A',B')$ has $\ge 2$ neighbours in $B'\sm A'$.
    Now let $v\in S(A',B')$ be a vertex; we claim that it has $\ge 2$ neighbours in $A'\sm B'$.
    If $v$ is in $S$, then $H_s$ is a triangle with tip $v$.
    Let $H_r$ denote the torso neighbouring $H_s$ in the direction of $o(A)$ on the path~$P$.
    If $H_r$ is 3-connected, then $v$ has two neighbours in $G_r$, both of which lie in $A'\sm B'$ since $G_s\se G[A\cap C]$.
    Otherwise $H_r$ is a triangle, and the adhesion-graph in between $H_r$ and $H_s$ is a $K_2$-subgraph of~$G$ by \cref{TutteBagelObservations}.
    Hence $v$ has two neighbours in $G_r$, both of which lie in $A'\sm B'$ since $r\neq o(A)$ by assumption.
    So we may assume next that $v$ lies in the $D$-link.
    Every neighbour of $v$ in $A\sm B$ also lies in $A'\sm B'$, except possibly when a neighbour of $v$ lies in $S$ or an edge incident to $v$ lies in $S$.
    The former exception is impossible by \cref{BagelMovingTip}, while the latter is excluded by \cref{BagelMovingCD}.
    This conclude the proof of \cref{BagelMovingLemmaC3C4_1}.

    \cref{BagelMovingLemmaC3C4_3}.
    We have to find two vertices in $B\cup D$ that are contained in $A'\sm B'$ and in $B'\sm A'$, respectively.
    Every vertex in $B\sm A$ lies in the intersection of $B\cup D$ with $B'\sm A'$, and $B\sm A$ is non-empty.
    Any vertex in $D\sm B$ lies in $A'\sm B'$ by \cref{BagelMovingCD}, so it remains to show that $D\sm B$ is non-empty.
    If $D\sm B$ is empty, then the $D$-link consists of two edges.
    But both these edges have ends in the corner for~$AD$ by the \allref{keylem:crossing}, a contradiction.
\end{proof}

\begin{lemma}[Link-Moving Lemma for edges]\label{BagelMovingLemma33}
    Assume the \allref{setting:TutteBagel}.
    Let $P$ and $Q$ be the two paths that are the two components of $O-o(C)-o(D)$.
    Assume that $o(A)\in P$ and $o(B)\in Q$.
    Let $c$ denote the end of $P$ that is adjacent or incident to $o(C)$ in~$O$.
    Define $d$ similarly.
    Let $e$ be an edge of $O$ such that the torsos at the ends of $e$ are 3-connected.
    Assume that $e$ stems from the $AC$-Tutte-path.
    Let $P':=dP\mathring{e}$ and $Q':=\mathring{e}Pc\cup Q$.
    Then:
    \begin{enumerate}
        \item\label{BagelMovingLemma33_1} $A':=\bigcup_{t\in V(P')}V(G_t)$ and $B':=\bigcup_{t\in V(Q')}V(G_t)$ are the two sides of a tetra-separation of~$G$.
        \item\label{BagelMovingLemma33_2} $S(A',B')=\pi\cup (D$-link$)$, where $\pi$ is the adhesion-set of $ac$.
        \item\label{BagelMovingLemma33_3} $(A',B')$ separates some two vertices in $B\cup D$.
    \end{enumerate}
\end{lemma}
\begin{proof}
    The proof is analogous to the proof of \cref{BagelMovingLemmaC3C4}, except that everything is straightforward because the two 3-connected torsos at the ends of $e$ both have disjoint adhesion-sets and have all their vertices contained in $A\cap C$.
\end{proof}

\begin{lemma}\label{RightUWchar}
    Assume the \allref{setting:TutteBagel}.
    Let $i:=1$ if $H_{t(1)}$ is 3-connected, and $i:=2$ otherwise.
    Similarly, let $j:=-1$ if $H_{t(-1)}$ is 3-connected, and $j:=-2$ otherwise.
    The right-reduction of $(U'_t,W'_t)$ has sides
    \[
        V(G_{t(0)})\quad\text{and}\quad \bigcup\,\{\,V(G_{t(\ell)}):i\le\ell\le j\text{ in the cyclic ordering of }\Z_{m+1}\,\}.
    \]
\end{lemma}
\begin{proof}
    This is just a different perspective on \cref{BagelBagAdhesionSet}.
\end{proof}

\begin{definition}[Tetra-star $\sigma(\cO)$]
    Assume the \allref{setting:TutteBagel}.
    The \defn{tetra-star} of $\cO$ is the set of all $(U_t,W_t)$ for good torsos $H_t$ of~$\cO$.
    We denote the tetra-star of $\cO$ by $\defnMath{\sigma(\cO)}$.
\end{definition}

\begin{lemma}\label{BagelInterlaceRightShift}
    Assume the \allref{setting:TutteBagel}.
    Let $(E,F)$ be a totally-nested tetra-separation of $G$ that interlaces $\sigma$.
    Then there is a 3-connected torso $H_t$ of~$\cO$ such that $(E,F)\le (U''_t,W''_t)$ or $(F,E)\le (U''_t,W''_t)$, where $(U''_t,W''_t)$ denotes the right-reduction of $(U'_t,W'_t)$.
\end{lemma}
\begin{proof}
    Since $(E,F)$ is totally-nested, we have $(E,F)\le (A,B)$ and $(E,F)\le (C,D)$, say. In particular, $E\se A\cap C$.
    Let $\beta$ denote the $AC$-banana, and let $(T,\cV)$ denote the $AC$-Tutte-banana.
    
    An element $x\in V(T)\cup E(T)$ is \defn{admissible} if
    \begin{itemize}
        \item $x$ is an edge with 3-connected torsos at its end-nodes; or
        \item $x$ is an interior node of $T$ with a triangle or 4-cycle as torso.
    \end{itemize}
    A mixed-1-separation $(M,N)$ of $T$ is \defn{admissible} if the vertex or edge in $S(M,N)$ is admissable.
    The \defn{$\beta$-lift} of an admissible $(M,N)$ is
    \begin{itemize}
        \item $(\bigcup_{t\in M}V_t,\bigcup_{t\in N}V_t)$ if $S(M,N)$ consists of an edge, and
        \item $(\bigcup_{t\in M\sm N}V_t,\bigcup_{t\in N\sm M}V_t)$ if $S(M,N)$ consists of a vertex.
    \end{itemize}
    Note that the $\beta$-lift is a mixed-2-separation of the banana~$\beta$.
    A \defn{$G$-lift} of an admissible $(M,N)$ is a tetra-separation of $G$ that induces the $\beta$-lift of $(M,N)$ on~$\beta$ and separates some two vertices in~$F$. 
    
    We claim that:
    \begin{equation}\label{eq:GliftsExist}
        \textit{Every admissible $(M,N)$ has a $G$-lift.}
    \end{equation}
    Given an admissible $(M,N)$, let $x$ denote the element of its separator.
    If $x$ is an edge, then the \allref{BagelMovingLemma33} produces a $G$-lift.
    So we may assume that $x$ is an interior node of $T$.
    Let $T'$ be obtained from $T$ by removing any endvertex with a triangle or 4-cycle torso that represents a link.
    Any vertex removed in this way is not admissible, and neither is its incident edge in~$T$.
    Hence $x\in T'$.

    Let us show that $x$ is not the only vertex of~$T'$.
    Otherwise the torso $H_x$ is a triangle or 4-cycle.
    By \cref{BagelLinkForm}, $H_x$ is a triangle and both the $A$-link and the $C$-link are represented by triangle-torsos.
    In particular, $A\cap C=3$.
    But $E$, as a side of a tetra-separation, has size $\ge 4$, and $E$ is included in $A\cap C$, a contradiction.
    So $x$ is not the only vertex of~$T'$.
    
    Let $y$ be an endvertex of $T'$ that is distinct from~$x$.
    By symmetry, we may assume that $y$ is adjacent or incident to $o(A)$ when viewed as a node of~$O$.
    Then the \allref{BagelMovingLemmaC3C4} produces a $G$-lift.
    This concludes the proof that \cref{eq:GliftsExist} holds.

    For every admissible $(M,N)$ we use \cref{eq:GliftsExist} to choose a $G$-lift $(\bar M,\bar N)$ so that for $(N,M)$ we choose $(\bar N,\bar M)$.
    We claim that exactly one of $E\se \bar M$ and $E\se\bar N$ holds, for all~$(M,N)$.
    Indeed, otherwise $(\bar M,\bar N)$ separates~$E$, and since it also separates~$F$ by definition, $(\bar M,\bar N)$ crosses $(E,F)$, a contradiction.
    We let
    \begin{align*}
        \omega &{}:=\{\,(M,N):(M,N)\text{ is admissable and }\bar N\supseteq E\,\}\\
        \Omega&{}:=\bigcap\,\{\,N\sm M:(M,N)\in\omega\,\}\cap V(T').
    \end{align*}
    We claim that $\Omega\neq\emptyset$.
    Assume for a contradiction that $\Omega$ is empty.
    
    Suppose first that $\omega$ is empty.
    Then $V(T')$ is empty.
    Hence $T$ is a $K_2$ both whose endvertices have triangle or 4-cycle torsos that represent links.
    By \cref{TutteBagelObservations}, both torsos are triangles.
    Hence the corner $AC$ is potter, so $A\cap C=2$, contradicting that $E\se A\cap C$ where $E$ has size~$\ge 4$.
    
    Suppose otherwise that $\omega$ is non-empty.    
    So $\omega$ has one or two maximal elements.
    Assume first that $\omega$ has a unique maximal element $(M,N)$.
    Since $\Omega$ is empty but $N\sm M$ is not, $N\sm M$ must consist of an end of $T$ that misses in $T'$.
    Hence $N\sm M=\{o(C)\}$, say, and the unique element $x$ of $S(M,N)$ is not an edge.
    So $x$ is a neighbour of $o(C)$, and both $x$ and $o(C)$ have a triangle or 4-cycle as torso.
    Hence $\bar N\cap A$ equals the adhesion-set of the edge $x o(C)$, contradicting that $E$ has size $\ge 4$ and $E\se \bar N\cap A$.
    
    Assume otherwise that $\omega$ has two maximal elements $(M_1,N_1)$ and $(M_2,N_2)$.
    Let $x_i$ denote the unique element of $S(M_i,N_i)$ for both~$i$.
    We claim that both $x_i$ are vertices.
    Otherwise $x_1$ is an edge, say.
    Then both endvertices of $x_1$ have 3-connected torsos.
    Hence $x_2$ cannot be an endvertex of~$x_1$.
    Therefore, an endvertex of $x_1$ is in $\Omega$, a contradiction.
    So both $x_i$ are vertices.
    Since $\Omega$ is empty, the vertices $x_i$ are neighbours.
    Hence $E$ is included in the adhesion-set of $x_1 x_2$, contradicting that $E$ has size~$\ge 4$.
    Thus $\Omega\neq\emptyset$.

    Every node in $\Omega$ has a 3-connected torso in~$\cO$.
    Since $\Omega$ spans a subpath of~$T$ that contains no admissible edges, we have $|\Omega|=1$.
    Hence $\Omega$ consists of exactly one node $t^*$, and its torso $H_{t^*}$ is 3-connected.
    
    To show that $(E,F)\le$ the right-reduction of $(U'_t,W'_t)$, by \cref{RightUWchar} and \cref{Righty} it remains to show $E\se V(G_{t^*})$.
    Assume for a contradiction that there is a vertex $v\in E\sm V(G_{t^*})$.
    Recall that $E\se A\cap C$, so in particular $E\se V(\beta)$.
    Hence there is a node $s$ of $T$ with $v\in V_s$, and $s\neq t^*$ by assumption.
    Let $r$ be the penultimate node on the path $s T t^*$.

    Assume first that the torso $H_r$ is 3-connected.
    Then the edge $rt^*$ is admissible.
    Let $(M,N)\in\omega$ with $S(M,N)=\{rt^*\}$.
    Note that $r\in M\sm N$ and $t^*\in N\sm M$ since $\{t^*\}=\Omega$.
    Let $(\hat M,\hat N)$ be the $\beta$-lift of $(M,N)$.
    Observe that $v\in\hat M\sm\hat N$.
    Therefore, $v\in\bar M\sm \bar N$, contradicting that $E\se\bar N$.

    Assume otherwise that the torso $H_r$ is either a triangle or a 4-cycle.
    Then $r$ is admissible or $r$ is one of $o(A)$ and $o(C)$.
    Assume first that $r$ is admissible.
    Let $(M,N)\in\omega$ with $S(M,N)=\{r\}$.
    Note that $r,s\in M$ and $t^*\in N\sm M$.
    Let $(\hat M,\hat N)$ be the $\beta$-lift of $(M,N)$.
    Then $v\in\hat M\sm\hat N$, so $v\in \bar M\sm \bar N$, contradicting that $E\se\bar N$.
    Assume otherwise that $r=o(A)$, say.
    Then $H_r$ is a triangle or 4-cycle representing the $A$-link.
    Since $v$ does not lie in $H_{t^*}$, we in particular have that $v$ does not lie in the adhesion-set of the edge $rt^*$.
    But then $v$ lies in $D\sm C$, contradicting that $E\se C$.
\end{proof}

\begin{lemma}\label{BagelSplittingStar}
    Assume the \allref{setting:TutteBagel}.
    Every tetra-separation of $G$ that interlaces the tetra-star $\sigma$ of $\cO$ is crossed.
\end{lemma}
\begin{proof}
    Assume for a contradiction that $\sigma$ is interlaced by a totally-nested tetra-separation $(E,F)$ of~$G$.
    By \cref{BagelInterlaceRightShift}, there is a 3-connected torso $H_t$ of~$\cO$ such that $(E,F)\le (U''_t,W''_t)$, say, where $(U''_t,W''_t)$ denotes the right-reduction of $(U'_t,W'_t)$.

    Assume first that $H_t$ is good.
    Then $(U_t,W_t)$ is a totally-nested tetra-separation by \cref{cor:bagel-totally-nested}.
    In particular, $(U_t,W_t)\in\sigma$.
    Since $(E,F)$ interlaces $\sigma$, we have one of $(U_t,W_t)<(E,F)$ and $(U_t,W_t)<(F,E)$.
    We cannot have $(U_t,W_t)<(F,E)$, since otherwise $\emptyset\neq  E\sm F\se (W_t\sm U_t)\cap (U''_t\sm W''_t)$ contradicts the fact that $U''_t\sm W''_t\se U_t\sm W_t$ by \cref{obs:shift-maintains-proper-sides}.
    Hence $(U_t,W_t)<(E,F)\le (U''_t,W''_t)$.
    Since $(U_t,W_t)$ is the left-reduction of $(U''_t,W''_t)$, some vertex in $S(E,F)$ must have $\le 1$ neighbour in $E\sm F$, contradicting the degree-condition.

    Assume otherwise that $G_t$ is bad.
    Hence $G_t$ has four or five vertices.
    If $H_t$ has four vertices, then all but at most one of the four vertices in $S(U'_t,W'_t)$ has $\ge 2$ neighbours in $W'_t\sm U'_t$.
    If $H_t$ has five vertices, then all four vertices in $S(U'_t,W'_t)$ have $\ge 2$ neighbours in $W'_t\sm U'_t$.
    In either case, we have $|U''_t\sm W''_t|\le 1$.
    Since $(E,F)\le (U''_t,W''_t)$, we have $E\sm F\se U''_t\sm W''_t$.
    Therefore, $|E\sm F|\le 1$, which by \cref{leftDegreeStrongerMatching} contradicts that $(E,F)$ is a tetra-separation.
\end{proof}

\begin{keylemma}\label{bagelTorso}
    Let $G$ be a 4-connected graph, and let $N$ denote the set of totally-nested tetra-separations of~$G$.
    Let $\sigma$ be a splitting star of~$N$ that is interlaced by two tetra-separations $(A,B)$ and $(C,D)$ of $G$ that cross with empty vertex-centre.
    Then the torso $\tau$ of $\sigma$ is a cycle of triangle-torsos and 3-connected torsos on $\le 5$ vertices.
\end{keylemma}

\begin{proof}
    By the \allref{keylem:crossing}, all links have size two.
    By \cref{lem:really-is-Tutte-bagel}, $(A,B)$ and $(C,D)$ induce a Tutte-bagel~$\cO$.
    By \cref{cor:bagel-totally-nested}, the tetra-star $\sigma(\cO)$ is a star of totally-nested tetra-separations of~$G$.
    By \cref{BagelSplittingStar}, the tetra-star $\sigma(\cO)$ is a splitting star of~$N$.
    Since both $(A,B)$ and $(C,D)$ interlace both $\sigma$ and $\sigma(\cO)$, we obtain $\sigma=\sigma(\cO)$.
    Hence it remains to compute the torso $\tau$ of $\sigma(\cO)$.
    The torso $\tau$ is 4-connected by \cref{dotTorso4con}.
    We can obtain $\tau$ from $\cO$ by replacing each good torso with a $K_4$, leaving only 3-connected torsos with five or four vertices, and contracting some free edges in torsos that are 4-cycles or triangles.
    Each torso that is a 4-cycle neighbours two 3-connected torsos by \cref{TutteBagelObservations}, and the 4-cycle-torso witnesses that both 3-connected torsos are good.
    Hence the free edges of the 4-cycle-torsos are contained in the separators of the elements of $\sigma(\cO)$ that stem from the good neighbouring torsos, and so they are contracted.
    Hence $\tau$ is a cycle of triangle-torsos and 3-connected torsos on $\le 5$ vertices.
\end{proof}

\subsection{Angry Tutte-bagels}\label{sec:BagelAngry}

Let $G$ be a 4-connected graph with a Tutte-bagel $\cO=(O,\cG)$.
A \defn{triangle-strip} of $\cO$ is a maximal connected subgraph $\Delta$ of $O$ such that every node in $\Delta$ has a triangle-torso with regard to~$\cO$.
Phrased differently, a triangle-strip is a component of 
\[
    O[\,\{t\in V(O):t\text{ has a triangle-torso}\}\,].
\]
Triangle-strips are paths unless all torsos of $O$ are triangles.
An element of $V(O)\cup E(O)$ is \defn{furious} if either
\begin{itemize}
    \item it is an edge with 3-connected torsos at both ends, or
    \item it is a vertex with a 4-cycle as torso.
\end{itemize}
We define
\[
    \defnMath{\alpha(\cO)}:=\#(\text{furious edges and vertices of }\cO)+\sum_\Delta\left\lceil\tfrac{2}{3}|\Delta|\right\rceil
\]
where the sum ranges over all triangle-strips $\Delta$ of $\cO$.

\begin{lemma}\label{CrossedBagelsAngryFactor}
    Assume the \allref{setting:TutteBagel}.
    Then $\alpha(\cO)\ge 4$.
\end{lemma}
\begin{proof}
    The four links injectively correspond to furious edges or vertices or nodes in triangle-strips.
    Hence it suffices to show that there is no path $rst$ in $O$ with all three nodes $r,s,t$ having triangle-torsos representing links.
    Assume otherwise.
    Then two adjacent corners are potter, contradicting \cref{obs:potter-adjacent-corner}.
\end{proof}

\begin{lemma}\label{AngryFactorYieldsCrossed}
    Let $G \neq K_5$ be a 4-connected graph with a Tutte-bagel $\cO=(O,\cG)$ such that all torsos of $\cO$ are bad and $\alpha(\cO)\ge 4$.
    Then $G$ has two tetra-separations that cross with all links of size two.
\end{lemma}
\begin{proof}
    We claim that we can select four elements $x(i)$ of $V(O)\cup E(O)$, where $i\in\Z_4$, such that
    \begin{itemize}
        \item every selected element is furious or belongs to some triangle-strip;
        \item no three consecutive nodes are chosen from any triangle-strip.
    \end{itemize}
    If all torsos of~$\cO$ are triangles, then we use~$|O| \geq 6$, which we have since~$G \neq K_5$ and $\alpha(\cO) \geq 4$. Otherwise, we use $\alpha(\cO) \geq 4$ and that the triangle-strips of~$\cO$ are induced paths in~$O$.
    
    We may assume that $O$ traverses the $x(i)$ in the cyclic ordering of $\Z_4$.
    Let $P_A,P_B$ denote the components of $O-x(0)-x(2)$, let $A:=\bigcup_{t\in V(P_A)}V(G_t)$ and define $B$ similarly.
    We use $x(1),x(3)$ to define $P_C,P_D$ and $C,D$ analogously.

    We define $S(i)$ for $i\in\Z_4$ to consist of
    \begin{itemize}
        \item the two vertices in the adhesion-set of $x(i)$ if $x(i)$ is a furious edge;
        \item the two free edges of the 4-cycle-torso at $x(i)$ if $x(i)$ is a furious vertex;
        \item the tip and the free edge of the triangle-torso at $x(i)$ if $x(i)$ stems from a triangle-strip.
    \end{itemize}
    
    We claim that $(A,B)$ and $(C,D)$ are tetra-separations with separators
    $S(A,B)=S(0)\cup S(2)$ and $S(C,D)=S(1)\cup S(3)$.
    By symmetry, it suffices to show that
    \begin{enumerate}
        \item\label{vxTech} every vertex in $S(0)$ has $\ge 2$ neighbours in $A\sm B$, and
        \item\label{edgeTech} every edge in $S(0)$ has an end in $A\sm B$ avoiding $\widehat{S(2)}$.
    \end{enumerate}
    For this, let $t(1),t(2),t(3)$ denote the first three vertices on $P_A$ that come just after $x(0)$ on~$O$.

    \casen{Case: $x(0)$ is a furious vertex.}
    So $H_{x(0)}$ is a 4-cycle and $S(0)$ consists of two edges, and we only have to show \cref{edgeTech}.
    The neighbouring torso $H_{t(1)}$ must be 3-connected by \cref{TutteBagelObservations}, so the adhesion-set of $x(0)t(1)$ is included in $A\sm B$ and avoids $\widehat{S(2)}$.
    Hence \cref{edgeTech} holds.
    
    \casen{Case: $x(0)$ is a furious edge.}
    Then $S(0)$ is the adhesion-set of the edge $x(0)$, and we only have to show \cref{vxTech}.
    Note that $t(1)$ is an end of $x(0)$, so the torso $H_{t(1)}$ is 3-connected.
    Let $U:=V(H_{t(1)}-S(0))$.
    Every vertex in $S(0)$ has two neighbours in $G$ that are contained in $U$.
    To ensure that these neighbours lie in $A\sm B$, it suffices to show that the adhesion-set $\pi$ of $t(1)t(2)$ is included in $A\sm B$.
    Since $t(1) t(2)$ is the first candidate for $x(1)$, with $x(2)$ coming later, $\pi$ can only meet $B$ if $H_{t(2)}$ is a triangle.
    But then $t(2)$ is the first candidate for $x(1)$, so $x(2)$ comes after $t(2)$.
    If $H_{t(3)}$ is not a triangle, we are done, and otherwise the tip of the triangle $H_{t(3)}$ avoids $\pi$ by \cref{trianglePath}.
    Hence $\pi\se A\sm B$ as desired.

    \casen{Case: $x(0)$ is in a triangle-strip.}
    Assume first that $H_{t(1)}$ is 3-connected.
    Since $H_{t(1)}$ is bad and neighboured by the triangle-torso $H_{x(0)}$, its other neighbouring torso $H_{t(2)}$ must be 3-connected.
    Hence the edge $t(1)t(2)$ is the first candidate for $x(1)$, and so the edge $t(2)t(3)$ is the first candidate for $x(2)$.
    Therefore, the entire vertex-set of $H_{t(1)}$ is included in $A\sm B$, which ensures \cref{vxTech} and \cref{edgeTech}.

    Assume otherwise that $H_{t(1)}$ is a triangle.
    Let $\pi$ denote the adhesion-set of the edge $t(1)t(2)$.
    Note that $\pi$ contains an end of the free edge in $S(0)$ and both vertices in $\pi$ are neighbours of the tip of the triangle $H_{x(0)}$.
    Hence to show both \cref{vxTech} and \cref{edgeTech}, it suffices to show that $\pi$ is included in $A\sm B$.
    The node $t(1)$ is the first candidate for $x(1)$, so $t(2)$ is the first candidate for $x(2)$.
    By \cref{TutteBagelObservations}, $H_{t(2)}$ is either a triangle or 3-connected.
    If $H_{t(2)}$ is 3-connected, then $\pi$ is included in $A\sm B$.
    Otherwise $H_{t(2)}$ is a triangle.
    Then $H_{x(0)}$, $H_{t(1)}$ and $H_{t(2)}$ are three consecutive triangles.
    The way we have selected the $x(i)$ then ensures that $x(2)$ comes after $t(2)$.
    Again, $H_{t(3)}$ is either a triangle or 3-connected. If $H_{t(3)}$ is 3-connected, then $\pi$ is included in $A\sm B$. Otherwise, if $H_{t(3)}$ is a triangle, then its tip is distinct from the tip of $H_{t(2)}$ by \cref{trianglePath}, and $t(3)$ is the first candidate for~$x(2)$.
    So $\pi$ is included in $A\sm B$.

    This completes the proof of \cref{vxTech} and \cref{edgeTech}.
    Hence $(A,B)$ and $(C,D)$ are tetra-separations with $S(A,B)=S(0)\cup S(2)$ and $S(C,D)=S(1)\cup S(3)$.
    The tetra-separations $(A,B)$ and $(C,D)$ cross with all links of size two as the links are precisely the four non-empty sets $S(i)$.
\end{proof}

\begin{keylemma}\label{AngryBagel}
    Let $G$ be a 4-angry graph.
    Then the following two assertions are equivalent:
    \begin{enumerate}
        \item\label{AngryBagel1} $G$ has two tetra-separations that cross with all links of size two;
        \item\label{AngryBagel2} $G\neq K_5$ and $G$ has a Tutte-bagel $\cO$ with $\alpha(\cO)\ge 4$ all whose torsos are bad.
    \end{enumerate}
\end{keylemma}
\begin{proof}
    \cref{AngryBagel1} $\Rightarrow$ \cref{AngryBagel2}.
    Clearly,~$G \neq K_5$ since~$K_5$ does not have a tetra-separation.
    By \cref{lem:really-is-Tutte-bagel}, $(A,B)$ and $(C,D)$ induce a Tutte-bagel~$\cO$.
    By \cref{cor:bagel-totally-nested}, the tetra-star $\sigma(\cO)$ is a star of totally-nested tetra-separations of~$G$.
    Since $G$ has no totally-nested tetra-separations by assumption,
    all torsos of $\cO$ are bad.
    We get $\alpha(\cO)\ge 4$ from \cref{CrossedBagelsAngryFactor}.

    \cref{AngryBagel2} $\Rightarrow$ \cref{AngryBagel1} is \cref{AngryFactorYieldsCrossed}.
\end{proof}

\section{Quasi-5-connected torsos}\label{sec:Quasi5con}

This section deals with the quasi-5-connected outcome of \cref{MainDecomp}; see \cref{Quasi5conSummary}.

A mixed-separation\pl\ $(A,B)$ of a graph $G$ \defn{almost interlaces} a star $\sigma$ of mixed-separations\pl\ of $G$ if $(C,D)\le (A,B)$ or $(C,D)\le (B,A)$ for all $(C,D)\in\sigma$.
The difference between `interlaces' and `almost interlaces' is `$<$' versus `$\le$' in the definitions.

\begin{lemma}\label{HyperLiftProperties}
    \textnormal{(Analogue of \cite[Lemma~2.6.7]{Tridecomp})}
    Let $G$ be a 4-connected graph.
    Let $\sigma$ be a star of mixed-4-separations of $G$ satisfying the matching-condition.
    Let $(A,B)$ be a separation of the torso $\tau$ of~$\sigma$.
    Then there exists a mixed-separation $(\hat A,\hat B)$ of~$G$ such that:
    \begin{enumerate}
        \item\label{HyperLiftProperties1} the order of $(\hat A,\hat B)$ is at most the order of $(A,B)$;
        \item\label{HyperLiftProperties2} $(\hat A,\hat B)$ almost interlaces~$\sigma$;
        \item\label{HyperLiftProperties3} $|\hat A\sm\hat B|\ge |A\sm B|$ and $|\hat B\sm\hat A|\ge |B\sm A|$; and
        \item\label{HyperLiftProperties4} for every $(U,W)\in\sigma$ we have $|W\cap (\hat A\sm\hat B)|\ge |A\sm B|$ and $|W\cap (\hat B\sm\hat A)|\ge |B\sm A|$.
    \end{enumerate}
\end{lemma}

\begin{proof}
    The proof is the same as the proof of \cite[Lemma~2.6.7]{Tridecomp}, with one exception: we replace Lemma~1.2.3 with the matching-condition.
\end{proof}

\begin{lemma}\label{tetraShield}
    Let $(A,B)$ be a tetra-separation in a 4-connected graph~$G$.
    Let $(C,D)$ be a mixed-4-separation\pl\ in $G$ with $(A,B)\le (C,D)$.
    Let $(C',D')$ denote either the left-right-reduction or the right-left-reduction of $(C,D)$.
    Then $(A,B)\le (C',D')$.
\end{lemma}
\begin{proof}
    We have $B\supseteq D\supseteq D'$.
    Since $A\se C$, it remains to show that every vertex in $S(C,D)\cap A$ also lies in $S(C',D')$.
    Every vertex $v$ in $S(C,D)\cap A$ also lies in $S(A,B)$.
    Hence every such $v$ has $\ge 2$ neighbours in $A\sm B\se C\sm D$.
    Thus $v\in S(C',D')$.
\end{proof}

\begin{lemma}\label{technicalInterlaces}
    Let $G$ be a 4-connected graph.
    Let $(U,W)$ be a tetra-separation of~$G$, and let $(A,B)$ be a mixed-4-separation of~$G$ such that $(U,W)\le (A,B)$.
    Suppose that $|W\cap (A\sm B)|\ge 1$.
    Let $(A',B')$ be either the left-right-sift or the right-left-reduction of $(A,B)$.
    Then $(U,W)<(A',B')$.
\end{lemma}
\begin{proof}
    We have $(U,W)\le (A',B')$ by \cref{tetraShield}.
    Recall that $A\sm B\se A'\sm B'$ by \cref{obs:reduction-maintains-proper-sides}.
    Hence $W\cap (A'\sm B')$ is nonempty, so in particular the inclusion $W\supset B'$ is proper.
\end{proof}

\begin{lemma}\label{big4sepInterlaces}
    Let $G$ be a 4-connected graph.
    Let $\sigma$ be a star of tetra-separations of~$G$.
    Suppose that the torso $\tau$ of $\sigma$ has a 4-separation $(A,B)$ such that both sides have size~$\ge 6$.
    Then $\sigma$ is interlaced by a tetra-separation of~$G$.
\end{lemma}
\begin{proof}
    By \cref{HyperLiftProperties}, there is a mixed-$k$-separation\pl\ $(\hat A,\hat B)$ of $G$ such that
    \begin{itemize}
        \item $k\le 4$ by \cref{HyperLiftProperties1};
        \item $(\hat A,\hat B)$ almost interlaces $\sigma$ by \cref{HyperLiftProperties2};
        \item $|\hat A\sm\hat B|\ge 2$ and $|\hat B\sm\hat A|\ge 2$ by \cref{HyperLiftProperties3}; and
        \item for every $(U,W)\in\sigma$ we have $|W\cap (\hat A\sm \hat B)|\ge 2$ and $|W\cap (\hat B\sm\hat A)|\ge 2$ by \cref{HyperLiftProperties4}.
    \end{itemize}
    Since $G$ is 4-connected, $(\hat A,\hat B)$ is a mixed-4-separation by \cref{kConMixed}.
    Let $(\bar A,\bar B)$ denote the left-right-reduction of $(\hat A,\hat B)$.
    Then $(\bar A,\bar B)$ is a tetra-separation by \cref{leftrightTetra}.
    By \cref{technicalInterlaces}, $(\bar A,\bar B)$ interlaces~$\sigma$.
\end{proof}

\begin{keylemma}\label{Quasi5conSummary}
    Let $G$ be a 4-connected graph, and let $N$ denote the set of totally-nested tetra-separations of~$G$.
    Let $\sigma$ be a splitting star of~$N$ that is not interlaced by any tetra-separation of~$G$.
    Then the torso $\tau$ of $\sigma$ is quasi-5-connected.
\end{keylemma}
\begin{proof}
    By \cref{dotTorso4con}, $\tau$ is 4-connected or a $K_4$, and we are done in the latter case.
    By the contrapositive of \cref{big4sepInterlaces}, every 4-separation of $\tau$ has a side of size five.
\end{proof}

\section{An Angry Theorem for 4-connectivity}\label{sec:Angry}

A key step in proving Tutte's decomposition theorem is to characterise the graphs that are \defn{2-angry}: those that are 2-connected and all whose 2-separators are crossed -- hence the name.
Tutte proved the first \emph{angry theorem}: the 2-angry graphs are precisely the 3-connected graphs and the cycles~\cite{TutteCon}.
For a modern proof, see \cite[Theorem 3.3.4]{Tridecomp}.
An angry theorem for 3-connectivity has been proved using tri-separations \cite[Theorem 1.1.5]{Tridecomp}.
However, like here, the angry theorem was not used in the proof of the decomposition theorem.
Still, angry theorems have recently found an application in group theory: 
the local version of the 2-angry theorem \cite{Local2sep} was used as a key ingredient in \cite{StallingsNilpotent}, and versions for higher connectivity are expected to be needed to continue in the direction of \cite{StallingsNilpotent}.
Here we prove an angry theorem for 4-connectivity.
A graph $G$ is \defn{4-angry} if it is 4-connected and every tetra-separation of $G$ is crossed by a tetra-separation.

\begin{theorem}[Angry Theorem for 4-connectivity]\label{4angry}
    A 4-connected graph $G$ with $|G|\ge 8$ is 4-angry if and only if it is one of the following:
    \begin{enumerate}
        \item\label{4angry1} quasi-5-connected;
        \item\label{4angry3} a cycle $\cO$ of triangles, 4-cycles, 3-connected graphs on $\le 5$ vertices, with only bad torsos and $\alpha(\cO)\ge 4$;
        \item\label{4angry2} a double-wheel with rim of length~$\geq 4$ or a double-wheel of triangles with rim of length~$\geq 4$;
        \item\label{4angry4} a sprinkled $K_{4,m}$ with $m\ge 4$.
    \end{enumerate}
\end{theorem}
\begin{proof}
    \casen{Forward implication.}
    Assume that $G$ is 4-angry.
    If $G$ has no tetra-separation, then \cref{4angry1} by \cref{quasi5conVtetra}.
    So $G$ has a tetra-separation $(A,B)$, which is crossed by a tetra-separation $(C,D)$ as $G$ is 4-angry.
    By the \allref{keylem:crossing}, all links have the same size $\ell$, and $\ell\in\{0,1,2\}$.
    \begin{itemize}
        \item If $\ell=0$, then \cref{4angry4} by \cref{keylem:K4m}.
        \item If $\ell=1$, then \cref{4angry2} by \cref{keylem:double-wheel-angry}.
        \item If $\ell=2$, then \cref{4angry3} by \cref{AngryBagel} and since $G \neq K_5$ by $|G| \geq 8$.
    \end{itemize}

    \casen{Backward implication.} We use one of \cref{quasi5conVtetra}, \cref{keylem:K4m}, \cref{keylem:double-wheel-angry} and \cref{AngryBagel}. Here \cref{quasi5conVtetra} needs $|G|\ge 8$.
\end{proof}

\section{A canonical \texorpdfstring{$Y$--$\Delta$}{YDelta} operation}\label{sec:YDelta}

\begin{lemma}\label{GDeltaLifts}
    Let $G$ be a quasi-4-connected graph.
    Let $(A^\Delta,B^\Delta)$ be a $k$-separation of $G^\Delta$.
    Let $U$ be the set of all degree-three vertices of~$G$.
    We obtain $A^s$ from $A^\Delta$ by adding every vertex $u\in U$ with $\Delta(u)\se G^\Delta [A^\Delta]$.
    We obtain $B^s$ from $B^\Delta$ by adding every vertex $u\in U$ with $\Delta(u)\se G^\Delta [B^\Delta]$ but $\Delta(u)\not\se G^\Delta [A^\Delta]$.
    Then $(A^s,B^s)$ is a $k$-separation of $G^s$ with $A^s\cap V(G^\Delta)=A^\Delta$ and $B^s\cap V(G^\Delta)=B^\Delta$ and $A^s\cap B^s=A^\Delta\cap B^\Delta$, see \cite[Lemma~2.6.4]{Tridecomp}.
    Let $A:= A^s\cap V(G)$ and $B:=B^s\cap V(G)$.
    Then $(A,B)$ is a mixed-separation of $G$ of order~$\le k$.\qed
\end{lemma}

\begin{lemma}\label{GDelta4con}
    Let $G$ be a quasi-4-connected graph on $\ge 7$ vertices.
    Then $G^\Delta$ is 4-connected.
\end{lemma}
\begin{proof}
    Assume for a contradiction that $G^\Delta$ is not 4-connected.
    Let $(A^\Delta,B^\Delta)$ be a separation of $G^\Delta$ of least possible order~$k$, so $k\le 3$.
    From $(A^\Delta,B^\Delta)$ we obtain a mixed-separation $(A,B)$ of $G$ of order $\le k$ as in \cref{GDeltaLifts}.

    We claim that $A\sm B$ and $B\sm A$ have size $\ge 2$.
    It suffices to show that $|A\sm B|\ge 2$ (the proof for $|B\sm A|\ge 2$ is a subproof of that for $|A\sm B|\ge 2$).
    Since $(A^\Delta,B^\Delta)$ is proper, there is a vertex $x\in A^\Delta\sm B^\Delta$.
    We distinguish two cases.

    \casen{Assume first that $x$ is a subdividing vertex of an edge $e=uv$ of~$G$.}
    Consider the triangles $\Delta(u),\Delta(v)\se G^\Delta$ that were added around $u$ and~$v$, respectively.
    Both contain the vertex~$x$, so both are included in $G^\Delta[A^\Delta]$ but not in $G^\Delta[B^\Delta]$.
    Hence both $u$ and $v$ lie in $A^s\sm B^s$.
    Since both $u$ and $v$ are in~$G$, they also lie in $A\sm B$.
    Hence $|A\sm B|\ge 2$.

    \casen{Assume otherwise that $x$ is a vertex in~$G^\Delta$ that is also a vertex in~$G$.}
    Then $x$ is not contained in~$U$, so $x$ has degree $\ge 4$ in~$G$.
    Now $x\in A^s\sm B^s$, and $x\in A\sm B$.
    Since $(A,B)$ has order~$\le 3$, at least one of the four neighbours of $x$ in~$G$ also lies in $A\sm B$.
    Hence $|A\sm B|\ge 2$.

    Therefore, both $A\sm B$ and $B\sm A$ have size~$\ge 2$.
    Hence $G$ is not quasi-4-connected by \cref{qkcNoMixedKsep}, a contradiction.
\end{proof}

\begin{example}\label{GDelta7optimal}
    The number 7 in \cref{GDelta4con} is optimal.
    To see this, let $G$ be a cycle of $C_4$-bags of length~3.
    Then $G^\Delta$ is a generalised 4-wheel, which is 3-connected but not quasi-4-connected.
\end{example}

\section{Appendix A: Circular saws}\label{sec:saw}

Here we verify that the circular saws from the introduction have the desired properties.

\begin{lemma}\label{SawPaths}
    Let $G$ be a $k$-regular circular saw on $\Z_n\times\Z_2$ with $n\gg k$.
    Let $aa'$ and $bb'$ be two edges of $G$ such that $\{a,a'\}$ and $\{b,b'\}$ are disjoint and there is no $\{a,a'\}$--$\{b,b'\}$ edge.
    Let $X$ denote the neighbourhood of $\{a,a'\}$, and let $Y$ denote the neighbourhood of $\{b,b'\}$.
    There are $2(k-1)$ disjoint $X$--$Y$ paths in~$G$.
\end{lemma}
\begin{proof}
    For each vertex $x\in X$ we construct two $x$--$Y$ paths $P^+_x,P^-_x$ inductively, as follows.
    See \cref{fig:SawPaths} for an example.
    
    In the first step of the construction of $P^+_x$, we let $P^+_x$ start and end with~$x$.
    Assume now that $P^+_x$ ends with the vertex $(\ell,i)$.
    If $(\ell,i)\in Y$, then we terminate the construction.
    Otherwise $(\ell,i)\notin Y$.
    If $i=0$ we append the vertex $(\ell+k-1,1)$, and if $i=1$ we add the vertex $(\ell,0)$.
    This completes the construction of $P^+_x$.

    In the first step of the construction of $P^-_x$, we let $P^-_x$ start and end with~$x$.
    Assume now that $P^-_x$ ends with the vertex $(\ell,i)$.
    If $(\ell,i)\in Y$, then we terminate the construction.
    Otherwise $(\ell,i)\notin Y$.
    If $i=0$ we append the vertex $(\ell,1)$, and if $i=1$ we add the vertex $(\ell-k+1,0)$.
    This completes the construction of $P^-_x$.

    Exactly one of the paths $P^+_x,P^-_x$ has an internal vertex in~$X$.
    Let $P_x$ denote the other path.
    Then $(P_x:x\in X)$ is a family of disjoint $X$--$Y$ paths.
    There are $|X|=2(k-1)$ of these paths.
\end{proof}

\begin{figure}[ht]
    \centering
    \includegraphics[height=7\baselineskip]{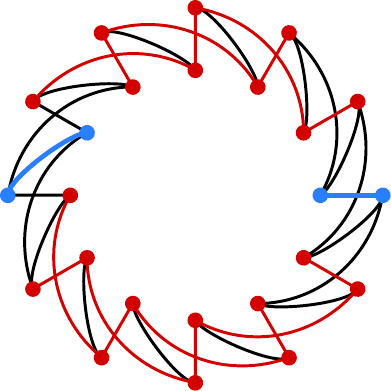}
    \caption{Paths constructed in the proof of \cref{SawPaths}. The edges $aa'$ and $bb'$ are coloured blue, the paths are coloured red.}
    \label{fig:SawPaths}
\end{figure}

\begin{lemma}\label{SawWorks}
    Let $G$ be a $k$-regular circular saw on $\Z_n\times\Z_2$ with $n\gg k\ge 3$.
    Then $G$ is $k$-connected and the neighbourhoods of its vertices are precisely its $k$-separators.
\end{lemma}
\begin{proof}
    We show first that $G$ is $k$-connected, and assume for a contradiction that some two vertices $a,b$ of $G$ are separated by a set $S$ of $k-1$ vertices.
    Let $a'$ be a neighbour of $a$ outside $S$, and let $b'$ be a neighbour of $b$ outside~$S$.
    Note that $2(k-1)\ge k+1$ since $k\ge 3$.
    Hence by \cref{SawPaths}, there are $k+1$ disjoint $N(\{a,a'\})$--$N(\{b,b'\})$ paths in~$G$, contradicting that $|S|\le k-1$.
    Hence $G$ is $k$-connected.

    Next, we assume for a contradiction that $G$ has a $k$-separation $(A,B)$ such that $S(A,B)$ fails to be the neighbourhood of some vertex of $G$.
    Let $a\in A\sm B$ and $b\in B\sm A$.
    Then there are vertices $a'\in N(a)\sm S(A,B)$ and $b'\in N(b)\sm S(A,B)$.
    Now \cref{SawPaths} yields $k+1$ disjoint $N(\{a,a'\})$--$N(\{b,b'\})$ paths in $G$, contradicting that $(A,B)$ has order~$k$.
\end{proof}

\section{Appendix B: Comparison of \texorpdfstring{\trisp s}{strict tri-separations} and tri-separations}\label{sec:comparison}

The \trisp -decomposition (\cref{3Decomp}) is different from the tri-separation decomposition \cite[Theorem~1]{Tridecomp} as can be seen in \cref{fig:Trisplit,fig:Triseparation}.
However, the \trisp -decomposition can be viewed as a refinement of the tri-separation decomposition via \cref{trisepGivesTrisplit}.

\begin{figure}[ht]
\centering
\begin{minipage}[b]{.45\textwidth}%
    \centering
    \captionsetup{width=1\textwidth}
    \includegraphics[height=7\baselineskip]{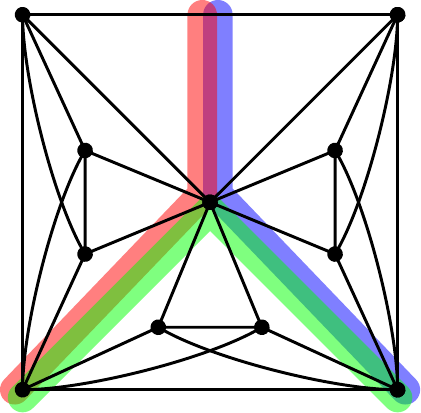}
    \captionof{figure}{The totally-nested \trisp s\\{\,}}
    \label{fig:Trisplit}
\end{minipage}\hfill\begin{minipage}[b]{.45\textwidth}
    \centering
    \captionsetup{width=1\textwidth}
    \includegraphics[height=7\baselineskip]{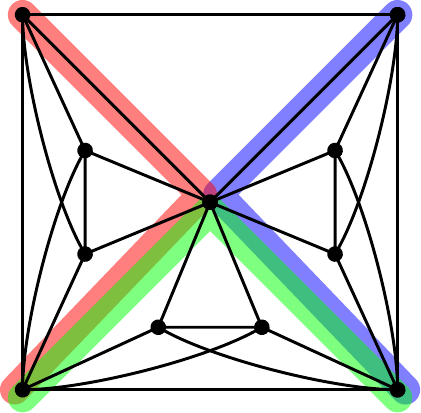}
    \captionof{figure}{The totally-nested non-trivial tri-separations}
    \label{fig:Triseparation}
\end{minipage}
\end{figure}

A tri-separation $(A,B)$ is \defn{negligible} if $|A\sm B|\le 1$ or $|B\sm A|\le 1$.

\begin{lemma}\label{trisepGivesTrisplit}
    Let $G$ be a 3-connected graph.
    Let $(A,B)$ be a totally-nested non-trivial tri-separation of $G$.
    Then the left-right-reduction $(A',B')$ of $(A,B)$ is a totally-nested \trisp , or $(A,B)$ is negligible.
\end{lemma}
\begin{proof}
    Assume that $(A,B)$ is not negligible.
    Then $(A',B')$ is a \trisp\ by \cref{leftrightTetra} (which generalises to \trisp s via an apex-argument as in the proof of \cref{3DecompStrengthening}).
    Assume for a contradiction that there is a \trisp\ $(C,D)$ of $G$ that crosses $(A',B')$ with $C$-link $\gamma'$ and $D$-link $\delta'$.
    We will show that $(C,D)$ crosses $(A,B)$ with non-empty $C$-link $\gamma$ and non-empty $D$-link $\delta$.
    For this, we consider the crossing-diagram of $(C,D)$ and $(A',B')$.
    By \cite[Lemma 1.3.10]{Tridecomp}, there are two cases.

    \casen{Case: the centre consists of one vertex and all links have size one.}
    Then $\gamma'$ and $\delta'$ are non-empty.
    If $\gamma'$ consists of a vertex $v$, then $\gamma$ also consists of~$v$.
    Otherwise $\gamma'$ consists of an edge~$e$.
    The edge $e$ cannot be dangling, since then its endvertex in $S(C,D)$ would violate the degree-condition.
    Hence the endvertices of $e$ are contained in the corners $A'C$ and $B'C$, respectively.
    So $e$ or an endvertex of $e$ is contained in $\gamma$.
    This concludes the proof that $\gamma$ is non-empty, and a similar argumentation shows that $\delta$ is non-empty.
    Hence $(C,D)$ and $(A,B)$ cross with non-empty opposite links $\gamma$ and~$\delta$.

    \casen{Case: the centre consists of three vertices and all links are empty.}
    Then $S(A',B')$ consists of three vertices, so $(A',B')=(A,B)$.
    But all corners are non-empty, so $(A',B')$ is not half-connected.
    However, $(A,B)$ is half-connected because it is totally-nested, by \cite[Lemma 2.3.4]{Tridecomp}, contradicting that $(A',B')=(A,B)$.
\end{proof}

\begin{acknowledgment}
We thank a referee for valuable comments on the introduction.
We thank another referee for bringing~\cite{Isomorphism} to our attention.
We thank Joseph Devine for exciting discussions.
We thank Agelos Georgakopoulos for telling us about a theorem of Godsil and Royle that we use in~\cite{TetraTrans}.
\end{acknowledgment}

\printbibliography
\end{document}